\documentclass{amsart}
\usepackage{hyperref}
\usepackage[ansinew]{inputenc}
\usepackage{graphicx}
\usepackage{amsmath}
\usepackage{amsthm}
\usepackage{amssymb,bbm, color}%
\usepackage[numbers, square]{natbib}
\usepackage{mathrsfs}

\usepackage{pgf, tikz}
\usetikzlibrary{trees}

\newcommand{\blue}{\color{blue}}

\newcommand{\bL}{\mathbb{L}}

\newcommand{\bD}{\mathbb{D}}

\newcommand{\cF}{\mathcal{F}}

\newcommand{\cH}{\mathcal{H}}

\newcommand{\Herm}{\mathrm{He}}

\newcommand{\fI}{\mathscr{I}}
\newcommand{\fD}{\mathscr{D}}

\DeclareMathOperator*{\argmin}{arg\,min}
\DeclareMathOperator*{\argmax}{arg\,max}

\newcommand{\E}{\mathbb E}

\newcommand{\R}{\mathbb{R}}
\newcommand{\N}{\mathbb{N}}

\newcommand{\C}{\mathbb{C}}
\newcommand{\Z}{\mathbb{Z}}

\renewcommand{\P}{\mathbb{P}}

\newcommand{\LLL}{\mathbb{L}}

\renewcommand{\Re}{\operatorname{Re}}
\renewcommand{\Im}{\operatorname{Im}}

\newcommand{\Var}{\mathop{\mathrm{Var}}\nolimits}

\newcommand{\nint}{\mathop{\mathrm{nint}}\nolimits}

\newcommand{\eps}{\varepsilon}
\newcommand{\eqdistr}{\stackrel{d}{=}}

\newcommand{\toas}{\overset{a.s.}{\underset{n\to\infty}\longrightarrow}}
\newcommand{\toasj}{\overset{a.s.}{\underset{j\to\infty}\longrightarrow}}
\newcommand{\toast}{\overset{a.s.}{\underset{t\to\infty}\longrightarrow}}

\newcommand{\ton}{\overset{}{\underset{n\to\infty}\longrightarrow}}

\newcommand{\bsl}{\backslash}
\newcommand{\ind}{\mathbbm{1}}

\newcommand{\dd}{{\rm d}}
\newcommand{\eee}{{\rm e}}

\newcommand{\stirling}[2]{\genfrac{[}{]}{0pt}{}{#1}{#2}}

\theoremstyle{plain}
\newtheorem{theorem}{Theorem}[section]
\newtheorem{lemma}[theorem]{Lemma}

\newtheorem{proposition}[theorem]{Proposition}

\theoremstyle{definition}

\theoremstyle{remark}
\newtheorem{remark}[theorem]{Remark}

\newcounter{step}[theorem]
\newcommand{\step}{\par\addtocounter{step}{1}
\vspace*{2mm}
\noindent
\textsc{Step \arabic{step}.}
}

\begin{document}

\author{Zakhar Kabluchko}
\address{Zakhar Kabluchko, Institut f\"ur Mathematische Statistik,
Universit\"at M\"unster,
Or\-l\'e\-ans--Ring 10,
48149 M\"unster, Germany}
\email{zakhar.kabluchko@uni-muenster.de}

\author{Alexander Marynych}
\address{Alexander Marynych, Faculty of Cybernetics, Taras Shevchenko National University of Kyiv, 01601 Kyiv, Ukraine}
\email{{marynych@unicyb.kiev.ua}}

\author{Henning Sulzbach}
\address{Henning Sulzbach,
School of Computer Science, McGill University,
3480 University Street
H3A 0E9 Montr\'eal, QC, Canada}
\email{{henning.sulzbach@gmail.com}}

\title[Edgeworth expansions for profiles of random trees]{General Edgeworth expansions with applications to profiles of random trees: \\ Extended version}

\thanks{}
\begin{abstract}
We prove an asymptotic Edgeworth expansion for the profiles of certain random trees including binary search trees, random recursive trees and plane-oriented random trees, as the size of the tree goes to infinity. All these models can be seen as special cases of the one-split branching random walk for which we also provide an Edgeworth expansion.  These expansions lead to new results on mode, width and occupation numbers of the trees, settling several open problems raised in Devroye and Hwang [\emph{Ann.\ Appl.\ Probab.}\ \text{16}(2): 886--918, 2006],  Fuchs, Hwang and Neininger [\emph{Algorithmica}, \text{46} (3--4): 367--407, 2006], and Drmota and Hwang [\emph{Adv.\ in Appl.\ Probab.},\ \text{37} (2): 321--341, 2005]. The aforementioned results are special cases and corollaries of a general theorem: an Edgeworth expansion for an arbitrary sequence of random or deterministic functions $\LLL_n:\Z\to\R$ which converges in the mod-$\phi$-sense.
Applications to Stirling numbers of the first kind will be given in a separate paper.
\end{abstract}

\maketitle

\section{Introduction}

\subsection{Introduction}\label{subsec:intro}
The aim of this paper is to study asymptotic properties of profiles for certain families of random trees when the size of the tree goes to infinity.
The profile is the function $k\mapsto \mathbb{L}_n(k)$, where $\mathbb{L}_n(k)$ is the number of nodes at depth $k$ in the tree of size $n$. These numbers are also called ``occupation numbers''.
We shall mainly be interested in the following families of random trees:
\begin{enumerate}
\item binary search trees (BSTs) and, more generally, $D$-ary recursive trees;
\item random recursive trees (RRTs);
\item plane-oriented recursive trees (PORTs) and, more generally, $p$-oriented trees.
\end{enumerate}
BSTs, RRTs and PORTs have been much studied in the literature; see \cite{chauvin_drmota_jabbour,chauvin_etal,devroye_branching,pittel2,fuchs_hwang_neininger,devroye_hwang,drmota_hwang}. Mahmoud's book \cite{mahmoud_book} and Drmota's monograph~\cite{drmota_book} contain further pointers to relevant literature.
It is well known that BSTs are intimately connected to the Quicksort algorithm.

Our main result is an asymptotic expansion for the profile which is somewhat similar to the classical Chebyshev--Edgeworth--Cram\'er expansion for sums of independent identically distributed (i.i.d.)\ integer-valued random variables.  As a consequence of our asymptotic expansion, we derive limit theorems for several functionals of the profile such as the mode, the width, and the occupation numbers, thus answering a number of open questions on these quantities.  Many known results on the profiles such as the (local) central limit theorem or limit theorems for occupation numbers on the scale of large deviations can be recovered in a unified way as corollaries of our expansion.

The scope of our method is by no means restricted to random trees. In Section~\ref{sec:general_expansion} we shall state and prove a very general asymptotic expansion (Theorem~\ref{theo:edgeworth_general}) which can be applied to any sequence of random or deterministic functions $\LLL_n:\Z\to\R$, $n\in\N$, provided certain natural conditions are satisfied. Recently, a closely related expansion was derived by \citet{feray_meliot_nikeghbali} in the framework of the mod-$\phi$-convergence. It has been observed by Nikeghbali and collaborators that mod-$\phi$-convergence, see Remark \ref{rem:modphi} for the definition, is a common phenomenon in probability, combinatorics, number theory and statistical mechanics; see~\cite{delbaen_kowalski_nikeghbali,feray_meliot_nikeghbali, jacod_kowalski_nikeghbali_mod_Gauss,kowalski_nikeghbali_mod_Poi, kowalski_nikeghbali_zeta, meliot_nikeghbali_statmech}.  In this work, we show how mod-$\phi$-convergence\footnote{In all examples mentioned in Section~\ref{subsec:intro} it is, in fact, mod-Poisson convergence. } can be applied to the analysis of random trees of logarithmic height. 

The paper is organized as follows. The general asymptotic expansion is stated in Section~\ref{sec:general_expansion}. In Section~\ref{sec:edgeworth_for_random_trees} we apply this expansion to the profile of the \emph{one-split branching random walk}, a model which contains profiles of all random trees listed above as special cases.  Since these results are quite general and require heavy notation, we motivate and prepare the reader in the next Section~\ref{subsec:BST_special_case} by formulating the results in the special case of binary search trees. In Section~\ref{subsec:random_trees_one_split} we shall explain how to formulate similar results for other random trees including RRTs and PORTs. Proofs are given in Sections~\ref{sec:proof_general_edgeworth} and~\ref{sec:proof_random_trees}.

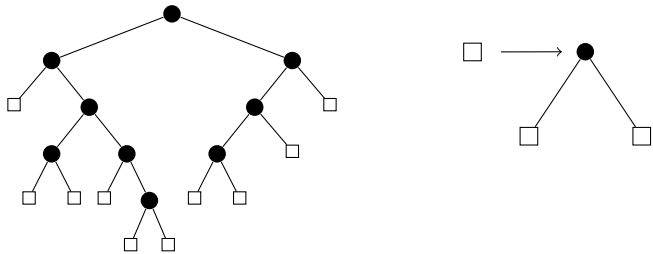
\begin{figure}[htb]
  \centering
  \begin{minipage}[t]{.45\linewidth}
    \centering
\begin{tikzpicture}[level distance=0.5cm,
level 1/.style={sibling distance=3.2cm},
level 2/.style={sibling distance=1cm},
level 3/.style={sibling distance=1cm},
level 4/.style={sibling distance= .6cm},
level 5/.style={sibling distance= .5cm},
level 6/.style={sibling distance= 1cm},
int/.style = {circle, scale = 0.7, draw=none, fill=black, anchor = north, growth parent anchor=south},
ext/.style = {rectangle, scale = 0.7, anchor = north, growth parent anchor=south}]
\tikzstyle{every node}=[circle,draw]

\node (Root) [int] {}
child { node [int] {}
 child { node [ext] {}
  }
 child { node [int] {}
  child { node [int] {}
   child { node [ext] {}
   }
   child { node [ext] {}
   }
  }
  child { node [int] {}
   child { node [ext] {}
   }
   child { node [int] {}
    child { node [ext] {}
    }
    child { node [ext] {}
    }
   }
  }
 }
}
child { node [int]{}
 child { node [int] {}
  child { node [int] {}
   child { node [ext] {}
   }
   child { node [ext] {}
   }
  }
  child { node [ext] {}
  }
 }
 child { node [ext] {}
 }
}
;
\end{tikzpicture}
\end{minipage} \begin{minipage}[t]{.35\linewidth}
    \centering
\raisebox{4em}{\begin{tikzpicture}[level distance=1cm,
level 1/.style={sibling distance=1.5cm},
level 2/.style={sibling distance=1cm},
int/.style = {circle, scale = 0.7, draw=none, fill=black, anchor = north, growth parent anchor=south},
ext/.style = {rectangle, scale = 1.0, draw, anchor = north, growth parent anchor=south}]
\draw (-1.5,0) node [ext] (A) {} ;
 \node(Root) [int] {}
child { node [ext] {}
}
child { node [ext] {}
}
;
\draw [shorten >=1cm,shorten <=1cm,>-] (A) -- (Root);
\end{tikzpicture} }
\end{minipage}
\caption{Left: Sample realization of the BST.  Right: Construction rule for the BST} \label{bild:BST}
\end{figure}


\subsection{Results for binary search trees}\label{subsec:BST_special_case}
For our purposes, the following construction of BSTs is most convenient.
There are  nodes of two types: the external ones (denoted by $\Box$) and the internal ones (denoted by $\bullet$); see Figure~\ref{bild:BST} (left). At time $n=0$ start with one external node (the root of the tree). At any step of the construction, pick one of the existing external nodes uniformly at random, declare it to be internal, and connect it to two new external nodes according to the rule shown on Figure~\ref{bild:BST} (right).
After $n$ steps, one obtains a random tree $T_n$ having $n$ internal and $n+1$ external nodes; see Figure~\ref{bild:BST} (left) for a sample realization.

The \textit{depth} of an external node is its distance to the root. Denote by $\bL_n(k)$ the number of external nodes of $T_n$ at depth $k\in \N_0:=\N\cup\{0\}$, and let $\LLL_n(k)=0$ for $k<0$.  The random function $k\mapsto \bL_n(k)$, $k\in\Z$,  is called the (external) \textit{profile} of the tree $T_n$.
Denote by $x_{1,n}, \ldots, x_{n+1,n}$ the depths of the (arbitrarily enumerated) external nodes of $T_n$ so that
$$
\LLL_n(k) = \#\{1\leq i\leq n+1\colon x_{i,n} = k\}, \quad k\in\Z.
$$

The BST profile has been much studied; see  \cite{chauvin_drmota_jabbour, chauvin_etal, devroye_hwang,drmota_hwang,drmota_janson_neininger,fuchs_hwang_neininger}, 
and \cite[Section~6.5]{drmota_book} 
for a survey.
In the following, we provide a short overview of known results. Let $\beta_-\approx -1.678$ and $\beta_+\approx 0.768$ be the solutions to the equation $2\eee^{\beta} (1-\beta)= 1$. The numbers $2\eee^{\beta_-}\approx 0.373$ and $2\eee^{\beta_+}\approx 4.311$ are called the \textit{fill-up level constant} and the \textit{height constant} of the BST because of the classical results 
$$
\frac 1 {\log n}\, {\min_{i=1,\ldots,n+1} x_{i,n}} \toas 2\eee^{\beta_-},
\quad
\frac 1 {\log n}\,  {\max_{i=1,\ldots, n+1} x_{i,n}}  \toas 2\eee^{\beta_+},
$$
going back to Devroye~\cite{devroyeheight}, see also~\cite{biggins_grey}.
Define the following random function
$$
W_n(\beta) = \frac 1 {n^{(2\eee^\beta-1)}} \sum_{i=1}^{n+1} \eee^{\beta x_{i,n}}, \quad \beta\in\C,
$$
which is the normalized moment generating function of the random counting measure $\sum_{k\in\Z}\LLL_n(k)\delta_{k}=\sum_{i=1}^{n+1}\delta_{x_{i,n}}$ where $\delta_x$ denotes a Dirac measure at point $x$. The basic fact underlying all further arguments is that $W_n$ converges as $n\to\infty$ to a random analytic function $W_\infty$ with probability $1$. More precisely, it is known~\cite{chauvin_etal} that there is an open set $\fD\subset \C$ containing the interval $(\beta_-,\beta_+)$ and a random analytic function $W_\infty$ defined on $\fD$ such that
$$
\sup_{\beta\in K} |W_n(\beta) -  W_\infty(\beta)| \toas 0
$$
for every compact set $K\subset \fD$. Note that $W_\infty(0)=1$ since $W_n(0)=(n+1)/n$ for all $n\in\N$.

It is useful to keep in mind the following principle: $k\mapsto \frac 1{n} \LLL_n(k)$ is ``close'' to the probability mass function of the Poisson distribution  with intensity $2\log n$. The moment generating function of the latter distribution is $\beta\mapsto n^{2\eee^{\beta}-2}$, and the general philosophy of mod-$\phi$-convergence~\cite{jacod_kowalski_nikeghbali_mod_Gauss,kowalski_nikeghbali_mod_Poi} suggests to view the limit function $W_\infty$ as a quantification of the ``difference'' between $\frac 1n \LLL_n$ and the Poisson distribution with parameter $2\log n$. Note that in our case, this function is random.

An important role will be played by the derivatives and logarithmic derivatives of $W_\infty$ (the latter are called random cumulants).   In particular, we shall frequently encounter the random variables
$$
\chi_1(0) := (\log W_\infty)' (0) = W_{\infty}'(0)
,\quad
\chi_{2}(0) := (\log W_\infty)'' (0) =
W_\infty''(0) - (W'_\infty(0))^2.
$$
It will be shown in Section~\ref{subsec:cumulants_one_split} that they can also be represented as the a.s.\ limits
\begin{align}
\chi_1(0) &= \lim_{n\to\infty}(\log W_n)'(0)=\lim_{n\to\infty} \left( \frac {1}{n+1} \sum_{i=1}^{n+1} x_{i,n} - 2\log n\right),\label{eq:chi_1_0_BST}\\
\chi_2(0) &= \lim_{n\to\infty}(\log W_n)''(0)\label{eq:chi_2_0_BST}\\
&=\lim_{n\to\infty} \left( \frac {1}{n+1} \sum_{i=1}^{n+1} x_{i,n}^2 - \left(\frac 1 {n+1} \sum_{i=1}^{n+1} x_{i,n}\right)^2 - 2\log n\right). \notag
\end{align}
It is useful to think of $\chi_1(0)$  (whose distribution is the celebrated Quicksort law~\cite{regnier,roesler}, modulo an additive constant) as a parameter describing the random shift of the BST profile with respect to (w.r.t.)\ the location $2\log n$.
The random variable $\chi_2(0)$ describes the random deviation of the empirical variance of the profile from the value $2\log n$ and seems to appear for the first time.

We now proceed to the asymptotic properties of the BST profile $(\LLL_n(k))_{k\in\Z}$, as $n\to\infty$.  The following (local) central limit theorem  was proved by \citet{chauvin_drmota_jabbour}:
\begin{equation}\label{eq:CLT_for_BST}
\sup_{k\in\Z} \left|\frac 1n \LLL_n(k) - \frac {1} {\sqrt{4\pi \log n}} \eee^{-\frac {(k-2\log n)^2}{4\log n}}\right|  = O\left(\frac 1 {\log n}\right)
\quad
\text{a.s.}
\end{equation}
As a special case of our results, we shall derive an asymptotic expansion complementing~\eqref{eq:CLT_for_BST}.
\begin{theorem}\label{theo:BST_CLT_expansion}
Let $(\LLL_n(k))_{k\in\Z}$ be the external profile of a binary search tree with $n+1$ external nodes. For every $r\in\N_0$, we have
$$
(\log n)^{\frac {r+1}{2}} \sup_{k\in\Z} \left| \frac 1n \LLL_n(k) - \frac {\eee^{-\frac {(k-2\log n)^2}{4\log n}}} {\sqrt{4\pi \log n}} \sum_{j=0}^r G_j\left(\frac {k-2\log n}{\sqrt {2\log n}};0\right) \frac 1 {(\log n)^{j/2}} \right| \toas 0,
$$
where $G_j(x;0)$ is a polynomial in $x$ of degree at most $3j$ whose coefficients can be  linearly expressed through the derivatives $W_{\infty}'(0),\ldots, W^{(j)}_\infty(0)$. For example,
$$
G_0(x;0)=1, \quad G_1(x;0)= \frac 1{\sqrt 2} \left( W_\infty'(0) x + \frac {x^3-3x}6\right);
$$
see~\eqref{eq:G_def_one_split}, \eqref{eq:D_j_one_split} below for an explicit general formula.
\end{theorem}
The above results deal with the profile around level $2\log n$, where ``most'' external nodes are located. The shape of the profile at levels around $c\log n$, $2\eee^{\beta_-} < c < 2\eee^{\beta_+}$, is described by the following result due to Chauvin et al. ~\cite{chauvin_etal}, compare also~\cite{chauvin_drmota_jabbour,drmota_janson_neininger,jabbour} for weaker formulations and~\cite{fuchs_hwang_neininger} for pointwise convergence theorems: 
\begin{equation}\label{eq:LD_BST}
\sup_{k\in \Z \cap (\log n) L} \left|\frac {n k!} {(2\log n)^k} \LLL_n(k)  -  W_\infty\left(\log\left(\frac{k}{2\log n}\right)\right)\right| \toas  0
\end{equation}
for every compact set $L\subset (2\eee^{\beta_-}, 2\eee^{\beta_+})$. We can derive an asymptotic expansion complementing~\eqref{eq:LD_BST}.
\begin{theorem}\label{theo:BST_LD_expansion}
Let $(\LLL_n(k))_{k\in\Z}$ be the external profile of a binary search tree with $n+1$ external nodes. For every $r\in\N_0$ and every compact set $L\subset (2\eee^{\beta_-}, 2\eee^{\beta_+})$ we have
$$
(\log n)^{r+1}
\sup_{k\in \Z \cap (\log n) L} \left| n \left(\frac{k}{2\eee \log n}\right)^k  \LLL_n(k)  - \frac{1}{ \sqrt{2\pi k}} \sum_{j=0}^{r} \frac{F_{2j}(0; \beta_n(k))}{(\log n)^{j}}\right| \toas  0,
$$
where $\beta_n(k)$ is the solution\footnote{Obviously, we have $\beta_n(k)=\log k- \log\log n-\log 2$, however we prefer to define $\beta_n(k)$ implicitly to conform with the definition of $\beta_n(k)$ in general case; see formula \eqref{def:beta_n_k} below.} of $2\eee^{\beta_n(k)} = k/\log n$ and $F_{2j}(0; \beta) := W_\infty(\beta) G_{2j}(0;\beta)$ is a linear combination of $ 1,W_\infty(\beta),\ldots, W_\infty^{(2j)}(\beta)$; see~\eqref{eq:G_def_one_split}, \eqref{eq:D_j_one_split} below for an explicit formula.  For example,
$$
F_0(0; \beta) = W_\infty(\beta),
\quad
F_2(0;\beta) = \frac 14 (W'_\infty(\beta) - W''_\infty(\beta)) - \frac 1 {24}.
$$
\end{theorem}

\begin{remark}
Theorems~\ref{theo:BST_CLT_expansion} and~\ref{theo:BST_LD_expansion} are special cases (with $\beta=0$ and $\beta=\beta_n(k)$, respectively) of Theorem~\ref{theo:edgeworth_one_split_BRW} (see also Theorem~\ref{theo:LD_one_split}) which deals with general one-split branching random walks and which we shall state and prove below. Similar results can easily be obtained for RRTs and PORTs; see Section~\ref{subsec:random_trees_one_split} for details.
\end{remark}

\begin{remark}
Our techniques yield analogous expansions for the mean profile: Theorems~\ref{theo:BST_CLT_expansion} and~\ref{theo:BST_LD_expansion} remain valid upon replacing $\bL_n(k)$ by $\E [\bL_n(k)]$ and the random polynomials $G_j, F_{2j}$ by their expectations.
\end{remark}

The above expansions can be used to answer a number of open questions on the BST profile. The \emph{mode} $u_n$ and the \emph{width} $M_n$ of a binary search tree are defined by
\begin{equation*}
u_n =\argmax_{k\in\N_0} \bL_n(k),
\quad
M_n = \max_{k\in\N_0} \bL_n(k).
\end{equation*}
These quantities were studied by \citet{chauvin_drmota_jabbour}, \citet{drmota_hwang} and \citet{devroye_hwang}. In particular, \citet[Theorem~4.1]{devroye_hwang} proved that $u_n$ is concentrated near $2\log n$ in the sense that for every $B>0$ there is  $C_0 = C_0(B)$  such that
$$
\P[|u_n - 2\log n| > T] \leq C_0 T^{-B},\quad n \in \N, \quad T\geq 1.
$$
We show that, starting from some random almost surely finite time $K$, the mode $u_n$ attains only one of two possible explicitly given values. This result is a special case of Theorem~\ref{theo:mode_one_split}.

\begin{theorem}\label{theo:BST_mode}
There is an a.s.\ finite random variable $K$ such that for $n>K$, the mode $u_n$ of the BST with $n+1$ external nodes is equal to one of the numbers
\begin{equation*}
\lfloor 2\log n + \chi_1(0) - 1/2\rfloor
\text{ or }
\lceil 2\log n + \chi_1(0) - 1/2\rceil,
\end{equation*}
where  $\lfloor\cdot\rfloor$, $\lceil\cdot\rceil$ denote the floor and the ceiling functions, respectively,  and $\chi_1(0)$ is the Quicksort-distributed random variable defined in~\eqref{eq:chi_1_0_BST}.
\end{theorem}

\begin{remark}
Using Proposition~\ref{cor:one_split_mode} we can say more about the behavior of the mode. More precisely, the following statements hold with probability $1$:

\vspace*{1mm}
\noindent
$\bullet$
there are arbitrarily long intervals of consecutive $n$'s for which the mode $u_n$ is unique and $u_n=\lceil 2\log n + \chi_1(0) - 1/2\rceil$;

\vspace*{1mm}
\noindent
$\bullet$
similarly, there are arbitrarily long intervals of consecutive $n$'s for which $u_n$ is unique and $u_n=\lfloor 2\log n + \chi_1(0) - 1/2\rfloor$;

\vspace*{1mm}
\noindent
$\bullet$
the set of $n\in\N$ such that $u_n$ is the integer closest to $2\log n + \chi_1(0) - 1/2$ has asymptotic density one (see~\eqref{def:asymp_freq} for the definition of the asymptotic density).
\end{remark}

%

As for the BST's width $M_n$, it is known~\cite{chauvin_drmota_jabbour,devroye_hwang,drmota_hwang} that
$$
\E [M_n] = \frac {n}{\sqrt{4\pi \log n}} \left(1+O\left(\frac 1 {\log n}\right)\right),
\quad
\frac {M_n\sqrt{4\pi\log n}}{n} \toas 1.
$$
Both \citet{devroye_hwang} as well as \citet[Section 5]{drmota_hwang} asked for the limit distribution of $M_n$ (if it exists).
The next result (which holds a.s.\ rather than in distribution) settles this issue, and  is a consequence of Theorem~\ref{theo:width_limit_law_one_split_BRW} and the remark following it.
\begin{theorem}\label{theo:BST_width}
Let $M_n$ be the width of a binary search tree with $n+1$ external nodes. With probability $1$, the set of subsequential limits of the sequence
$$
\tilde M_n := 4\log n \left(1- \frac{\sqrt{4\pi \log n}\, M_n}{n}\right),
\quad n\in\N,
$$
is the interval
$
[ \chi_2(0)-1/ {12},
\chi_2(0)+1/6]
$
with $\chi_2(0)$ as in~\eqref{eq:chi_2_0_BST}.
Furthermore, with $\theta_n= \min_{k\in\Z} |2\log n + \chi_1(0) - 1/2-k|$ we have
\begin{equation}\label{eq:theo:BST_width}
\tilde M_n - \theta_n^2 \toas \chi_2(0) - \frac 1 {12}.
\end{equation}
\end{theorem}
\begin{remark}
Let us stress that the centering $\theta_n^2$ in~\eqref{eq:theo:BST_width} is random since it involves $\chi_1(0)$. In order to obtain a non-random centering, we have to pass to a subsequence. If $(n_{j})_{j\in\N}\subset \N$ is any sequence with  $n_j \to +\infty$ and  $\{2\log n_j\}\to \alpha\in [0,1]$ as $j\to\infty$ (where $\{\cdot\}$ denotes the fractional part), then $\lim_{j\to\infty} \theta_{n_j} = |\{\alpha + \chi_1(0)\} - 1/2|$ and we obtain
$$
\tilde M_{n_j} \toasj \chi_2(0) - \frac 1 {12} + \left(\{\alpha + \chi_1(0)\} - \frac 12\right)^2.
$$
Since the set of accumulation points of the sequence $(\{2\log n\})_{n\in \N}$ is the interval $[0,1]$,
we obtain for $\tilde M_n$ a family of subsequential limit distributions indexed by $\alpha\in [0,1]$ with values $\alpha=0$ and $\alpha=1$ corresponding to the same limit.
\end{remark}

In the next theorem we describe the asymptotic behavior of the ``occupation numbers'' $\LLL_n(k_n)$, where $(k_n)_{n\in\N}$ is a deterministic sequence with sufficiently regular behavior. These quantities were the main object of study in \citet{fuchs_hwang_neininger}; see also~\cite{chauvin_drmota_jabbour,chauvin_etal} and (for results on lattice branching random walks) \cite{chen,gruebel_kabluchko_BRW,kabluchko_distr_of_levels}.
It is known~\cite{chauvin_etal,fuchs_hwang_neininger} (and not difficult to deduce from~\eqref{eq:LD_BST}) that if $k_n = 2\eee^\beta \log n + \alpha \sqrt {2\eee^\beta \log n} + o(\sqrt{\log n})$ for some $\beta\in (\beta_-,\beta_+)$ and  $\alpha\in\R$, then
\begin{equation}\label{eq:BST_near_beta_log_n1}
\frac{\sqrt {2\eee^\beta \log n}}{n^{2\eee^\beta-1}} \eee^{\beta k_n} \LLL_n(k_n) \toas \frac {W_\infty(\beta)} {\sqrt {2\pi}} \eee^{-\frac 12 \alpha^2}.
\end{equation}
Furthermore, the convergence of moments of order $\kappa$, for any $\kappa \in (1,2) \cup \{2,3,\ldots\}$  with $\E{[W_\infty^\kappa(\beta)]} < \infty$, was proved in~\cite{fuchs_hwang_neininger}.
Another consequence of~\eqref{eq:LD_BST} is that for $k_n = 2\eee^{\beta}\log n + c_n$, where $\beta\in (\beta_-,\beta_+)$, $c_n = o(\log n)$ we have
\begin{equation}\label{eq:BST_near_beta_log_n2}
\frac{\sqrt {2\eee^\beta \log n}}{\eee^{c_n} n^{2\eee^\beta-1}}  \left(\frac{k_n}{2\log n}\right)^{k_n} \LLL_n(k_n) \toas \frac{W_\infty(\beta)}{\sqrt {2\pi}}.
\end{equation}

For $\beta=0$, the limit random variable in~\eqref{eq:BST_near_beta_log_n1}, \eqref{eq:BST_near_beta_log_n2} is degenerate because $W_\infty(0)=1$, and a more refined analysis is needed to obtain a non-degenerate limit law. It turns out that all such results hold also in the almost sure sense.

\begin{theorem}\label{theo:BST_occupation_numbers}
Let $(\LLL_n(k))_{k\in\Z}$ be the  external profile of a  BST with $n+1$ external nodes. Put $\LLL_n^{\circ}(k):=\LLL_n(k)- \E [\LLL_n(k)]$ and let  $(k_n)_{n\in\N}$ be a deterministic  integer sequence.

\vspace*{2mm}
\noindent
\emph{(a)}
If $k_n = 2\log n + \alpha \sqrt {2\log n} + o(\sqrt{\log n})$ for some $\alpha\in\R$, then
$$
\frac{\log n}{n} \LLL_n^{\circ}(k_n) \toas \frac {\chi_1(0) - \E [\chi_1(0)]}{2\sqrt{2\pi}} \alpha \eee^{-\frac 12 \alpha^2}.
$$

\vspace*{2mm}
\noindent
\emph{(b)}
If $k_n = 2\log n + c_n$, where $c_n = o(\log n)$ and $\lim_{n\to\infty} |c_n| =\infty$, then
$$
\frac {(\log n)^{3/2}} {nc_n\eee^{c_n}}  \left(\frac{k_n}{2\log n}\right)^{k_n} \LLL_n^{\circ}(k_n) \toas \frac {\chi_1(0) - \E [\chi_1(0)]}{ 4\sqrt{\pi}}.
$$
In particular, if $c_n = o(\sqrt{\log n})$ and  $\lim_{n\to\infty} |c_n| = \infty$, then
$$
\frac{(\log n)^{3/2}}{nc_n} \LLL_n^{\circ}(k_n) \toas \frac {\chi_1(0) - \E [\chi_1(0)]}{4 \sqrt \pi}.
$$

\vspace*{2mm}
\noindent
\emph{(c)}
If $k_n = 2\log n + c_n$, where $c_n$ is bounded, then
\begin{multline*}
\frac{(\log n)^{3/2}}{n} \LLL_n^{\circ}(k_n)
- \frac {\chi_1(0) - \E [\chi_1(0)]} {4\sqrt{\pi}} \left(c_n + \frac{1}{2}\right)
\toas -\frac {W_\infty''(0) - \E [W_\infty''(0)]} {8\sqrt{\pi}}.
\end{multline*}
\end{theorem}
\begin{remark}
More specifically, if in case (c) we have $k_n = \lfloor2\log n\rfloor + a$ for $a\in\Z$, then the set of subsequential limits of the sequence $\left(\frac 1n (\log n)^{3/2} \LLL_n^{\circ}(k_n)\right)_{n\in\N}$ equals, with probability $1$, the closed interval
$$
\left\{
\frac {\chi_1(0) - \E [\chi_1(0)]} {4\sqrt{\pi}} \left(a +y \right)
-
\frac {W_\infty''(0) - \E[W_\infty''(0)]} {8\sqrt{\pi}}
\colon -\frac 12\leq y\leq \frac 12
\right\}.
$$
A subsequential limit of this form is attained along any subsequence $(n_j)_{n\in\N}$ with $\{2\log n_j\} \to \frac 12 - y$ as $j\to\infty$.
\end{remark}

\begin{remark}
Theorem~\ref{theo:BST_occupation_numbers} and Equations~\eqref{eq:BST_near_beta_log_n1}, \eqref{eq:BST_near_beta_log_n2} are special cases of more general Theorems~\ref{theo:L_n_lim_distr_beta_neq_0}, \ref{theo:L_n_lim_distr_all_together} which deal with one-split branching random walks.
As will be explained in Section~\ref{subsec:random_trees_one_split}, analogous results can be obtained for RRTs and PORTs simply by inserting suitable parameters into the theorems listed.
\end{remark}

In cases (a) and (b), distributional convergence (and, in fact, convergence of all moments) was proved by \citet{fuchs_hwang_neininger}. Our approach (which is very different from the method of moments and the contraction method  used in~\cite{fuchs_hwang_neininger}) yields a.s.\ convergence.
In case (c), \citet{fuchs_hwang_neininger} showed that $\LLL_n(k_n)$, centered by its expectation and normalized by its standard deviation, has no non-degenerate limit law. Our result identifies all possible weak (and, in fact, even a.s.)\ subsequential limits of the appropriately normalized $\LLL_n(k_n)$.
One may also ask for multivariate limit laws for the occupation numbers. For example, in case (c) it is natural to investigate the joint limit distribution of the random vector $(\bL_n(\lfloor 2\log n \rfloor+a))_{a=-K,\ldots,K}$ where $K\in\N_0$ is fixed. Since our results are a.s.,\ they automatically yield such multivariate limit theorems, whereas the moment method and the contraction method seem less convenient to treat such multivariate problems.
Finally, let us mention that there is one more case in which $W_\infty(\beta)$ is a.s.\ constant, namely $\beta=-\log 2$; see Section~\ref{subsec:profile_BST_log_n} for a detailed analysis of this case.

\section{The general Edgeworth expansion} \label{sec:general_expansion}
\subsection{Assumptions on the profiles}
Consider a sequence $\bL_1,\bL_2,\ldots$ such that each
$
\bL_n= (\bL_n(k))_{k\in\Z}
$
is a real-valued stochastic process defined on the integer lattice $\Z$. We assume that all $\bL_1,\bL_2,\ldots$  are defined on a common probability space $(\Omega,\cF,\P)$. We shall consider the random function
$$
\bL_n:\Z\to\R,\;\; k\mapsto \bL_n(k)
$$
as a ``\textit{random profile}''. As has already been mentioned, in our applications to random trees, $\bL_n(k)$ is the number of nodes of depth $k$ in a random tree  at time $n$.
Our aim is to prove that under appropriate assumptions, $\bL_n$ satisfies an Edgeworth-type asymptotic expansion with probability $1$. Let us state these assumptions. 

\vspace*{2mm}
\noindent
\textbf{Assumption A1:} There is an open, non-empty interval $(\beta_-,\beta_+)\subset \R$ containing $0$ such that for every $n\in\N$ and every $\beta \in (\beta_-,\beta_+)$,
\begin{equation}
\sum_{k\in\Z} |\bL_n(k)| \eee^{\beta k} < \infty \quad \text{a.s.}
\end{equation}

The interval $(\beta_-,\beta_+)$ need not be bounded. For example, Assumption~A1 is satisfied  on whole $\R$ if for every $n\in\N$ the profile support $\{k\in\Z\colon \bL_n(k)\neq 0\}$ is a finite set with probability $1$.

\vspace*{2mm}
The next assumption essentially states that the Laplace transform of the profile given by
$$
\beta \mapsto \sum_{k\in\Z} \bL_n(k) \eee^{\beta k}
$$
converges, after an appropriate normalization, to a random analytic function on some domain $\fD$ in the complex plane which contains the interval $(\beta_-,\beta_+)$.
To state this assumption we need the following ingredients:
\begin{itemize}
\item a sequence $(w_n)_{n\in\N}\subset \R$ such that $\lim_{n\to\infty} w_n = + \infty$;
\item an open, connected set $\fD\subset \{\beta\in\C\colon \beta_- < \Re \beta <\beta_+\}$ such that $\fD \cap \R = (\beta_-,\beta_+)$;
\item a (deterministic) analytic function $\varphi: \fD\to \C$ such that, for real $\beta\in (\beta_-,\beta_+)$, we have $\varphi(\beta)\in\R$ and $\varphi''(\beta)>0$.
\end{itemize}
It follows from Assumption~A1 that, with probability $1$,  the normalized Laplace transform
\begin{equation}\label{eq:biggins_def}
W_n(\beta) := \eee^{-\varphi(\beta) w_n} \sum_{k\in\Z} \bL_n(k)\eee^{\beta k}, \quad \beta \in \fD,
\end{equation}
is a random analytic function on $\fD$.

\vspace*{2mm}
\noindent
\textbf{Assumption A2:}
With probability $1$, the sequence of random analytic functions $(W_n)_{n\in\N}$ converges locally uniformly on $\fD$, as $n\to\infty$, to some random analytic function $W_\infty$ such that $\P[W(\beta)\neq 0\text{ for all }\beta\in(\beta_{-},\beta_{+})]=1$.
\vspace*{2mm}
\noindent

Moreover, we require that the speed of convergence is superpolynomial in $w_n$.

\vspace*{2mm}
\noindent
\textbf{Assumption A3:}
For every compact set $K\subset \fD$ and $r\in\N$ we can find an a.s.\ finite random variable $C_{K,r}$ such that for all $n\in\N$,
\begin{equation}\label{eq:Psi_n_to_W_infty_speed}
\sup_{\beta\in K} |W_n(\beta) - W_\infty(\beta)| < C_{K, r} w_n^{-r}.
\end{equation}


The last assumption is of technical nature. In the classical Edgeworth expansion for sums of i.i.d.\ integer-valued  variables, it corresponds to the assumption that $\Z$ is the minimal lattice on which the distribution is concentrated.

\vspace*{2mm}
\noindent
\textbf{Assumption A4:} For every compact set $K\subset (\beta_-,\beta_+)$, every $a>0$ and $r\in\N_0$, we have
\begin{equation} \label{assump_4}
\sup_{\beta\in K} \left[\eee^{-\varphi(\beta) w_n} \int_a^{\pi} \left|\sum_{k\in\Z} \bL_n(k) \eee^{k(\beta + iu)} \right| \dd u\right] = o(w_n^{-r}) \quad \text{a.s.\ as } n\to\infty.
\end{equation}

\subsection{Statement of the general Edgeworth expansion}
Consider a sequence of profiles $\bL_1,\bL_2,\ldots$ satisfying Assumptions A1--A4. We are going to state an Edgeworth expansion for $\bL_n$ as $n\to\infty$. In fact, we shall obtain an expansion of the ``tilted'' profile $k\mapsto \eee^{\beta k - \varphi(\beta) w_n} \bL_n(k)$ which is uniform as long as $\beta$ stays in a certain range.

We shall see that the following parameters $\mu(\beta)$ and $\sigma(\beta)$ play the role of the ``drift'' and the ``standard deviation'' of the tilted profile:
\begin{equation}\label{eq:mu_sigma_def}
\mu(\beta)=\varphi'(\beta),
\;\;\;
\sigma^2(\beta)=\varphi''(\beta).
\end{equation}
Introduce the normalized coordinate
\begin{equation}\label{eq:x_n_k_def}
x_n(k) = x_n(k;\beta) = \frac{k-\mu(\beta) w_n}{ \sigma(\beta) \sqrt{w_n}}, \quad k\in\Z.
\end{equation}
Define the ``deterministic cumulants'' $\kappa_j(\beta)$ and the ``random cumulants'' $\chi_{j}(\beta)$ by
\begin{equation}\label{eq:chi_infty_def}
\kappa_j(\beta) = \varphi^{(j)}(\beta),
\quad
\chi_{j}(\beta) = (\log W_\infty)^{(j)} (\beta).
\end{equation}
The general Edgeworth expansion for the profile $\bL_n$ reads as follows.
\begin{theorem}\label{theo:edgeworth_general}
Let $\bL_1,\bL_2,\ldots$ be a sequence of random profiles satisfying Assumptions A1--A4. Fix $r\in\N_0$ and a compact set $K\subset (\beta_-,\beta_+)$. Then,
\begin{equation}\label{eq:edgeworth_general_exp}
w_n^{\frac{r+1}{2}} \sup_{k\in\Z} \sup_{\beta\in K} \left|\eee^{\beta k - \varphi(\beta) w_n} \bL_n(k) - \frac{W_\infty(\beta) \eee^{-\frac 12 x_n^2(k)}}{\sigma(\beta) \sqrt{2\pi w_n}}  \sum_{j=0}^r \frac{G_j(x_n(k);\beta)}{w_n^{j/2}}   \right| \toas 0.
\end{equation}
Here, $G_j(x)=G_j(x;\beta)$, $j\in\N_0$, is a polynomial of degree at most $3j$ given by
\begin{equation}\label{eq:G_def}
G_j(x) = \frac {(-1)^j} {j!} \, \eee^{\frac 12 x^2} B_j(D_1,\ldots,D_j) \eee^{-\frac 12 x^2},
\end{equation}
where $B_j$ is the $j$-th Bell polynomial (defined in Remark~\ref{rem:bell_poly_def} below) and $D_1,D_2,\ldots$ are linear differential operators (with random coefficients) given by
\begin{equation}\label{eq:D_def}
D_j
=
\frac{\varphi^{(j+2)}(\beta)}{(j+1)(j+2)} \left(\frac{1}{\sigma(\beta)}\frac{\dd}{\dd x}\right)^{j+2}
+ \chi_j(\beta) \left(\frac{1}{\sigma(\beta)} \frac{\dd}{\dd x}\right)^{j}.
\end{equation}
\end{theorem}
\begin{remark}\label{rem:bell_poly_def}
The (complete) \textit{Bell polynomials} $B_j(z_1,\ldots,z_j)$ are defined by the formal identity
\begin{equation}\label{eq:bell_poly_def1}
\exp\left\{ \sum_{j=1}^{\infty} \frac{x^j}{j!} z_j \right\}
=
\sum_{j=0}^{\infty} \frac{x^j}{j!} B_j(z_1,\ldots,z_j).
\end{equation}
It follows that $B_0=1$ and for $j\in\N$,
\begin{equation}\label{eq:bell_poly_def}
B_j(z_1,\ldots,z_j) = \sum{}^{'}\frac{j!}{i_1!\ldots i_j!} \left(\frac{z_1}{1!}\right)^{i_1}\ldots  \left(\frac{z_j}{j!}\right)^{i_j},
\end{equation}
where the sum $\sum{}^{'}$ is taken over all $i_1,\ldots,i_j\in\N_0$ satisfying $1i_1+2i_2+\ldots +j i_j =j$.
We shall need the first three Bell polynomials which are given by
\begin{equation}\label{eq:Bell_poly_first}
B_0=1,\quad
B_1(z_1)=z_1, \quad
B_2(z_1,z_2) = z_1^2 + z_2.
\end{equation}
\end{remark}
\begin{remark}
It follows from~\eqref{eq:G_def}, \eqref{eq:D_def}, \eqref{eq:Bell_poly_first} that $G_0,G_1,G_2$ are given by
\begin{align}
G_0(x)
&=
1,\label{eq:G0}\\
G_1(x)
&=
\frac{\chi_1(\beta)}{\sigma(\beta)} x + \frac{\kappa_3(\beta)}{6\sigma^3(\beta)} \Herm_3(x),\label{eq:G1}\\
G_2(x)
&=
\frac{\chi_1^2(\beta) + \chi_2(\beta)}{2\sigma^{2}(\beta)}\Herm_2(x)
+
\left(\frac{\kappa_4(\beta)}{24\sigma^4(\beta)} + \frac{\kappa_3(\beta)\chi_1(\beta)}{6\sigma^4(\beta)}\right) \Herm_4(x)\label{eq:G2}\\
&+
\frac{\kappa_3^2(\beta)}{72 \sigma^6(\beta)}\Herm_6(x), \notag
\end{align}
where  $\Herm_n(x)$ denotes the $n$-th ``probabilist'' \textit{Hermite polynomial}:
$$
\Herm_n(x)= \eee^{\frac 12 x^2} \left(-\frac{\dd}{\dd x}\right)^n \eee^{-\frac 12 x^2}.
$$
The first few Hermite polynomials relevant to us are
\begin{align}
&\Herm_1(x)= x,\quad
\Herm_2(x)= x^2-1, \quad
\Herm_3(x) = x^3-3x,\label{eq:Herm1}\\
&\Herm_4(x)= x^4-6x^2+3, \quad
\Herm_6(x)= x^6 - 15 x^4 + 45 x^2 - 15. \label{eq:Herm2}
\end{align}
\end{remark}

\begin{remark}\label{rem:G_odd_even}
We have $G_j(-x)= (-1)^j G_j(x)$ for all $j\in\N_0$. In particular, $G_j(0) = 0$ for odd $j$.
Indeed, by~\eqref{eq:D_def}, $D_k$ is a linear combination of the differential operators $(\dd/\dd x)^k$ and $(\dd/\dd x)^{k+2}$.  It follows from~\eqref{eq:bell_poly_def} that $B_j(D_1,\ldots,D_j)$ is a linear combination of the differential operators of the form
$$
\left(\frac{\dd}{\dd x}\right)^{m_1 i_1} \ldots \left(\frac{\dd}{\dd x}\right)^{m_j i_j} = \left(\frac{\dd}{\dd x}\right)^{m_1 i_1 + \ldots + m_j i_j},
$$
where each $m_k$ is either $k$ or $k+2$, so that $m_1 i_1 + \ldots + m_j i_j$ has the same parity as $j$ because of the relation $1i_1+2i_2+\ldots+ ji_j=j$. Hence, by~\eqref{eq:G_def}, $G_j(x)$ is a linear combination of Hermite polynomials $\Herm_k(x)$, where $k$ has the same parity as $j$. The statement follows from the relation $\Herm_k(-x) = (-1)^k \Herm (x)$.
\end{remark}

\begin{remark}
In Section~\ref{subsec:alternative_expression} we will show that $F_j(x;\beta) := W_\infty(\beta) G_j(x;\beta)$ is a polynomial in $x$ (which is evident) whose coefficients are \textit{linear combinations} (rather than rational functions) of $1, W_\infty(\beta), \ldots, W_\infty^{(j)}(\beta)$  (which is not evident). For example,
$$
W_\infty(\beta)\chi_1(\beta) = W'_\infty(\beta),
\quad
W_\infty(\beta)(\chi_1^2(\beta) + \chi_2(\beta)) = W''_\infty(\beta),
$$
thus proving the above claim for $G_2(x;\beta)$; see~\eqref{eq:G2}.
\end{remark}

Taking $r=0$ and $\beta=0$ in Theorem~\ref{theo:edgeworth_general} we obtain the following local limit theorem for the profile $\bL_n$.
\begin{theorem}\label{ref:LLT_general}
Let $\bL_1,\bL_2,\ldots$ be a sequence of random profiles satisfying Assumptions A1--A4. Then,
$$
\sqrt{w_n} \sup_{k\in\Z}
\left|\eee^{-\varphi(0) w_n} \bL_n(k) - \frac{W_\infty(0)}{\sigma(0) \sqrt{2\pi w_n}} \exp\left\{-\frac 12\left(\frac{k - \mu(0) w_n}{\sigma(0) \sqrt{w_n}}\right)^2 \right\}\right| \toas 0.
$$
\end{theorem}

Theorem~\ref{theo:edgeworth_general} contains one free parameter $\beta$ which can be chosen as a function of $k$ and $n$. With $\beta=0$ we obtain an asymptotic expansion complementing Theorem~\ref{ref:LLT_general}. On the other hand, it is natural to choose $\beta=\beta_{n}(k)$ as the solution to
\begin{align} \label{def:beta_n_k}
\varphi'(\beta_{n}(k)) = \frac{k}{w_n},\quad \frac{k}{w_n}\in \varphi'((\beta_{-},\beta_{+})),
\end{align}
where the strict monotonicity of $\varphi'$ has to be recalled. Then, $x_{n}(k) = 0$ by definition \eqref{eq:x_n_k_def}, and we obtain the following result.
\begin{theorem}\label{theo:LD_general}
Let $\bL_1,\bL_2,\ldots$ be a sequence of random profiles satisfying Assumptions A1--A4. Then, for all $r\in\N_0$ and any compact set $K\subset (\beta_-,\beta_+)$,
\begin{equation}\label{eq:general_LD_expansion}
w_n^{r+1} \sup_{k\in \Z \cap w_n \varphi'(K)}\left|\frac{\eee^{k\beta_n(k)}}{\eee^{\varphi(\beta_n(k))w_n}}   \LLL_n(k)  - \frac{W_\infty(\beta_n(k))}{ \sigma(\beta_n(k)) \sqrt{2\pi w_n}} \sum_{j=0}^{r} \frac{G_{2j}(0; \beta_n(k))}{w_n^{j}}\right| \toas  0.
\end{equation}
\end{theorem}
\begin{remark}
Note that only half-integer powers of $w_n$ are present in the sum in~\eqref{eq:general_LD_expansion} because $G_{2j-1}(0)=0$ for $j\in\N$; see Remark~\ref{rem:G_odd_even}. In particular, with $r=0$ we obtain a precise large deviations asymptotics
$$
w_n \sup_{k\in \Z \cap w_n \varphi'(K)}\left|\eee^{k\beta_n(k) - \varphi(\beta_n(k))w_n} \LLL_n(k)  - \frac{W_\infty(\beta_n(k))}{ \sigma(\beta_n(k)) \sqrt{2\pi w_n}}\right| \toas  0.
$$
\end{remark}

\begin{remark}[On mod-$\phi$-convergence] \label{rem:modphi}
Let $\phi$ be a non-degenerate infinitely divisible distribution with cumulant generating function $\eta(\beta) = \log \int_\R \eee^{\beta x} \phi(\dd x)$.
\citet{feray_meliot_nikeghbali} called a sequence of real random variables $(X_n)_{n\in\N}$  \textit{mod-$\phi$ convergent} with speed $w_n\to+\infty$ if
\begin{equation}\label{eq:mod_phi_def}
\lim_{n\to\infty} \frac{\E [\eee^{\beta X_n}]}{ \eee^{\eta(\beta)w_n}} = \Psi_{\infty}(\beta)
\end{equation}
locally uniformly on some strip $\{\beta\in\C\colon \delta_- < \Re \beta <\delta_+\}$,  where $\Psi_\infty(\beta)$ is an analytic function  which does not vanish on $(\delta_-,\delta_+)$. Variations of this definition can be found in~\cite{delbaen_kowalski_nikeghbali,jacod_kowalski_nikeghbali_mod_Gauss,kowalski_nikeghbali_mod_Poi,kowalski_nikeghbali_zeta,meliot_nikeghbali_statmech}.  Assuming that~\eqref{eq:mod_phi_def} holds with speed $O(w_n^{-r})$, for all $r\in\N$, they obtained Edgeworth expansions for both lattice and non-lattice $X_n$.  In particular, Theorem~3.4 of~\citet{feray_meliot_nikeghbali} is closely related to expansion~\eqref{eq:general_LD_expansion}. In our setting, the distribution of $X_n$, namely the function $k\mapsto \P[X_n=k]$, is replaced by the profile $k\mapsto \LLL_n(k)$ (which may be random). More importantly, the analogue of~\eqref{eq:mod_phi_def} given in Assumptions A2 and A3 holds in some open neighborhood $\fD$ of $(\beta_-,\beta_+)$, but it fails to hold in the strip $\{\beta\in\C\colon \beta_- < \Re \beta <\beta_+\}$ in our applications to random trees. The function $W_\infty$ replacing $\Psi_\infty$ may be random in our setting. Also note that the expansion in Theorem~\ref{theo:edgeworth_general} is uniform in the ``tuning'' parameter $\beta$ and its terms are given explicitly using Bell and Hermite polynomials, which is convenient in applications.
\end{remark}

\subsection{Mode and width}
Using the Edgeworth expansion stated in Theorem~\ref{theo:edgeworth_general} we can obtain limit theorems for the \textit{width}
$M_n$ and the \textit{mode} $u_n$ of the profile $k\mapsto \bL_n(k)$.  These are defined by
\begin{equation}\label{eq:width_mode_def}
M_n = \max_{k\in\Z} \bL_n(k),
\quad
u_n =\argmax_{k\in\Z} \bL_n(k).
\end{equation}
\begin{theorem}\label{theo:mode}
Consider a sequence of random profiles $\bL_1,\bL_2,\ldots$ satisfying Assumptions A1--A4. There is an a.s.\ finite random variable $K$ such that for $n>K$, the mode $u_n$ is equal to $\lfloor u_n^*\rfloor$ or $\lceil u_n^*\rceil$, where
\begin{equation}\label{eq:u_n_star}
u_n^* = \varphi'(0) w_n + \chi_1(0) - \frac{\kappa_3(0)}{2\sigma^2(0)}.
\end{equation}
\end{theorem}
\begin{remark}
The uniqueness of the $\argmax$ is a rather subtle question and is not discussed here (see, e.g., \cite{erdoes_on_hammersley} where uniqueness of the mode is proved for Stirling numbers of the first kind). In the case when the $\argmax$ is non-unique Theorem~\ref{theo:mode} has to be understood as follows: for $n>K$ there are at most two maximizers of $\LLL_n(k)$ and they belong to the set $\{\lfloor u_n^*\rfloor, \lceil u_n^*\rceil\}$.
\end{remark}

The next result on the width $M_n=\bL_n(u_n)$ is not surprising in view of the local limit Theorem~\ref{ref:LLT_general}.
\begin{theorem}\label{theo:width}
Consider a sequence of random profiles $\bL_1,\bL_2,\ldots$ satisfying Assumptions A1--A4. Then the width $M_n$ satisfies
\begin{equation}
\sigma(0) \sqrt{2\pi w_n} \eee^{-\varphi(0) w_n}  M_n   \toas W_\infty(0).
\end{equation}
\end{theorem}
In our applications to random trees, the limiting random variable $W_\infty(0)$ is a.s.\ constant. It is therefore natural to ask whether it is possible to obtain a more refined result with a non-degenerate limit.
\begin{theorem}\label{theo:width_limit_law}
Consider a sequence of random profiles $\bL_1,\bL_2,\ldots$ satisfying Assumptions A1--A4. Let
$$
\tilde M_n
:=
2\sigma^2(0) w_n \left( 1 - \frac{\sqrt{2\pi w_n}\, \sigma(0)\, M_n}{W_\infty(0) \eee^{\varphi(0)w_n}} \right).
$$
If $\theta_n := \min_{k\in\Z} |u_n^*-k|$ denotes the distance between $u_n^*$ and the nearest integer, then
$$
\tilde M_n - \theta_n^2 \toas \chi_2(0) + \frac{\kappa_3^2(0)}{6\sigma^4(0)} - \frac{\kappa_4(0)}{4\sigma^2(0)}.
$$

\end{theorem}

We conclude this section with several generalizations of our main results, all of which are consequences of the proof of Theorem \ref{theo:edgeworth_general}.
\begin{remark} Let $(\bL_t)_{t\geq 0}$ be a continuous-time profile satisfying the obvious continuous-time formulations of Assumptions A1--A4 for some 
real-valued function $(w_t)_{t\geq 0}$ with $w_t \to +\infty$ as $t \to \infty$. Then, all results in this section apply analogously to the profile $(\bL_t)_{t\geq 0}$.
\end{remark}
\begin{remark} All results remain valid if the sequence $w_n, n \in \N$, is random and $w_n \to +\infty$ almost surely.
\end{remark}
\begin{remark} Under Assumptions A1 and A4, if there exists a sequence of random analytic functions $\tilde W_n, n \in \N,$ on $\mathscr D$ such that
the convergence \eqref{eq:Psi_n_to_W_infty_speed} holds with $\tilde W_n$ instead of $W_\infty$, then Theorems \ref{theo:edgeworth_general}, \ref{ref:LLT_general} and \ref{theo:LD_general} hold with $W_\infty$ replaced
by $\tilde W_n$.
\end{remark}
\begin{remark} Let Assumptions A1 and A2 be fulfilled and assume that the
convergence \eqref{eq:Psi_n_to_W_infty_speed}  in Assumption A3 holds for some
real $r_1 \geq 1/2$ (rather than for all $r \in \N_0$), and  the convergence \eqref{assump_4} for some real
$r_2 > 1/2$.  Then, for any $r \leq 2 r_1-1$, $r \in \N_0$ and $r \leq \alpha < \min \{1+r, 2 r_2-1\}$,
we have
 $$w_n^{\frac{\alpha+1}{2}} \sup_{k\in\Z} \sup_{\beta\in K} \left|\eee^{\beta k - \varphi(\beta) w_n} \bL_n(k) - \frac{W_\infty(\beta) \eee^{-\frac 12 x_n^2(k)}}{\sigma(\beta) \sqrt{2\pi w_n}}  \sum_{j=0}^r \frac{G_j(x_n(k);\beta)}{w_n^{j/2}}   \right| \toas 0.$$
\end{remark}
\begin{remark} Similarly to the previous remark, take Assumptions A1 and A2 for granted and assume that \eqref{assump_4} holds for some real $r > 1/2$. Further, in the notation of Assumption A3, instead of \eqref{eq:Psi_n_to_W_infty_speed}, impose that
$$\sup_{\beta\in K} |W_n(\beta) - W_\infty(\beta)| \toas 0.$$
Then, $$\sup_{k\in\Z} \sup_{\beta\in K} \left|\eee^{\beta k - \varphi(\beta) w_n} \bL_n(k) - \frac{W_\infty(\beta) \eee^{-\frac 12 x_n^2(k)}}{\sigma(\beta) \sqrt{2\pi w_n}}     \right| \toas 0.$$
\end{remark}

\subsection{Example: The classical Edgeworth expansion}\label{subsec:edgeworth_iid}
In this and the subsequent section we consider two examples of deterministic profiles.
Let $Z_1,Z_2,\ldots$ be i.i.d.\ integer-valued random variables with mean $\mu:=\E Z_1$, variance $\sigma^2:=\Var Z_1\neq 0$, and cumulant generating function
\begin{equation}\label{eq:varphi_classical_edgeworth}
\varphi(\beta) := \log \E \eee^{\beta Z_1}
\end{equation}
which is finite on some interval $(\beta_-,\beta_+)$ containing zero.
Consider a sequence of deterministic profiles in which $\bL_n$ is defined as the probability mass function of the sum $Z_1+\ldots+Z_n$, that is
$$
\bL_n(k) = \P[Z_1+\ldots+Z_n = k], \quad k\in\Z.
$$
Then, Assumptions A1--A3 are satisfied with $\varphi$ as in~\eqref{eq:varphi_classical_edgeworth}, $w_n=n$ and $W_\infty(\beta) = W_n(\beta) = 1$. Hence, the cumulants $\chi_{k}$ vanish. Assumption A4 is also satisfied if we additionally assume that the minimal step of the distribution of $Z_1$ is $1$. In other words, there is no non-trivial sublattice $a\Z +b$, with $a\in \{2,3,\ldots\}$ and $b\in\Z$, such that $\P[Z_1\in a\Z+b]=1$.

Applying Theorem~\ref{theo:edgeworth_general} with $\beta=0$ we obtain the classical
Chebyshev--Edgeworth--Cram\'er asymptotic expansion  for sums of i.i.d.\ lattice random variables, see Theorem~13 in \citet[Ch.~VII, p.~205]{petrov_book}:
\begin{equation}\label{eq:asympt_expansion_random_walk}
\lim_{n\to\infty} n^{\frac {r+1}2} \sup_{k\in\N} \left|\P[Z_1+\ldots+Z_n=k] - \frac {\eee^{-\frac 12 x^2_n(k)}} {\sigma\sqrt {2\pi n}}
\sum_{j=0}^{r}
\frac {q_j(x_n(k))} {n^{j/2}}  \right|=0,
\end{equation}
where $q_j$ is a polynomial of degree at most $3j$ whose coefficients can be expressed through the cumulants $\kappa_2,\ldots,\kappa_{j+2}$. To obtain $q_j$, remove in $G_j$ (see~\eqref{eq:G_def} for its definition) all terms involving the $\chi_k$'s.
The first three terms in the expansion are given by
\begin{align}
q_0(x) = 1,
\quad
q_1(x) = \frac{\kappa_3}{6\sigma^3} \Herm_3(x),
\quad
q_2(x) = \frac{\kappa_4}{24 \sigma^4} \Herm_4(x) + \frac{\kappa_3^2}{72 \sigma^6} \Herm_6(x).
\end{align}
Applying Theorem~\ref{theo:edgeworth_general} with arbitrary $\beta$ one can obtain asymptotic expansions for large deviation probabilities; see~\cite{gruebel_kabluchko_BRW}. Note, however, that the moment condition which we imposed on $Z_1$ can be relaxed; see Theorem~13 in \citet[Ch.~VII, p.~205]{petrov_book}.

\subsection{Example: Stirling numbers of the first kind}
The (unsigned) \textit{Stirling numbers} of the first kind are defined by the formula
\begin{equation}\label{eq:stirling_def}
\theta^{(n)} := \theta(\theta+1)\ldots (\theta+n-1)  = \sum_{k=1}^n \stirling{n}{k} \theta^k.
\end{equation}
The following  sequence of deterministic profiles given by the probability mass function of the \textit{Ewens distribution} with parameter $\theta>0$
\begin{equation}\label{eq:stirling_profile}
\bL_n(k)= \frac {\theta^k} {\theta^{(n)}}\stirling{n}{k}\ind_{\{k\in \{1,\ldots,n\}\}} 
\end{equation}
can be shown to satisfy Assumptions A1--A4.
Applications to Stirling numbers and the Ewens distribution will be studied in a separate paper~\cite{kabluchko_marynych_sulzbach_on_stirling}.

\section{Edgeworth expansions for random trees}\label{sec:edgeworth_for_random_trees}

\subsection{One-split branching random walk}\label{subsec_one_split_BRW_def}

Consider a system of particles on $\Z$ which evolves in discrete time as  follows. At time $0$, we have a single particle located at $0$. In each step \textit{one} of the particles is chosen uniformly at random. This particle is replaced by a random cluster of particles whose displacements w.r.t.\ the original particle are described by  a point process $\zeta = \sum_{i=1}^N \delta_{Z_i}$ (where $N$, the number of particles, is  a.s.\ finite) on $\Z$. In other words, if the original particle is located at $x$, its descendants are located at $x+Z_1,\ldots, x+Z_N$. All random mechanisms involved in this definition are independent.

\begin{remark} The difference between this model and the usual discrete-time, many splits BRW (for which the Edgeworth expansion was obtained in~\cite{gruebel_kabluchko_BRW}) is that in the one-split BRW, only \textit{one} particle (chosen uniformly at random) is allowed to split, whereas in the many-split BRW \textit{all} particles split at the same time.
 We shall see that there are many differences between these models.
\end{remark}

Denote by $S_n$ the number of particles after $n$ splitting events, and let their positions be $x_{1,n},\ldots, x_{S_n,n}$. Let us denote by $\bL_n(k)$ the number of particles at site $k\in\Z$ after $n$ splitting events:
\begin{equation}\label{eq:L_T_k_def_trees}
\bL_n(k) =  \#\{1\leq j\leq S_n \colon x_{j,n} = k\}.
\end{equation}
We are interested in the function $k\mapsto \bL_n(k)$ which is called  the \textit{profile} of the one-split BRW.

We are going to state our assumptions on the one-split BRW. Denote by $\nu_k$ the expected number of particles at site $k\in\Z$ in the cluster process $\zeta$:
\begin{equation}\label{eq:nu_k_def}
\nu_k = \E \zeta(\{k\})=\E \left[\sum_{i=1}^{N}\ind_{\{Z_i=k\}}\right], \quad k\in\Z.
\end{equation}

The first assumption states that non-zero jumps are possible with positive probability and  thus excludes the case in which all particles stay at $0$. The second assumption requires the one-split BRW to be supercritical and excludes the possibility that it can become extinct.

\vspace*{2mm}
\noindent
\textbf{Assumption B1:}
We have $\nu_k>0$ for at least one $k\in \Z \bsl\{0\}$. 

\vspace*{2mm}
\noindent
\textbf{Assumption B2:} The cluster point process $\zeta$ is a.s.\ non-empty, and the probability that it has at least $2$ particles is positive. In other words, $N\geq 1$ a.s.\ and $\P[N=1]\neq 1$.


\begin{remark}
It is possible to replace this assumption by a weaker one requiring that $\E N > 1$ (supercriticality), in which case all results would hold a.s.\ on the survival event.
\end{remark}

Denote by $m(\beta)$ the moment generating function of the intensity of the cluster point process $\zeta$ minus $1$:
\begin{equation}\label{eq:m_def}
m(\beta) = \sum_{k\in \Z} \eee^{\beta k} \nu_k - 1 = \E\left[\sum_{i=1}^{N}\eee^{\beta Z_i}\right]-1.
\end{equation}
The expected number of particles at time $n$ is $\E S_n = 1 + m(0) n$, where, by Assumption B2,
\begin{equation}\label{eq:m_0_def}
m(0) = \sum_{k\in\Z} \nu_k - 1 = \E N - 1 > 0.
\end{equation}


\noindent
\textbf{Assumption B3:} The function $m$ is finite on some non-empty open interval $\fI$ containing $0$.

\vspace*{2mm}

It follows that the function $m$ is well-defined for $\beta \in \{z \in \C : \Re z \in \fI\}$ and strictly convex and infinitely differentiable on $\fI$. We shall need the function
$$
\varphi(\beta) = \frac{m(\beta)}{m(0)}, \quad \Re \beta \in \fI.
$$
Denote by $(\beta_-,\beta_+)\subset \fI$ the open interval on which $\varphi'(\beta)\beta < \varphi(\beta)$:
\begin{align}
\beta_-&=\inf\{\beta\in\fI \colon \varphi'(\beta) \beta < \varphi(\beta)\},\\
\beta_+&=\sup\{\beta\in\fI \colon \varphi'(\beta) \beta < \varphi(\beta)\}.
\end{align}
The interval $(\beta_-,\beta_+)$ is non-empty because it contains $0$.
The endpoints of the intervals $\fI$ and $(\beta_-,\beta_+)$ are allowed to be infinite.

The (normalized) moment-generating function of the one-split BRW profile is defined, for  $\Re \beta \in \fI$, by
\begin{equation}\label{eq:jabbour_one_split}
W_n(\beta) = \frac 1 {n^{\varphi(\beta)}}   \sum_{i=1}^{S_n} \eee^{\beta x_{i,n}}.
\end{equation}

The following aperiodicity  condition plays an important role in the verification of Assumption~A4. Here, and subsequently, we denote by
$\nu=\sum_{k\in \Z} \nu_k \delta_k$ the intensity measure of the point process $\zeta$.

\vspace*{2mm}
\noindent
\textbf{Assumption B4:} $\nu$ is not concentrated on any proper additive subgroup of $\Z$. In other words, $\nu(\Z \backslash a\Z)\neq 0$ for all $a \in \{2, 3, \ldots\}$.

\vspace*{2mm}

Assumption B4 can be imposed without loss of generality: if $\nu(a^* \Z) = 1$ for some $a^* \geq 2$ (chosen to be maximal with this property), then we can rescale the jumps by $a^*$ and work equivalently with the one-split BRW governed by the intensity measure $\nu^*$, where $\nu^*(\{k\}) = \nu(\{k / a^*\})$.
Note that this contrasts the situation in the many-split BRW~\cite{gruebel_kabluchko_BRW} and in Section~\ref{subsec:edgeworth_iid},  where it was necessary to exclude measures $\nu$ concentrated on lattices of the form $a\Z+b$.  

\vspace*{2mm} Finally, we also need the following moment condition which supplements Assumption~B3.

\vspace*{2mm}
\noindent
\textbf{Assumption B5:}
For any $\beta \in (\beta_-, \beta_+)$ there is $\gamma= \gamma(\beta)>1$ such that
$$
\E\left[ \left(\sum_{i=1}^{N}\eee^{\beta Z_i} \right)^\gamma\right] < \infty.
$$
\begin{remark}
This is easily shown to be equivalent to the following assumption: For every compact set $K\subset (\beta_-,\beta_+)$ there is $\gamma= \gamma(K)>1$ such that the above expectation is bounded uniformly in $\beta\in K$.
\end{remark}

The next theorem states that the sequence of the one-split BRW profiles satisfies Assumptions A2 and A3 with $w_n=\log n$.

\begin{theorem}\label{theo:W_n_converges_one_split_BRW}
Under Assumptions B1--B3 and B5, there is an open neighborhood $\fD$ of the interval $(\beta_-,\beta_+)$ in the complex plane such that, with probability $1$, $W_n$ converges to some random analytic function $W_\infty$ locally uniformly on $\fD$. Moreover, for every compact set $K\subset \fD$ and $r\in\N$ we can find an a.s.\ finite random variable $C_{K,r}$ such that for all $n\in\N$,
\begin{equation}\label{eq:Psi_n_to_W_infty_speed_one_split}
\sup_{\beta\in K} |W_n(\beta) - W_\infty(\beta)| < C_{K, r} (\log n)^{-r}.
\end{equation}
With probability $1$, the function $W_\infty$ has no zeros on the interval $(\beta_-,\beta_+)$. 
\end{theorem}

The proof of the theorem will be given in Section~\ref{subsec:embedding} and uses an embedding into a continuous-time BRW  in conjunction with results of \citet{biggins_uniform} (see also~\cite{uchiyama}). The explicit form of the neighborhood $\fD$ plays no role in the sequel. However, let us stress that we cannot take $\fD$ to be the strip $\{\beta\in \C\colon \beta_- < \Re \beta < \beta_+\}$.  In the case of the BSTs, the exact shape of $\fD$ can be found in~\cite{chauvin_etal}: it is a bounded set.
For this reason, the  asymptotic expansion obtained by \citet{feray_meliot_nikeghbali} does not apply directly.

\begin{remark}\label{rem:W_infty_0}
Note that $\varphi(0)=1$ and by the law of large numbers,
$$
W_\infty(0) = \lim_{n\to\infty} W_n(0) = \lim_{n\to\infty} \frac {S_n} n =  m(0) \;\;\; \text{a.s.}
$$
\end{remark}

\subsection{Random trees and one-split BRWs}\label{subsec:random_trees_one_split}
We can identify profiles of random trees and profiles of the one-split BRW as follows: Particles correspond to (external or internal) nodes, and positions of particles correspond to the depths of the nodes.
In the following,  we describe the cluster point process, and give explicit formulas for the quantities $m(0)$, $\varphi(\beta)$, $\mu(0)=\varphi'(0)$, $\sigma^2(0)=\varphi''(0)$ and $\kappa_j(0)= \varphi^{(j)}(0)$, $j\in\N$, which will be relevant in our limit theorems.

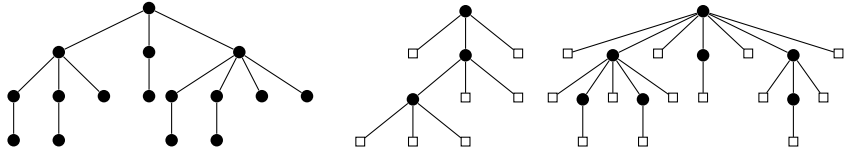
\begin{figure}[htb]
  \centering
  \begin{minipage}[t]{.33\linewidth}
    \centering
    \begin{tikzpicture}[level distance=0.5cm,
level 1/.style={sibling distance=1.2cm},
level 2/.style={sibling distance=0.6cm},
level 3/.style={sibling distance=0.6cm},
level 4/.style={sibling distance= .6cm},
level 5/.style={sibling distance= .5cm},
level 6/.style={sibling distance= 1cm},
int/.style = {circle, scale = 0.5, draw=none, fill=black, anchor = north, growth parent anchor=south},
ext/.style = {rectangle, scale = 0.5, anchor = north, growth parent anchor=south}]
\tikzstyle{every node}=[circle,draw]

\node (Root) [int] {}
child { node [int] {}
 child { node [int] {}
  child { node [int] {}
  }
 }
 child { node [int] {}
  child { node [int] {}
  }
 }
 child { node [int] {}
 }
}
child { node [int]{}
 child { node [int]{}
 }
}
child { node [int]{}
 child { node [int] {}
  child { node [int] {}
  }
 }
 child { node [int] {}
  child { node [int] {}
  }
 }
 child { node [int] {}
 }
 child { node [int] {}
 }
}
;
\end{tikzpicture}
  \end{minipage} \hspace{-0.7cm} \begin{minipage}[t]{.33\linewidth}
    \centering
   \begin{tikzpicture}[level distance=0.5cm,
level 1/.style={sibling distance=0.7cm},
level 2/.style={sibling distance=0.7cm},
level 3/.style={sibling distance=0.7cm},
int/.style = {circle, scale = 0.5, draw=none, fill=black, anchor = north, growth parent anchor=south},
ext/.style = {rectangle, scale = 0.5, anchor = north, growth parent anchor=south}]
\tikzstyle{every node}=[circle,draw]

\node (Root) [int] {}
child { node [ext] {}
}
child { node [int] {}
 child { node [int] {}
  child { node [ext] {}
  }
  child { node [ext] {}
  }
  child { node [ext] {}
  }
 }
 child { node [ext] {}
 }
 child { node [ext] {}
 }
}
child { node [ext] {}
}
;
\end{tikzpicture}
  \end{minipage} \hspace{-1cm} \begin{minipage}[t]{.33\linewidth}
    \centering
    \begin{tikzpicture}[level distance=0.5cm,
level 1/.style={sibling distance=0.6cm},
level 2/.style={sibling distance=0.4cm},
level 3/.style={sibling distance=0.4cm},
level 4/.style={sibling distance=0.4cm},
level 5/.style={sibling distance=0.8cm},
int/.style = {circle, scale = 0.5, draw=none, fill=black, anchor = north, growth parent anchor=south},
ext/.style = {rectangle, scale = 0.5, anchor = north, growth parent anchor=south}]
\tikzstyle{every node}=[circle,draw]
\node (Root) [int] {}
child { node [ext] {}
}
child { node [int] {}
 child { node [ext] {}
 }
 child { node [int] {}
  child { node [ext] {}
  }
 }
 child { node [ext] {}
 }
 child { node [int] {}
  child { node [ext] {}
  }
 }
 child { node [ext] {}
 }
}
child { node [ext] {}
}
child { node [int] {}
 child { node [ext] {}
 }
}
child { node [ext] {}
}
child { node [int] {}
 child { node [ext] {}
 }
 child { node [int] {}
  child { node [ext] {}
  }
 }
 child { node [ext] {}
 }
}
child { node [ext] {}
}
;
\end{tikzpicture}
  \end{minipage}

  \caption{{\small Sample realizations of random trees. Left: RRT. Middle: $D$-ary recursive tree with $D=3$. Right: PORT.}
\label{bild:trees_sample}}
\end{figure}


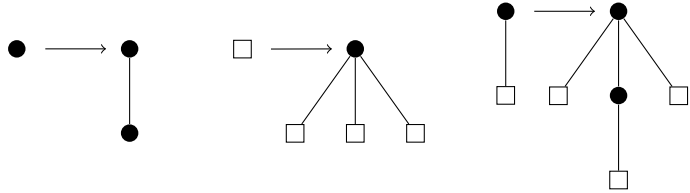
\begin{figure}
\centering
\begin{tikzpicture}[level distance=1cm,
level 1/.style={sibling distance=0.8cm},
level 2/.style={sibling distance=0.7cm},
int/.style = {circle, scale = 0.7, draw=none, fill=black, anchor = north, growth parent anchor=south},
ext/.style = {rectangle, scale = 1.0, draw, anchor = north, growth parent anchor=south}]
\draw (4,0) node [int] (A) {} ;
 \node(Root) at (5.5,0) [int] {}
child { node [int] {}
} ;
\draw [shorten >=1cm,shorten <=1cm,>-] (A) -- (Root);
\draw (7,0) node [ext] (B) {} ;
 \node(Root2) at (8.5,0) [int] {}
child { node [ext] {}
}
child { node [ext] {}
}
child { node [ext] {}
};
\draw [shorten >=1cm,shorten <=1cm,>-] (B) -- (Root2);

\node at (10.5,0.5)  [int] (C) {} ;
\node at (10.5,-0.6162)  [ext] (e) {};

 \node(Root3) at (12,0.5) [int] {}
child { node [ext] {}
}
child { node [int] {}
 child { node [ext] {}
 }
}
child { node [ext] {}
};
\draw [shorten >=1cm,shorten <=1cm,>-] (C) -- (Root3);
\draw [-] (C) -- (e);

\end{tikzpicture}
\caption
{\small Construction rules for random trees. Left: RRT. Middle: $D$-ary recursive tree with $D=3$. Right: PORT.}
\label{bild:trees_rules}
\end{figure}


\vspace*{2mm}
\noindent
\textrm{(i)} External profiles of BSTs defined in Section~\ref{subsec:BST_special_case} correspond to the one-split BRW with the deterministic displacement point process $\zeta=2\delta_1$ because at any step of the construction an external node at depth $k$ is replaced by two new external nodes at depth $k+1$; see Figure~\ref{bild:BST} (right).  We have
$$
\varphi(\beta)=2\eee^{\beta}-1,\;\;
m(0)=1,\;\;
\mu(0)=\sigma^2(0)=\kappa_j(0)=2,
\;\; j\in\N.
$$
The constants $\beta_-\approx -1.678$ and $\beta_+\approx 0.768$ are the solutions of $2\eee^{\beta} (1-\beta) = 1$.

\vspace*{2mm}
\noindent
\textrm{(ii)} Random recursive trees (RRTs), see Figure~\ref{bild:trees_sample} (left), can be defined as follows. At time $n=0$ start with one node (denoted by $\bullet$) at level $0$. At any step, pick one of the existing nodes (say, $x$) uniformly at random and connect it to a new node one level deeper than $x$; see Figure~\ref{bild:trees_rules} (left).   Let $\LLL_n(k)$ be the number of nodes at depth $k$ in a RRT with $n+1$ nodes. RRTs correspond to the one-split BRW with the deterministic displacement point process $\zeta=\delta_0+\delta_1$.
In particular,
$$
\varphi(\beta) = \eee^{\beta}, \;\;
m(0) = 1,\;\;
\mu(0) = \sigma^2(0) = \kappa_j(0) = 1,
\;\; j\in\N.
$$
We have $\beta_-=-\infty$ and $\beta_+=1$.
For results on the profile of RRTs, we refer to~\cite{devroye_hwang,drmota_hwang,fuchs_hwang_neininger,schopp},  see also~\cite[Section~6.3]{drmota_book} for a detailed discussion of the three main methods applied in this context: the martingale method, the method of moments and the contraction method.

\vspace*{2mm}
\noindent
\textrm{(iii)} $D$-ary recursive trees\footnote{Not to be confused with $m$-ary \emph{search} trees, which is a different model; see~\cite[Section~1.4.2]{drmota_book}.} with $D\in\{2,3,\ldots\}$ {are a special case of so-called increasing trees introduced by Bergeron, Flajolet and Salvy \cite{Bergeron1992}}, see also~\cite[Sections~1.3.3 and~6.5]{drmota_book} and~\cite{schopp} for results on the profile. The model reduces to BSTs for $D=2$; see Figure~\ref{bild:trees_sample} (middle). These trees can be constructed in a similar manner as BSTs with the only difference that at each step $D$ new external nodes are attached; see Figure~\ref{bild:trees_rules} (middle).  The external profile of $D$-ary recursive trees correspond to the one-split BRW with the displacement point process $\zeta=D\delta_1$. We have
$$
\varphi(\beta)= \frac{D\eee^{\beta}-1}{D-1},\;\;
m(0)=D-1,\;\;
\mu(0) = \sigma^2(0) = \kappa_j(0) = \frac D{D-1},
\;\; j\in\N.
$$
The constants $\beta_- < 0$ and $\beta_+ > 0$ are the  solutions of $D \eee^{\beta}(1-\beta) = 1$.

\vspace*{2mm}

\vspace*{2mm}
\noindent
\textrm{(iv)} Plane-oriented recursive trees (PORTs), see Figure~\ref{bild:trees_sample} (right) for a sample realization and~\cite[Section~1.3.2]{drmota_book} for a discussion of this model, are constructed in the following way. At time $0$ start with an internal node $\bullet$ at level $0$ connected to an external node $\Box$ at level $1$:

\begin{center}
\begin{tikzpicture}[level distance=0.8cm,
level 1/.style={sibling distance=0.7cm},
int/.style = {circle, scale = 0.7, draw=none, fill=black, anchor = north, growth parent anchor=south},
ext/.style = {rectangle, scale = 0.7, draw, anchor = north, growth parent anchor=south}]
 \node(Root) [int] {}
 child {node [ext] {}};
\end{tikzpicture}
\end{center}
At each step choose one external node uniformly at random, declare it internal and add $3$ new external nodes as shown on Figure~\ref{bild:trees_rules} (right). After $n$ steps we obtain a tree with $2n+1$ external nodes. As opposed to BSTs, RRTs and $D$-ary recursive trees, the external profiles of PORTs follow the dynamics of a one-split BRW initiated with one particle (external node) at position \emph{one} (short: initiated at one) at time zero. The displacement point process is $\zeta=2\delta_0+\delta_{1}$.
We obtain
$$
\varphi(\beta)=\frac 12 (\eee^{\beta}+1),\;\;
m(0)=2,\;\;
\mu(0) =\sigma^2(0) = \kappa_j(0) = \frac 12,
\;\; j\in\N.
$$
We have $\beta_-=-\infty$, whereas $\beta_+\approx 1.278$ is the solution of $\eee^{\beta}(\beta-1) = 1$. The profile of PORTs was studied in~\cite{hwang_PORT,katona,schopp,sulzbach}.

\vspace*{2mm}
\noindent
\textrm{(v)} $p$-oriented trees (which reduce to PORTs if $p=2$) correspond to the one-split BRW initiated at one with $\zeta= p \delta_0 + \delta_{1}$, where $p\in \{2,3,\ldots\}$. They also fall under the general model introduced in~\cite{Bergeron1992}. We have
$$
\varphi(\beta)= \frac 1  p (\eee^\beta + p -1),\;\;
m(0)= p ,\;\;
\mu(0) = \sigma^2(0) = \kappa_j(0) = \frac 1  p,
\;\; j\in\N.
$$
We have $\beta_-=-\infty$, whereas $\beta_+$ is the solution of $\eee^{\beta}(\beta-1) = p-1$.
{For further information on $p$-oriented trees, we refer to Sections 1.3.3 and 6.5 in Drmota's monograph~\cite{drmota_book}.} 

\begin{remark}
Writing $(\LLL_n(k))_{k \in\N}$ for the external profile of PORTs (or $p$-oriented trees), and $(\LLL_n^*(k))_{k \in\N_0}$, for the profile of the corresponding standard one-split BRW initiated at zero, we can identify
$\LLL_n(k) = \LLL_{n}^*(k-1)$, for $n\in\N_0, k \in \N$. Denoting by $W_\infty^*(\beta)$ the almost sure limit in Theorem~\ref{theo:W_n_converges_one_split_BRW} for the profile $(\LLL_n^*(k))_{k \in \N_0}$, the limiting process $W_\infty(\beta)$ for the profile $(\LLL_n(k))_{k \in \N}$ is equal to $\eee^\beta  W_\infty^*(\beta)$. In particular, for the random cumulants, we have
$\chi_1(\beta) = 1+ \chi_1^*(\beta)$ and $\chi_k(\beta) = \chi_k^*(\beta)$ for all $k \geq 2$.
\end{remark}

\begin{remark}
In all examples listed above the displacement point process $\zeta$ is concentrated on $\{0,1\}$ and therefore we have $\varphi(\beta)= 1 + \varphi'(0)(\eee^{\beta}-1)$ (since $\varphi(0)=1$ by definition) and hence, almost surely
$$
n^{\varphi(\beta)} =  n \cdot \eee^{(\varphi(\beta)-1) \log n} = n\cdot \eee^{\varphi'(0) (\eee^{\beta} - 1) \log n}\sim \frac {S_n} {m(0)}
\eee^{\varphi'(0) (\eee^{\beta} - 1) \log n}, \; n\to\infty.
$$
Thus, Theorem~\ref{theo:W_n_converges_one_split_BRW} states that the sequence $(S_n^{-1} \sum_{k\in\N_0} \LLL_n(k) \delta_k)_{n\in\N}$ of random probability measures on $\Z$ converges in the mod-Poisson sense with probability $1$; see~\cite{kowalski_nikeghbali_mod_Poi} and Remark \ref{rem:modphi}.
\end{remark}

\vspace*{2mm}
\noindent
\textrm{(vi)}
So far we considered ``horizontally projected profiles''. The \textit{vertically projected external profile} of a binary search tree can be defined as follows.
At time $0$, assign to the root of the BST the abscissa $0$. During the construction of the BST, if some external node with abscissa $i$ is chosen, then its two descendants are assigned abscissas $i-1$ and $i+1$.
The abscissa of a node describes its so-called \emph{left-right imbalance} since it measures the difference between the number of times the path from the root to the node turns right rather than left.
Denote by $\LLL_n(k)$ the number of external nodes with abscissa $k\in\Z$ in a BST with $n+1$ external nodes. This profile corresponds to the one-split BRW with $\zeta = \delta_{-1} + \delta_{+1}$ and we have
$$
\varphi(\beta)= \eee^\beta + \eee^{-\beta} -1,\;
m(0)= 1 ,\;
\mu(0) = \kappa_{2j-1}(0) = 0, \sigma^2(0) = \kappa_{2j}(0) = 2,
\; j\in\N.
$$
The constants $\beta_+\approx 0.9071$ and $\beta_-=-\beta_+$ are the solutions of the equation $(\eee^{\beta} - \eee^{-\beta}) \beta = \eee^{\beta} + \eee^{-\beta}-1$. The left-right imbalance of nodes and the corresponding path length were studied by Kuba and Panholzer~\cite{kubapanhol}, the profile by Schopp \cite{schopp}.

\subsection{Jabbour martingale}
In all models listed in the previous section, the number of descendants of any particle in the one-split BRW is deterministic. Recall that this number equals  $m(0) + 1$, so that for BSTs, RRTs and PORTs we have $m(0) = 1, 1, 2$, respectively. In this case, it turns out that the Laplace transform of the particle positions divided by its expectation is a martingale. In the case of BSTs, this martingale has been introduced by \citet{jabbour}; see also \citet{chauvin_drmota_jabbour}. The next theorem generalizes Jabbour's martingale to general one-split BRWs with a deterministic number of descendants.

\begin{theorem}
Consider a one-split branching random walk in which the number of descendants of every particle is deterministic and equals $m(0)+1\in\N$.  Assume that the function $m(\beta)$ defined by~\eqref{eq:m_def} is finite on some interval $\fI$ containing $0$.   Then, for all  $\beta\in \{z \in \C \colon \Re z \in \fI\}$, the sequence $(J_n(\beta))_{n\in\N_0}$ defined by
$$
J_n(\beta) := \frac 1 {\alpha_n(\beta)} \sum_{i=1}^{1 + m(0) n} \eee^{\beta x_{i,n}},
\quad
\alpha_n (\beta) = \prod_{k=0}^{n-1} \left(1 + \frac{m(\beta)}{1 + m(0) k} \right),
$$
is a martingale w.r.t.\ the filtration $(\cF_n)_{n\in\N_0}$, where $\cF_n$ is a $\sigma$-algebra generated by the first $n$ generations of the one-split BRW.  Also, $\E J_n(\beta) = 1$.
\end{theorem}
\begin{proof}
Note that the number of particles at time $n$ is $S_n = 1+ m(0) n$. Let $Z_n(\beta) =  \sum_{i=1}^{1 + m(0) n} \eee^{\beta x_{i,n}}$, where $x_{1,n},\ldots, x_{S_n,n}$ are the positions of the particles at time $n$.  Denoting by $\zeta_n=\sum_{j=1}^{m(0)+1}\delta_{Z_{j,n}}$ the point process of descendants used to pass from generation $n$ to generation $n+1$, we have
\begin{align*}
\E[Z_{n+1}(\beta)|\cF_n]
&=
\frac 1 {1+m(0) n} \sum_{i=1}^{1 + m(0) n} \E \left[Z_n(\beta) - \eee^{\beta x_{i,n}} + \sum_{j=1}^{m(0) +1} \eee^{\beta (x_{i,n} + Z_{j,n})} \Bigg|\cF_n \right]\\
&=
Z_n(\beta) + \left[\sum_{i=1}^{1 + m(0) n} \eee^{\beta x_{i,n}} \left(\E \sum_{j=1}^{m(0) +1} \eee^{\beta Z_{j,n}} - 1\right)\right]\frac{1}{1+m(0)n}\\
&=
Z_n(\beta) \left(1 + \frac{m(\beta)}{1+m(0) n} \right),
\end{align*}
where we used that $m(\beta) = \E \sum_{j=1}^{m(0) +1} \eee^{\beta Z_{j,n}} - 1$; see~\eqref{eq:m_def}. It follows that $J_n(\beta)$ is a martingale. Since $J_0(\beta) = 1$, we have $\E J_n(\beta) = 1$ for all $n\in\N_0$.
\end{proof}

For the function $W_n$ introduced in~\eqref{eq:jabbour_one_split} we obtain (in the case of deterministic number of descendants)
\begin{align} \label{mean_w_exact}
\E W_n(\beta) = \frac{\alpha_n(\beta)}{n^{\varphi(\beta)}} =
n^{-\frac{m(\beta)}{m(0)}} \frac{\left(\frac{m(\beta) + 1}{m(0)}\right)^{(n)}}{\left(\frac{1}{m(0)}\right)^{(n)}},
\quad  \Re \beta\in \fI,
\end{align}
where $z^{(n)}:=z(z+1)\ldots (z+n-1)$ is the rising factorial. For $\Re \beta\in \fI$, using the formula $z^{(n)} \sim  n^z \Gamma(n)/\Gamma(z)$ as $n\to\infty$, we obtain
\begin{align} \label{mean_wlimit_exact} \lim_{n\to\infty} \E W_n(\beta)  = \frac{\Gamma\left(\frac{1}{m(0)}\right)}{\Gamma\left(\frac{m(\beta) + 1}{m(0)}\right)}. \end{align}
Further, for  $\beta\in (\beta_-,\beta_+)$ we shall show that $\E W_\infty(\beta) = \lim_{n\to\infty} \E W_n(\beta)$; see Section~\ref{subsec:proof_edgeworth_mean}, below.



\subsection{Cumulants of the profile}\label{subsec:cumulants_one_split}
Recall that $x_{1,n}, \ldots, x_{S_n,n}$ denote the positions of the particles in a one-split BRW after $n$ splits. For $\beta\in \fI$ consider
$$
\chi_{k,n}(\beta) = \left(\frac{\dd}{\dd \beta}\right)^k \log \sum_{i=1}^{S_n} \eee^{\beta x_{i,n}}.
$$
It is easy to see that $\chi_{k,n}(\beta)$ is the $k$-th cumulant of the Gibbs probability measure assigning to each point $x_{i,n}$ the weight proportional to $\eee^{\beta x_{i,n}}$.
\begin{remark}
The most interesting case is $\beta =0$. Then, $\chi_{k,n} := \chi_{k,n}(0)$ is the $k$-th cumulant of the empirical measure assigning to each particle $x_{i,n}$ the same weight $1/S_n$. For example,
\begin{align*}
\chi_{1,n} = \frac 1{S_n} \sum_{i=1}^{S_n} x_{i,n},\;\;
\chi_{2,n} = \frac 1{S_n} \sum_{i=1}^{S_n} (x_{i,n}-\chi_{1,n})^2,\;\; 
\chi_{3,n} = \frac 1{S_n} \sum_{i=1}^{S_n} (x_{i,n}-\chi_{1,n})^3
\end{align*}
are the empirical mean, the empirical variance and the empirical central third moment of the particle positions in the one-split BRW.
In the context of random trees, $S_n \chi_{1,n}$ is the external path length of the tree.  Specifically, in the BST case,  $(n+1)\chi_{1,n} - 2n$ can be interpreted as the number of key comparisons used by the Quicksort algorithm to sort $n$ randomly ordered items.
\end{remark}
\begin{theorem}\label{theo:chi_one_split_BRW}
Consider a one-split BRW satisfying Assumptions B1--B3 and B5. Uniformly on any compact set $K\subset  (\beta_-,\beta_+)$ we have
\begin{equation}\label{eq:theo:chi_one_split_BRW}
\left(\frac{\dd}{\dd \beta}\right)^k \log  W_n (\beta)  =  \chi_{k,n} (\beta) - \varphi^{(k)}(\beta) \log n \toas \chi_{k}(\beta),
\end{equation}
where the limiting random variable $\chi_k(\beta)$ is given by
$$
\chi_k(\beta) = \left(\frac{\dd}{\dd \beta}\right)^k \log W_\infty(\beta).
$$
\end{theorem}
\begin{proof}
The equality in~\eqref{eq:theo:chi_one_split_BRW} follows from the definition of $W_n$; see~\eqref{eq:jabbour_one_split}. We prove the convergence.
Let $\fD$ be an open neighborhood of $(\beta_-,\beta_+)$ as in Theorem~\ref{theo:W_n_converges_one_split_BRW}. In the probability space on which the
one-split BRW is defined, consider some outcome $\omega$ and let $\fD':=\fD'(\omega) \subseteq \fD$ be an open subset with $K \subseteq \fD'$ such
that $W_\infty$ is almost surely non-zero on the closure of $\fD'$, and the analytic functions $W_n(\cdot;\omega)$ converge, as $n\to\infty$, to $W_\infty(\cdot;\omega)$
in $\cH(\fD')$, the set of analytic functions on $\fD'$ with the topology of locally uniform convergence. The set of such outcomes has probability $1$. By Theorem~\ref{theo:W_n_converges_one_split_BRW},
$$
\log W_n (\beta) \toas \log W_\infty(\beta) \quad \text{in } \cH(\fD').
$$
Observe that the logarithm can be defined continuously since $W_\infty$ and $W_n$ do not vanish for sufficiently large $n$.
By the Cauchy formula, taking the $k$-th derivative is a continuous map from $\cH(\fD')$ to $\cH(\fD')$.  Consequently,
$$
\left(\frac{\dd}{\dd \beta}\right)^k \log  W_n (\beta) \toas \left(\frac{\dd}{\dd \beta}\right)^k \log W_\infty(\beta) \quad \text{in } \cH(\fD'),
$$
and hence uniformly in $\beta\in K$. This concludes the proof.
\end{proof}

\subsection{Edgeworth expansion for one-split BRW}
We are going to state an Edgeworth expansion for the profile of the one-split BRW. We recall the parameters $\mu(\beta)$ and $\sigma(\beta)$ from~\eqref{eq:mu_sigma_def},
\begin{equation}\label{eq:mu_sigma_def_one_split}
\mu(\beta) = \varphi'(\beta),
\;\;\;
\sigma^2(\beta) = \varphi''(\beta),
\end{equation}
as well as the deterministic cumulants $\kappa_j(\beta) = \varphi^{(j)}(\beta)$, $j\in\N$. As in~\eqref{eq:x_n_k_def} (with $w_n=\log n$), we introduce the normalized coordinate
\begin{equation}\label{eq:x_n_k_def_one_split}
x_n(k) = x_n(k;\beta) = \frac{k-\mu(\beta) \log n}{ \sigma(\beta) \sqrt{\log n}}, \quad k\in\Z.
\end{equation}

\begin{theorem}\label{theo:edgeworth_one_split_BRW}
Let $(\bL_n(k))_{k\in\Z}$ be the profile at time $n$ of a one-split branching random walk satisfying Assumptions B1--B5. Fix $r\in\N_0$ and a compact set $K\subset (\beta_-,\beta_+)$. Then,
\begin{equation}
(\log n)^{\frac{r+1}{2}}
\sup_{k\in\Z} \sup_{\beta\in K}
\left|\frac{\eee^{\beta k}\bL_n(k)}{n^{\varphi(\beta)}} - \frac{W_\infty(\beta) \eee^{-\frac 12 x_n^2(k)}}{\sigma(\beta) \sqrt{2\pi \log n}}  \sum_{j=0}^r \frac{G_j(x_n(k);\beta)}{(\log n)^{j/2}} \right| \toas 0.
\end{equation}
Here, $G_j(x)=G_j(x;\beta)$ is a polynomial of degree at most $3j$ given by
\begin{equation}\label{eq:G_def_one_split}
G_j(x) = \frac {(-1)^j} {j!} \, \eee^{\frac 12 x^2} B_j(D_1,\ldots,D_j) \eee^{-\frac 12 x^2},
\end{equation}
where $B_j$ is the $j$-th Bell polynomial (see Remark~\ref{rem:bell_poly_def}) and $D_1,D_2,\ldots$ are differential operators (with random coefficients) given by
\begin{equation}\label{eq:D_j_one_split}
D_j
=
\frac{\varphi^{(j+2)}(\beta)}{(j+1)(j+2)} \left(\frac{1}{\sigma(\beta)} \frac{\dd}{\dd x} \right)^{j+2} + \chi_j(\beta) \left(\frac{1}{\sigma(\beta)}\frac{\dd}{\dd x}\right)^{j}
\end{equation}
with $\chi_j(\beta)$ as in Theorem~\ref{theo:chi_one_split_BRW}.
\end{theorem}
\begin{remark}
The expressions for the first three terms in the expansion have the same form as in~\eqref{eq:G0}, \eqref{eq:G1}, \eqref{eq:G2}.
\end{remark}

\begin{remark}
An Edgeworth expansion for the profile of a \textit{many-split} BRW was obtained in~\cite{gruebel_kabluchko_BRW}. Theorem~\ref{theo:edgeworth_general} from the present paper can be applied to the many-split BRW, but both the representation of the terms of the expansion and the proof given in~\cite{gruebel_kabluchko_BRW} differ from ours. In Section~\ref{subsec:alternative_expression} we provide an alternative representation for the terms in Theorem~\ref{theo:edgeworth_general} which allows to derive the many-split BRW expansion of~\cite{gruebel_kabluchko_BRW}. There are many differences between the one-split and many-split BRW cases. For example, in the former case the expansions are in negative powers of $\sqrt{\log n}$, whereas in the latter case negative powers of $\sqrt n$ appear.
\end{remark}

Taking $\beta=0$ and $r=0$ in Theorem~\ref{theo:edgeworth_one_split_BRW}, and recalling that $W_\infty(0)=m(0)$ and $\varphi(0)=1$, we obtain the following local limit theorem for the one-split BRW.
\begin{theorem}
Let $(\bL_n(k))_{k\in\Z}$ be the profile at time $n$ of a one-split branching random walk satisfying Assumptions B1--B5. Then,
\begin{equation}
\sqrt {\log n}\,
\sup_{k\in\Z}
\left|\frac{\bL_n(k)}{n} - \frac{m(0)}{\sigma(0) \sqrt{2\pi \log n}} \exp\left\{-\frac 12\left(\frac{k-\mu(0) \log n}{ \sigma(0) \sqrt{\log n}} \right)^2\right\}\right| \toas 0.
\end{equation}
\end{theorem}
More terms can be obtained by taking $\beta=0$ and arbitrary $r\in\N_0$. Another possibility is to take $\beta=\beta_{n}(k)$ as in~\eqref{def:beta_n_k}, that is $\varphi'(\beta_{n}(k))= k/\log n$. Then, $x_{n}(k) = 0$ and we obtain the following expansion containing only half-integer powers of $\log n$ (c.f.\ Theorem~\ref{theo:LD_general}):
\begin{theorem}\label{theo:LD_one_split}
Let $(\bL_n(k))_{k\in \Z}$ be the profile at time $n$ of a one-split branching random walk satisfying Assumptions B1--B5. Then, for all $r\in\N_0$ and any compact set $K\subset (\beta_-,\beta_+)$,
\begin{equation*}
(\log n)^{r+1} \sup_{k\in \Z \cap (\log n) \varphi'(K)}
\left|\frac{\eee^{k\beta_n(k)}}{n^{\varphi(\beta_n(k))}}   \LLL_n(k)  - \frac{1}{\sqrt{2\pi \log n}} \sum_{j=0}^{r} \frac{F_{2j}(0; \beta_n(k))}{\sigma(\beta_n(k))(\log n)^{j}}\right| \toas  0,
\end{equation*}
where $F_{2j}(0; \beta) := W_\infty(\beta) G_{2j}(0;\beta)$ is a linear combination of $1,W_\infty(\beta),\ldots, W_\infty^{(2j)}(\beta)$ (see Section~\ref{subsec:alternative_expression} for the proof of the latter claim).
\end{theorem}

Our results easily imply an expansion similar to Theorem~\ref{theo:edgeworth_one_split_BRW} for the \emph{mean} of the profile when the number of particles in the first generation is almost surely constant.

\begin{theorem}\label{theo:edgeworth_one_split_BRW_mean}
Let $(\bL_n(k))_{k\in\Z}$ be the profile at time $n$ of a one-split branching random walk with the deterministic number of descendants and satisfying Assumptions B1--B5. Fix $r\in\N_0$ and a compact set $K\subset (\beta_-,\beta_+)$. Then,
\begin{equation*}
(\log n)^{\frac{r+1}{2}}
\sup_{k\in\Z} \sup_{\beta\in K}
\left|\frac{\eee^{\beta k}\E[\bL_n(k)]}{n^{\varphi(\beta)}} -  \frac{\E [W_\infty(\beta)]\, \eee^{-\frac 12 x_n^2(k)}}{\sigma(\beta) \sqrt{2\pi \log n}}  \sum_{j=0}^r \frac{\tilde G_j(x_n(k);\beta)}{(\log n)^{j/2}} \right| \ton 0.
\end{equation*}
Here, $\tilde G_j(x;\beta)$ is defined by the same formulas~\eqref{eq:G_def_one_split}, \eqref{eq:D_j_one_split} as $G_j(x;\beta)$, but with $\chi_j(\beta)$ replaced by its deterministic analogue
$$
\tilde \chi_j(\beta) = \left(\frac{\dd}{\dd \beta}\right)^j \log \E W_\infty(\beta)
=
- \left(\frac{\dd}{\dd \beta}\right)^j \log \Gamma\left(\frac{m(\beta) + 1}{m(0)}\right).
$$
\end{theorem}

Again, it is natural to choose $\beta$  as in~\eqref{def:beta_n_k}. Then, $x_n(k) = 0$ and one obtains an expansion containing half-integer powers of $\log n$ only.

\subsection{Width and mode of the one-split BRW}
Recall the definitions of the \textit{width} $M_n$ and the \textit{mode} $u_n$ of a one-split BRW at time $n$ in \eqref{eq:width_mode_def}.
In the setting of random trees, the mode is the level having the largest number of nodes, while the width is the maximal number of nodes at a level. From Theorems~\ref{theo:mode}, \ref{theo:width} and \ref{theo:width_limit_law} we obtain the following results for the one-split BRW.
\begin{theorem}\label{theo:mode_one_split}
Consider a one-split BRW satisfying Assumptions B1--B5. There is an a.s.\ finite random variable $K$ such that for $n>K$, the mode $u_n$ is equal to one of the numbers $\lfloor u_n^*\rfloor$ or $\lceil u_n^*\rceil$, where
\begin{equation}\label{eq:u_n_star_one_split}
u_n^* = \varphi'(0) \log n + \chi_1(0) - \frac{\kappa_3(0)}{2\sigma^2(0)}.
\end{equation}
\end{theorem}

\begin{remark}
In fact, one can provide more information on which of the two values, $\lfloor u_n^*\rfloor$ or $\lceil u_n^*\rceil$, is the mode. Let $\nint(u_n^*) = \argmin_{k\in \Z}|u_n^{\ast}-k|$ be the integer closest to $u_n^*$ with convention that $\nint(u_n^*)=\lfloor u_n^{\ast}\rfloor$ if $u_n^{\ast}$ is a half-integer. The proof of Theorem~\ref{theo:mode}, see formula~\eqref{eq:L_n_k_plus_1_L_n_k} below, shows that, for every $\eps>0$, we can find an a.s.\ finite random variable $K(\eps)$ such that for all $n>K(\eps)$ satisfying $\min_{k\in\Z} |u_n^*-k-\frac 12|>\eps$, the mode $u_n$ is unique and equals $\nint(u_n^*)$.

\vspace*{1mm}
\noindent
\textit{Case 1:} $\varphi'(0)=0$ (meaning that the one-split BRW has no drift, which applies to Example (vi) of Section~\ref{subsec:random_trees_one_split}). If the random variable $\chi_1(0)$ has no atoms, then $\chi_1(0) - \frac 12 \kappa_3(0)/\sigma^2(0)$ is not a half-integer with probability $1$. It follows that there is an a.s.\ finite random variable $K_1$ such that
$$
u_n = \nint\left(\chi_1(0) - \frac{\kappa_3(0)}{2\sigma^2(0)}\right)
$$
for all $n>K_1$. Absolute continuity of $\chi_1(0)$ in Example (vi) follows from the fixed point equation derived in~\cite{kubapanhol}.

\vspace*{1mm}
\noindent
\textit{Case 2:} $\varphi'(0)\neq 0$ (which is true in examples (i)--(v) of Section~\ref{subsec:random_trees_one_split}). The arithmetic properties of the sequence $(\{\varphi^{\prime}(0)\log n\})_{n\in\N}$, with $\{\cdot\}$ denoting the fractional part, allow us to deduce an additional information compared to the general result given by Theorem~\ref{theo:mode}. Here, we say that a set $A \subset \N$ has \emph{asymptotic density} $\alpha \in [0,1]$ if
\begin{equation}\label{def:asymp_freq}
\lim_{n\to\infty} \frac {\# (A \cap \{1, \ldots, n\})}{n} = \alpha.
\end{equation}
\end{remark}

\begin{proposition}\label{cor:one_split_mode}
Consider a one-split BRW satisfying Assumptions B1--B5 with $\varphi'(0) \neq 0$.  Then, almost surely,

\vspace*{2mm}
\noindent
\emph{(i)} there are arbitrary long intervals of consecutive $n$'s for which $u_n$ is unique and $u_n = \lfloor u_n^*\rfloor$; and, similarly, arbitrary long intervals for which $u_n$ is unique and $u_n = \lceil u_n^*\rceil$;

\vspace*{2mm}
\noindent
\emph{(ii)} the asymptotic density of the set $A = \{n\in\N: u_n \ \text{is unique and} \ u_n = \nint(u_n^*)\}$ equals one.
\end{proposition}

The next two results on the width $M_n$ are special cases of Theorems~\ref{theo:width} and~\ref{theo:width_limit_law}.
\begin{theorem}\label{theo:width_one_split}
Consider a one-split BRW satisfying Assumptions B1--B5.  Then, the width $M_n$ satisfies
\begin{equation}
\frac{\sqrt{2\pi \log n} \, \sigma(0)\, M_n}{m(0)\, n}   \toas 1.
\end{equation}
\end{theorem}

\begin{theorem}\label{theo:width_limit_law_one_split_BRW}
Consider a one-split BRW satisfying Assumptions B1--B5. Let
$$
\tilde M_n
:=
2\sigma^2(0) \log n \left( 1 - \frac{\sqrt{2\pi \log n}\, \sigma(0)\, M_n}{m(0)\,n} \right).
$$
If $\theta_n := \min_{k\in\Z} |u_n^*-k|$ denotes the distance between $u_n^*$ and the nearest integer, then
$$
\tilde M_n - \theta_n^2 \toas \chi_2(0) + \frac{\kappa_3^2(0)}{6\sigma^4(0)} - \frac{\kappa_4(0)}{4\sigma^2(0)}.
$$
\end{theorem}
\begin{remark}
Again, the details depend on whether $\varphi'(0)$ vanishes or not.

\vspace*{1mm}
\noindent
\textit{Case 1:} Suppose that $\varphi'(0)=0$. 
Then, $\theta_n$ does not depend on $n$  and we obtain
$$
\tilde M_n \toas \chi_2(0) + \frac{\kappa_3^2(0)}{6\sigma^4(0)} - \frac{\kappa_4(0)}{4\sigma^2(0)} + \min_{k\in\Z}\left|\chi_1(0)- \frac{\kappa_3(0)}{2\sigma^2(0)}-k\right|^2.
$$

\vspace*{1mm}
\noindent
\textit{Case 2:} $\varphi'(0)\neq 0$. The sequence $(\{\varphi^{\prime}(0)\log n\})_{n\in\N}$ is dense in $[0,1]$. This implies that the set of subsequential limits of the sequence $\theta_n$ is equal to $[0,1/2]$. It follows that
\begin{align*}
\liminf_{n\to\infty} \tilde M_n &= \chi_2(0) + \frac{\kappa_3^2(0)}{6\sigma^4(0)} - \frac{\kappa_4(0)}{4\sigma^2(0)} \quad\text{a.s.},\\
\limsup_{n\to\infty} \tilde M_n &= \chi_2(0) + \frac{\kappa_3^2(0)}{6\sigma^4(0)} - \frac{\kappa_4(0)}{4\sigma^2(0)} + \frac 14
\quad\text{a.s.},
\end{align*}
and every point between the $\liminf$ and $\limsup$ is a.s.\ a subsequential limit of $\tilde M_n$. Thus, we have infinitely many different a.s.\ (and hence, weak) subsequential limits of $\tilde M_n$. If $\chi_2(0)$ is non-degenerate, it follows from the convergence of types lemma that the random variable $M_n$ cannot be normalized by an affine transformation to converge (in the weak sense) to a non-degenerate limit law.  This agrees with the observation of \citet[Theorem 2]{fuchs_hwang_neininger}.
\end{remark}

\subsection{Occupation numbers in the one-split BRW}\label{subsec:occupation_one_split_BRW}
Consider a one-split BRW with profiles $\bL_1,\bL_2,\ldots$. In this section we shall state limit theorems on the ``occupation numbers'' $\bL_n(k_n)$, where $k_n$ is a (deterministic) integer sequence  with some regular type of behavior. These limit theorems can be applied to random trees (including BSTs, RRTs and PORTs; see Section~\ref{subsec:random_trees_one_split}) and improve on the results of~\citet{fuchs_hwang_neininger}. In these applications $\bL_n(k_n)$ is interpreted as the number of nodes at depth $k_n$ in a random tree. To prove these theorems, we shall use a suitable number of terms in the Edgeworth expansion of $\bL_n$ stated in Theorem~\ref{theo:edgeworth_one_split_BRW}. Our aim is to find a non-degenerate limit distribution for $\bL_n(k_n)$, but it turns out that our results hold even in the sense of a.s.\ convergence.
As in~\eqref{def:beta_n_k}, define $\beta_n=\beta_n(k)$ to be the solution of
\begin{align} \label{def:beta_n_k_repeat}
\varphi'(\beta_{n}(k)) = \frac{k}{\log n},\quad \frac{k}{\log n}\in \varphi'((\beta_{-},\beta_{+})).
\end{align}

\begin{theorem}\label{theo:L_n_lim_distr_beta_neq_0}
Consider a one-split BRW satisfying Assumptions B1--B5.
Let $k_n$ be an integer  sequence such that, for some $\beta\in (\beta_-,\beta_+)$, we have $k_n = \varphi'(\beta) \log n + o(\log n)$. Then, with $\beta_n = \beta_n(k_n)$ as in~\eqref{def:beta_n_k_repeat}, we have
\begin{equation}\label{eq:L_n_lim_distr_beta_neq_0_a}
\frac{\sqrt{\log n}}{n^{\varphi(\beta_n) - \beta_n \varphi'(\beta_n)}} \, \bL_n(k_n)
\toas
\frac{W_\infty(\beta)}{\sqrt{2\pi}\sigma(\beta)}.
\end{equation}
If, for some $\alpha\in\R$,
\begin{equation}\label{eq:k_n_cond1}
k_n = \varphi'(\beta) \log n + \alpha \sigma(\beta) \sqrt {\log n} + o(\sqrt {\log n}), \quad n\to\infty,
\end{equation}
then
\begin{equation}\label{eq:L_n_lim_distr_beta_neq_0}
\frac{\sqrt{\log n}}{n^{\varphi(\beta)}} \, \eee^{\beta k_n} \bL_n(k_n)
\toas
\frac{\eee^{-\frac {1}{2}\alpha^2 }}{\sqrt{2\pi}\sigma(\beta)}  W_\infty(\beta).
\end{equation}
\end{theorem}
\begin{proof}[Proof of Theorem~\ref{theo:L_n_lim_distr_beta_neq_0}]
From Theorem~\ref{theo:edgeworth_one_split_BRW} with $r=0$ and $K \subseteq (\beta_-, \beta_+)$ compact, we have
\begin{equation}\label{eq:tech1}
\sqrt{\log n}\, \sup_{\beta' \in K} \sup_{k\in\Z}\left|\frac{\eee^{\beta' k}\bL_n(k)}{n^{\varphi(\beta')}} - \frac{W_\infty(\beta') \eee^{-\frac 12 \left(\frac{k-\varphi'(\beta')\log n}{\sigma(\beta')\sqrt{\log n}}\right)^2}}{\sigma(\beta') \sqrt{2\pi \log n}} \right| \toas 0.
\end{equation}
From here, \eqref{eq:L_n_lim_distr_beta_neq_0_a} follows readily upon taking $k=k_n$, $\beta'=\beta_n(k_n)$ as in~\eqref{def:beta_n_k_repeat} (which converges to $\beta$, as $n\to\infty$), recalling the continuity of $W_\infty$ and $\sigma$ and noting that the term in the exponent vanishes.
Formula~\eqref{eq:L_n_lim_distr_beta_neq_0} follows from~\eqref{eq:k_n_cond1} upon choosing $\beta' = \beta$ in~\eqref{eq:tech1} and using the observation
$$
x_n(k_n) = \frac{k_n - \varphi'(\beta)\log n}{\sigma(\beta) \sqrt{\log n}} \ton \alpha.
$$
The proof is complete.
\end{proof}

\begin{remark}
If, in addition to the conditions stated in the theorem, we assume that the BRW has a deterministic number of descendants, then the convergences~\eqref{eq:L_n_lim_distr_beta_neq_0_a} and~\eqref{eq:L_n_lim_distr_beta_neq_0} also hold in $L_1$ sense.
\end{remark}

Theorem~\ref{theo:L_n_lim_distr_beta_neq_0} is applicable in the case $\beta=0$, however, $W_\infty(0)=m(0)$ is a.s.\ constant (see Remark~\ref{rem:W_infty_0}) meaning that the limits in~\eqref{eq:L_n_lim_distr_beta_neq_0_a} and~\eqref{eq:L_n_lim_distr_beta_neq_0} are degenerate. It is therefore natural to ask whether non-degenerate limits can be obtained by choosing a more refined normalization of $\bL_n(k_n)$. Denote by $\bL_n^\circ(k)$ the profile centered by its expectation:
$$
\bL_n^\circ(k) = \bL_n(k) - \E [\bL_n(k)], \quad k\in\Z.
$$
In the following we assume that the integer sequence $k_n$ can be represented in the form
$$
k_n=\varphi'(0) \log n + c_n,
$$
where $c_n$ is a sequence on which we impose various growth conditions.
While Theorem~\ref{theo:L_n_lim_distr_beta_neq_0} can be derived from the first term in the Edgeworth expansion  (meaning that $r=0$), the following more refined theorem requires more terms (meaning that $r=1$ or $r=2$).


\begin{theorem}\label{theo:L_n_lim_distr_all_together}
Consider a one-split BRW with deterministic number of descendants and satisfying Assumptions B1--B5. Let $(k_n)_{n\in\N}$ be an integer sequence. 

\vspace*{2mm}
\noindent
\emph{(a)} If $k_n = \varphi'(0) \log n + \alpha \sigma(0) \sqrt {\log n} + o(\sqrt {\log n})$ for some $\alpha\in\R$, then
\begin{equation}\label{eq:L_n_lim_distr_2}
\frac{\log n}{n} \, \bL_n^\circ(k_n)
\toas
\frac{ m(0) \alpha \eee^{-\frac {1}{2}\alpha^2 }}{\sqrt{2\pi}\sigma^2(0)}  \left(\chi_1(0) - \E \chi_1(0)\right).
\end{equation}

\vspace*{2mm}
\noindent
\emph{(b)}
If $k_n = \varphi'(0) \log n + c_n$ with $\lim_{n\to\infty} |c_n| = \infty$ and $c_n=o(\log n)$, then, with $\beta_n$ as in~\eqref{def:beta_n_k_repeat},
\begin{equation}\label{eq:L_n_lim_distr_2_a}
\frac{(\log n)^{\frac 3 2}}{c_n n^{\varphi(\beta_n) - \beta_n \varphi'(\beta_n)}} \, \bL_n^\circ(k_n)
\toas
\frac{ m(0) \left(\chi_1(0) - \E \chi_1(0)\right) }{\sqrt{2\pi}\sigma^3(0)}.
\end{equation}
In particular, if $\lim_{n\to\infty} |c_n| = \infty$ and $c_n=o(\sqrt {\log n})$, then
\begin{equation}\label{eq:xia_chen_general1}
\frac{(\log n)^{\frac 32}}{n c_n} \bL_n^{\circ}(k_n) \toas
\frac {m(0)(\chi_1(0) - \E \chi_1(0))}{\sqrt{2\pi} \sigma^3(0)}.
\end{equation}

\vspace*{2mm}
\noindent
\emph{(c)}
If $k_n = \varphi'(0) \log n + c_n$  where $c_n$ is bounded, then
\begin{equation}\label{eq:xia_chen_general}
\frac {(\log n)^{\frac 32}}{n} \bL_n^{\circ}(k_n) - R^\circ(c_n) \toas 0,
\end{equation}
where $R^{\circ}(c)= R(c) - \E R(c)$ and $R(c)$ is a random variable given by
\begin{equation}
R(c) := \frac {m(0)} {\sqrt{2\pi}\sigma^3(0)} \left( \chi_1(0) \left( c + \frac{\kappa_3(0)}{2\sigma^2(0)}\right) - \frac {\chi_1^2(0) + \chi_2(0)}2 \right), \quad c\in\R.
\end{equation}
\end{theorem}

\begin{proof}[Proof of Theorem~\ref{theo:L_n_lim_distr_all_together}]
Taking $\beta=0$ and $r=1$ in Theorem~\ref{theo:edgeworth_one_split_BRW} and using formula~\eqref{eq:G1}, we obtain
$$
\frac{\bL_n(k_n)}{n}
=
\frac{m(0) \eee^{-\frac 12 x_n^2(k_n)}}{\sigma(0) \sqrt{2\pi \log n}} \left(1 + \frac {\frac{\chi_1(0)}{\sigma(0)}x_n(k_n) + \frac{\kappa_3(0)}{6\sigma^3(0)}\Herm_3(x_n(k_n))}{\sqrt {\log n}}\right)
+o\left(\frac 1 {\log n}\right)
$$
almost surely. By Theorem~\ref{theo:edgeworth_one_split_BRW_mean}, we have an analogous expansion for the expectation of $\LLL_n(k_n)$. Subtracting both expansions, we obtain
\begin{equation}\label{eq:tech2}
\frac{\bL_n^\circ(k_n) }{n}
=
\frac{m(0) \eee^{-\frac 12 x_n^2(k_n)}}{\sigma(0) \sqrt{2\pi \log n}}
\cdot
\frac{\chi_1(0) - \E \chi_1(0)}{\sigma(0)\sqrt {\log n}} x_n(k_n) + o\left(\frac 1 {\log n}\right) \;\;\text{a.s.},
\end{equation}
where
$$
x_n(k_n) = \frac{k_n - \varphi'(0)\log n}{\sigma(0) \sqrt{\log n}} = \frac{c_n} {\sigma(0)\sqrt{\log n}}.
$$
To prove~\eqref{eq:L_n_lim_distr_2}, it is enough to notice that $\lim_{n\to\infty} x_n(k_n) = \alpha$. Inserting this into~\eqref{eq:tech2}, we obtain~\eqref{eq:L_n_lim_distr_2}.

For the remaining results, we need to apply the Edgeworth expansion with $r = 2$. First, choosing $\beta = 0$ in Theorems~\ref{theo:edgeworth_one_split_BRW}, \ref{theo:edgeworth_one_split_BRW_mean},  using~\eqref{eq:G1}, \eqref{eq:G2}, and subtracting the expansions for $\LLL_n(k_n)$ and $\E [\LLL_n(k_n)]$, we obtain
\begin{multline*}
\frac 1n \bL_n^{\circ}(k_n)
=
\frac{m(0)\eee^{-\frac 12 x_n^2(k_n)}}{\sigma(0) \sqrt{2\pi \log n}}
\Biggl(
\frac {\chi_{1}(0) -\E \chi_{1}(0)} {\sigma(0)\sqrt {\log n}} x_n(k_n)
\\+
\kappa_3(0)\frac{\chi_{1}(0) - \E \chi_{1}(0)}{6 \sigma^4(0)\log n} \Herm_4(x_n(k_n))
\\+
\frac{\chi_{1}^2(0) + \chi_{2}(0) -\E[\chi_{1}^{2}(0) + \chi_{2}(0)]}{2 \sigma^2(0)\log n} \Herm_2(x_n(k_n))
\Biggr)
+ o\left(\frac 1 {(\log n)^{\frac 32}}\right)\;\; \text{a.s.}
\end{multline*}
Multiplying both sides of the last display by $(\log n)^{3/2} / c_n$ yields~\eqref{eq:xia_chen_general1} because $\lim_{n\to\infty} x_n(k_n) = 0$ and $\Herm_2(x) = -1 +o(1), \Herm_4(x) = 3 + o(1)$ as $x \to 0$. For the proof of~\eqref{eq:xia_chen_general} use the same expansion as above and note that $x_n(k_n)=O(1/\sqrt {\log n})$.

It remains to show that~\eqref{eq:L_n_lim_distr_2_a} holds. Here, we use the Edgeworth expansion with $r = 2$ and $\beta_n$ as in~\eqref{def:beta_n_k_repeat}. First, note that, by a simple Taylor expansion, we have
\begin{align} \label{asy:bet}
\beta_n = \frac{c_n(1 + o(1)) }{\sigma^2(0) \log n} .
\end{align}
Next, by Theorem~\ref{theo:edgeworth_one_split_BRW}  and~\eqref{eq:G0}, \eqref{eq:G2},
\begin{multline*}
\frac {\eee^{\beta_n k_n}}{n^{\varphi(\beta_n)}} \bL_n(k_n)
=
\frac{1}{\sigma(\beta_n) \sqrt{2\pi \log n}}
\Biggl(
W_\infty(\beta_n)
+
\frac{3}{\log n} \left( \frac{\kappa_4(\beta_n)}{24 \sigma^4(\beta_n)} + \frac{\kappa_3(\beta_n) W_\infty'(\beta_n)}{6 \sigma^4(\beta_n)} \right) \\
- \frac{W_\infty''(\beta_n)}{2 \sigma^2(\beta_n)\log n}
-   \frac{15 \kappa_3^2(\beta_n)}{72 \sigma^6(\beta_n) \log n}
\Biggr) + o\left(\frac 1 {(\log n)^{\frac 32}}\right) \;\;\text{a.s.},
\end{multline*}
and similarly, by Theorem~\ref{theo:edgeworth_one_split_BRW_mean} with $\beta=\beta_n$,
\begin{multline*}
\frac {\eee^{\beta_n k_n}}{n^{\varphi(\beta_n)}} \E [\bL_n(k_n)]
=
\frac{1}{\sigma(\beta_n) \sqrt{2\pi \log n}}
\Biggl(
\E W_\infty(\beta_n)
+
\frac{3}{\log n} \left( \frac{\kappa_4(\beta_n)}{24 \sigma^4(\beta_n)} + \frac{\kappa_3(\beta_n) \E W_\infty'(\beta_n)}{6 \sigma^4(\beta_n)} \right) \\
- \frac{\E W_\infty''(\beta_n)}{2 \sigma^2(\beta_n)\log n}
-  \frac{15 \kappa_3^2(\beta_n)}{72 \sigma^6(\beta_n) \log n}
\Biggr) + o\left(\frac 1 {(\log n)^{\frac 32}}\right).
\end{multline*}
Since $|c_n| \to \infty$, taking the difference of both expansions yields
\begin{align} \label{conv5}
\frac{(\log n)^{3/2}}{c_n} \frac {\eee^{\beta_n k_n}}{n^{\varphi(\beta_n)}} \bL^\circ_n(k_n) = \frac{(W_\infty(\beta_n) - \E{W_\infty(\beta_n)}) \log n}{c_n \sigma(\beta_n) \sqrt{2\pi}} + o(1) \;\;\text{a.s.}
\end{align}
Since $\beta_n \to 0$, we have, almost surely, $W_\infty(\beta_n) = m(0) + W_\infty'(0) \beta_n + o(\beta_n)$. Since $\E W_\infty(\beta)$ is analytic in a neighbourhood of $\beta = 0$, the analogous expansion holds for the mean. The assertion now follows
from~\eqref{conv5} together with~\eqref{asy:bet}. (Note that, even though the higher order terms in the Edgeworth expansion appearing in the proof are asymptotically irrelevant, we cannot obtain the result using the expansion for $r = 1$).
\end{proof}


\begin{remark}
In the setting of Theorem~\ref{theo:L_n_lim_distr_all_together}, part~(c), the $\limsup$ and the $\liminf$ of the sequence $\frac 1n (\log n)^{3/2} \bL_n^{\circ}(k_n)$
are a.s.\ finite (but not necessarily equal to each other). Whether or not this sequence has an a.s.\ limit depends on the value of $\varphi'(0)$.

\vspace*{1mm}
\noindent
\textit{Case 1:} $\varphi'(0)=0$ (which applies to Example (vi) of Section~\ref{subsec:random_trees_one_split}). It is natural to take $k_n = a \in \Z$.  Then, $c_n=a$ and we obtain
$$
\frac {(\log n)^{\frac 32}}{n} \bL_n^{\circ}(k_n) \toas  R^\circ(a).
$$

\vspace*{1mm}
\noindent
\textit{Case 2:} $\varphi'(0)\neq 0$ (which applies to Examples (i)--(v) of Section~\ref{subsec:random_trees_one_split}).
It is natural to take $k_n=\lfloor\varphi'(0)\log n\rfloor + a$, where $a\in\Z$, which means that $c_n = a - \{\varphi'(0) \log n\}$. The set of accumulation points of the sequence $c_n$ is the interval $[a-1,a]$.  Hence, we can parametrize the set of all a.s.\  subsequential limits of $\frac 1n (\log n)^{3/2} \bL_n^{\circ}(k_n)$ as follows:
\begin{equation}\label{eq:limiting_distr2}
\{R^\circ(a-z) \colon z\in [0,1] \}.
\end{equation}
\end{remark}

\begin{remark}
Also we point out that in Theorem~\ref{theo:L_n_lim_distr_all_together} the assumption that a BRW has deterministic number of descendants is used only to derive the Edgeworth expansion for the centering $\E[\bL_n(k)]$ by using Theorem~\ref{theo:edgeworth_one_split_BRW_mean}. If such an expansion holds {\it a priori}, the results of the above theorems remain valid without this constraint.
\end{remark}

\subsection{Profile of binary search trees around level \texorpdfstring{$\log n$}{log n}}\label{subsec:profile_BST_log_n}

Applying the results of Section~\ref{subsec:occupation_one_split_BRW} to the special case of BSTs we obtain Equations~\eqref{eq:BST_near_beta_log_n1}, \eqref{eq:BST_near_beta_log_n2} and  Theorem~\ref{theo:BST_occupation_numbers} stated in the introduction. In Equations~\eqref{eq:BST_near_beta_log_n1}, \eqref{eq:BST_near_beta_log_n2} (which deal with levels near $2\eee^\beta \log n$, $\beta\in\R$), the limit random variable is a multiple of $W_\infty(\beta)$. For $\beta=0$, the limit $W_\infty(0)=1$ is degenerate, and
we collected more precise results describing the behavior of the profile around level $\varphi'(0) \log n = 2\log n$ in Theorem~\ref{theo:BST_occupation_numbers}.

However, there is one more value of $\beta$ for which $W_\infty(\beta)$ is degenerate, namely $\beta=-\log 2 \approx -0.693$.  By construction of the BSTs, we have $W_n(-\log 2) = 1 = W_\infty(-\log 2)$ for all $n\in\N$. The value $\beta=-\log 2$ corresponds to the behavior of the BST profile around level $\varphi'(-\log 2) \log n = \log n$. We conclude this section with a discussion of this case.
Similarly to Theorem~\ref{theo:L_n_lim_distr_all_together}, \citet[Theorem 6]{fuchs_hwang_neininger} showed that the scaling behaviour of $\LLL_n(k_n)$ with $k_n = \log n + c_n$ depends drastically on whether $|c_n| \to \infty$ or $c_n = O(1)$.
The next theorem is proved along the same lines as Theorem~\ref{theo:L_n_lim_distr_all_together}.

\begin{theorem}\label{theo:BST_occupation_numbers_logn_case}
Let $(\LLL_n(k))_{k\in\Z}$ be the profile of a random binary search tree with $n+1$ external nodes.  Let $(k_n)_{n\in\N}$ be an integer sequence.

\vspace*{2mm}
\noindent
\emph{(a)}  If $k_n = \log n + \alpha \sqrt{\log n} + o(\sqrt{\log n})$ with $\alpha\in\R$,
then
$$
\frac{\log n}{2^{k_n}} \LLL_n^\circ(k_n) \toas \frac{\chi_1(-\log 2) - \E [\chi_1(-\log 2)]}{\sqrt{2\pi}}\alpha \eee^{-\frac 1 2 \alpha^2}.
$$


\vspace*{2mm}
\noindent
\emph{(b)}
If $k_n = \log n + c_n$, with $\lim_{n\to\infty} |c_n| = \infty$ and $c_n = o(\log n)$,  then, with $\beta_n$ as in~\eqref{def:beta_n_k_repeat},
$$
\frac{(\log n)^{3/2}}{c_n  n^{2 \eee^{\beta_n} (1-\beta_n) -1 }} \LLL_n^\circ(k_n) \toas \frac{\chi_1(-\log 2) - \E [\chi_1(-\log 2)]}{\sqrt{2\pi}}.
$$
In particular, if  $\lim_{n\to\infty} |c_n| = \infty$ but $c_n = o(\sqrt{\log n})$, then
$$
\frac{(\log n)^{3/2}}{c_n  2^{k_n}} \LLL_n^\circ(k_n) \toas \frac{\chi_1(-\log 2) - \E [\chi_1(-\log 2)]}{\sqrt{2\pi}}.
$$

\noindent
\emph{(c)}
If $k_n = \log n + c_n$, where $c_n$ is bounded, then
\begin{equation*}
\frac {(\log n)^{\frac 32}}{2^{k_n}} \bL_n^{\circ}(k_n) - R_*^\circ(c_n) \toas 0,
\end{equation*}
where $R_*^{\circ}(c)= R_*(c) - \E R_*(c)$ and $R_*(c)$ is a random variable given by\
\begin{equation*}
R_*(c) := \frac {1} {\sqrt{2\pi}} \left( \chi_1(-\log 2) \left( c + \frac{1}{2}\right) - \frac {\chi_1^2(-\log 2) + \chi_2(-\log 2)}2 \right),
\quad c\in\R.
\end{equation*}
\end{theorem}

\begin{remark}
The random variable $\chi_1(-\log 2)$ is not almost surely constant: in the space of distributions with zero mean and finite variance, $\chi_1(-\log 2)$ is uniquely characterized by the stochastic fixed-point equation
\begin{align} \label {fix:chi2} \chi_1(-\log 2) \stackrel{d}{=} \frac 12 \chi^{(1)}_1(-\log 2) + \frac 12 \chi^{(2)}_1(-\log 2) + 1+ \frac 12 \left(\log U + \log (1-U)\right), \end{align}
where $\chi^{(1)}_1(-\log 2),  \chi^{(2)}_1(-\log 2)$ are distributional copies of $\chi_1(-\log 2)$, $U$ is uniformly distributed on $[0,1]$, and all three variables are independent.
This follows from the arguments on page 35 in~\cite{fuchs_hwang_neininger}, see also display (35) in~\cite{chauvin_etal} for a less explicit variant of~\eqref{fix:chi2}.
\end{remark}

Let us finally mention that the random variable $W_\infty(\beta)$ is non-degenerate for all $\beta\in (\beta_-,\beta_+)$ except $\beta = 0$ and $\beta = -\log 2$. Indeed, we have the stochastic fixed point equation (see, e.g.,\ \eqref{eq:functional_eq_for_W_infty} below)
$$
\eee^{-\beta} W_\infty(\beta) \eqdistr  U^{2\eee^{\beta} - 1} W_{1,\infty}(\beta) + (1-U)^{2\eee^{\beta} - 1} W_{2,\infty}(\beta),
$$
where $W_{1,\infty}(\beta), W_{2,\infty}(\beta)$ are distributional copies of $W_\infty(\beta)$, $U$ is uniformly distributed on $[0,1]$, and all three variables are independent. A constant random variable $W_\infty(\beta)=c>0$ satisfies this equation if and only if  $2\eee^{\beta}-1 \in \{0,1\}$.  This corresponds to $\beta \in \{0, -\log 2\}$.

\section{Proof of the general Edgeworth expansion}\label{sec:proof_general_edgeworth}
\subsection{Proof of Theorem~\ref{theo:edgeworth_general}}
The proof is based on studying the characteristic function of the profile.
For notational reasons, we shall use $\mu$ and $\sigma^2$ as shorthands for $\mu(\beta)$ and $\sigma(\beta)$.  Consider the following signed measure on $\R$: for $\beta \in (\beta_-, \beta_+)$,
\begin{equation}\label{eq:mu_n_def}
\mu_n := \mu_n(\beta) := \sum_{k\in\Z}  \eee^{\beta k - \varphi(\beta) w_n}\bL_n(k) \delta\left(\frac{k-\mu w_n}{\sigma \sqrt{w_n}}\right).
\end{equation}
Here, $\delta(z)$ is the Dirac delta-measure at $z\in\R$. The characteristic function of $\mu_n$ has the form
\begin{equation}\label{eq:psi_n_def}
\psi_n(s)
:= \psi_n(s, \beta) := \int_\R \eee^{i s t} \mu_n(\dd t)
=
\eee^{-\varphi(\beta) w_n - is \frac{\mu w_n}{\sigma \sqrt{w_n}}} \sum_{k\in\Z} \bL_n(k) \eee^{k\left(\beta  +   \frac{is}{\sigma \sqrt{w_n}}\right)}.
\end{equation}

Fix some $\beta_0 \in (\beta_-,\beta_+)$ and \emph{random} $\varepsilon_0 > 0$ such that
$\bD_{3\varepsilon_0}(\beta_0) \subset \fD$ and $W_\infty$ is non-zero on $\bD_{3\varepsilon_0}(\beta_0)$. Here, $\bD_r(\beta_0) =\{z\in \C\colon |z-\beta_0|<r\}$ denotes an open disk with radius $r$ centered at $\beta_0$.
For any $\beta \in \fI_0 := (\beta_0 - \varepsilon_0, \beta_0 + \varepsilon_0)$, we have $\bD_{2\varepsilon_0}(\beta) \subset \fD$. In the following, all estimates are going to be uniform in $\beta\in \fI_0$. Since any compact set $K\subset (\beta_-,\beta_+)$ can be covered by finitely many such intervals $\fI_0$, the uniformity in $\beta\in K$ follows.
After recalling the definition of $W_n$, see~\eqref{eq:biggins_def}, we obtain that,  for all $ \beta \in  \fI_0$,  as long as the variable $s\in \R$ satisfies
$$
\left|\frac{s}{\sigma \sqrt{w_n}}\right|<\eps_0,
$$
the function $\psi_n$ is well-defined and can be written in the form
\begin{equation}\label{eq:psi_n_def1}
\psi_n(s) = \eee^{-\varphi(\beta) w_n - is \frac{\mu w_n}{\sigma \sqrt{w_n}} + w_n \varphi\left(\beta + \frac{is}{\sigma \sqrt{w_n}}\right)} W_n\left(\beta + \frac{is}{\sigma \sqrt{w_n}}\right).
\end{equation}

Our aim is to derive an asymptotic expansion of $\psi_n(s)$ in powers of $w_n^{-1/2}$. Consider a modification of $\psi_n(s)$ in which $W_n$ is replaced by $W_\infty$ and $w_n^{-1/2}$ is replaced by a new variable $u$.  For any fixed $s\in\R$ and $ \beta \in  \fI_0$, the function
\begin{equation}\label{eq:psi_tilde_def1}
\tilde \psi(s; u)
=
\exp\left\{
- \frac{\varphi(\beta)}{u^{2}} - is \frac{\mu}{\sigma u} + \frac 1{u^{2}} \varphi\left(\beta + \frac{isu}{\sigma}\right)
+ \log W_\infty \left(\beta + \frac{is u}{\sigma}\right)
\right\}
\end{equation}
is well-defined and analytic in $u$ in the disk $|u|<\sigma \eps_0/|s|$. Note that $\log W_\infty$ is defined as an analytic function because $W_\infty$ does not vanish on $\bD_{3\varepsilon_0}(\beta_0)$. Thus, as long as $\left|\frac{su}{\sigma}\right|<\eps_0$,
$$
\log\tilde \psi(s;u) = \sum_{k=0}^{\infty} \frac{a_k(s)}{k!} u^k,
$$
where
\begin{equation}\label{eq:a_k_s}
a_k(s) = a_k(s, \beta) :=  \frac{\varphi^{(k+2)}(\beta)}{(k+2)(k+1)}\left(\frac{is}{\sigma}\right)^{k+2} +\chi_k(\beta)\left(\frac{is}{\sigma}\right)^k.
\end{equation}
Recall from the definition of Bell polynomials, see~\eqref{eq:bell_poly_def1}, that there is a formal identity
$$
\exp\left\{ \sum_{k=1}^{\infty} \frac{a_k}{k!} x^k\right\}
=
\sum_{k=0}^{\infty} \frac{B_k(a_1,\ldots,a_k)}{k!} x^k.
$$
It follows that the following holds (not only formally!) for $|u| < \eps_0 \sigma / |s|$:
\begin{equation}\label{eq:psi_tilde_def2}
\tilde \psi (s; u) = W_\infty(\beta) \eee^{-\frac{s^2}{2}} \sum_{k=0}^{\infty} \frac{B_k(a_1(s),\ldots, a_k(s))}{k!} u^k.
\end{equation}
To see that~\eqref{eq:psi_tilde_def2} holds not only formally, note that $\tilde \psi (s; u)$, being an analytic function of $u$ in the disk $|u| < \eps_0 \sigma / |s|$, has a convergent Taylor expansion. But in order to compute the coefficients of this expansion, we can use formal series.
We shall need a uniform estimate for the remainder term in~\eqref{eq:psi_tilde_def2}.
\begin{lemma}
Recall that $a_k(s)$ is given by~\eqref{eq:a_k_s}. There exists an a.s.\ finite random variable $M >0$ such that, for all $\beta \in   \fI_0$,
$$
\left|\frac{a_k(s)}{k!}\right| \leq M^{k} (|s|+1)^{k+2}
$$
for all  $s\in\R$ and $k\in\N$.
\end{lemma}
\begin{proof}
Since the functions $\varphi$ and $\log W_\infty$ are analytic on the disk $\bD_{2\eps_0}(\beta_0)$, the Cauchy formula implies that, for $\beta \in  \fI_0$,
$$
\left|\frac{\varphi^{(k+2)}(\beta)}{(k+2)!} \right| \leq \sup_{\gamma \in \bD_{\eps_0}(\beta_0)} | \varphi(\gamma)| \varepsilon_0^{-k-2},
\;\;\;
\left|\frac{\chi_k(\beta)}{k!} \right| \leq \sup_{\gamma \in \bD_{\eps_0}(\beta_0)} | \log W_\infty(\gamma)| \varepsilon_0^{-k},
$$
for all $k\in\N$. With $M' = \max \{1, \sup_{\gamma \in \bD_{\eps_0}(\beta_0)} | \varphi(\gamma)| , \sup_{\gamma \in \bD_{\eps_0}(\beta_0)} | \log W_\infty(\gamma)| \},$
and $C = \max(1,\sup_{\gamma \in \bD_{\eps_0}(\beta_0)} (\varepsilon_0 \sigma(\gamma))^{-1}),$
it follows from~\eqref{eq:a_k_s} that
$$
\left|\frac{a_k(s)}{k!}\right|
\leq
\left|\frac s {\sigma}\right|^{k+2} \left|\frac{\varphi^{(k+2)}(\beta)}{(k+2)!} \right|
+ \left|\frac s {\sigma}\right|^{k} \left|\frac{\chi_k(\beta)}{k!} \right|
\leq
M'( |s|^{k+2} C^{k+2} + |s|^k C^k)
$$
which yields the desired estimate choosing $M = M' C^3.$
\end{proof}

\begin{lemma}\label{lem:B_k_estimate}
There is an a.s.\ finite random variable $M_1 > 0$ such that, for all $\beta \in  \fI_0$,
$$
\frac 1 {k!} \left|B_k(a_1(s),\ldots, a_k(s)) \right| \leq M_1^k (|s|+1)^{3k}
$$
for all $k\in\N$ and $s\in\R$.
\end{lemma}
\begin{proof}
By definition of the Bell polynomial $B_k$, see~\eqref{eq:bell_poly_def},
\begin{align*}
\frac 1 {k!} \left|B_k(a_1(s),\ldots, a_k(s))\right|
&\leq
\sum{}^{'} \frac 1{j_1!\ldots j_k!} \left|\frac{a_1(s)}{1!}\right|^{j_1} \ldots \left|\frac{a_k(s)}{k!}\right|^{j_k}\\
&\leq
\sum{}^{'} \frac 1{j_1!\ldots j_k!} M^{1j_1+\ldots+kj_k} (|s|+1)^{\sum_{m=1}^k (m+2)j_m},
\end{align*}
where the sum $\sum{}^{'}$ is taken over all $j_1,\ldots,j_k \in\N_0$ satisfying $1j_1+2j_2+\ldots +k j_k =k$.
Using that $1j_1+\ldots+kj_k=k$ (and consequently $j_1+\ldots+j_k \leq k$) and the inequality $\sum \frac 1{j_1!\ldots j_k!}\leq \eee^k$, we obtain the required estimate choosing $M_1 = \eee M$.
\end{proof}

\begin{lemma}\label{lem:exp_psi_tilde_remainder}
Fix $r\in\N_0$. There exist a.s.\ finite random variables $U>0$ and $M_2 > 0$ such that for all $\beta \in  \fI_0, u\in (-U,U)$ and $s\in\R$ with $1+|s|<u^{-1/4}$, we have
$$
\left|\tilde \psi(s;u) - W_\infty(\beta)\eee^{-\frac 12 s^2} \sum_{k=0}^r \frac{B_k(a_1(s),\ldots, a_k(s))}{k!} u^k\right| \leq M_2 \eee^{-\frac 12 s^2} (1+|s|)^{3r+3} |u|^{r+1}.
$$
\end{lemma}
\begin{proof}
Using formula~\eqref{eq:psi_tilde_def2} for $\tilde \psi(s;u)$ and then Lemma~\ref{lem:B_k_estimate} we obtain
\begin{align*}
\text{LHS}
&\leq
|W_\infty(\beta)| \eee^{-\frac 12 s^2} \sum_{k=r+1}^{\infty} \frac{|B_k(a_1(s),\ldots, a_k(s))|}{k!} |u|^k\\
&\leq
|W_\infty(\beta)| \eee^{-\frac 12 s^2} \sum_{k=r+1}^{\infty} M_1^k (|s|+1)^{3k}|u|^k\\
&\leq
\frac{M_2}{2} \eee^{-\frac 12 s^2} (|s|+1)^{3r+3}|u|^{r+1}\sum_{k=0}^{\infty} M^k_1 (|s|+1)^{3k} |u|^k,
\end{align*}
where $M_2 = 2 M_1^{r+1} \sup_{\gamma \in \bD_{\eps_0}(\beta_0)} |W_\infty(\gamma)|$. The sum on the right-hand side can be estimated using the assumptions  $1+|s|<u^{-1/4}$ and $|u|<U$ as follows:
$$
\sum_{k=0}^{\infty} M_1^k (|s|+1)^{3k} |u|^k
\leq
\sum_{k=0}^{\infty} M_1^k |u|^{-\frac 34 k} |u|^k
\leq
\sum_{k=0}^{\infty} M_1^k U^{k/4} = 2,
$$
where the last step holds if we choose $ U = (16 M_1^4)^{-1}$.
\end{proof}

We are now able to state the expansion for the characteristic function $\psi_n$ with an estimate for the remainder term. Let
\begin{equation}\label{eq:V_r_n_def}
V_{r,n}(s) = W_\infty(\beta)\eee^{-\frac{1}{2}s^2}\sum_{k=0}^r \frac{B_k(a_1(s),\ldots, a_k(s))}{k!} w_n^{-\frac k2}
\end{equation}
\begin{lemma}\label{lem:exp_psi_n_remainder}
There exist a.s.\ finite numbers $K>0$ and $M_3>0$ such that
$$
\left|\psi_n(s) - V_{r,n}(s) \right| \leq M_3 \eee^{-\frac 12 s^2} (|s|+1)^{3r+3} w_n^{-\frac{r+1}{2}}.
$$
for all $\beta \in  \fI_0, n>K$ and $s\in \R$ satisfying $1+|s|<w_n^{1/8}$.
\end{lemma}
\begin{proof}
 We have
\begin{equation}\label{eq:est1}
\text{LHS}
\leq
\left|\tilde \psi (s;w_n^{-\frac 12}) - V_{r,n}(s)\right| + \left|\psi_n(s) - \tilde \psi (s; w_n^{-\frac 12}) \right|.
\end{equation}
We estimate the terms on the right-hand side in two steps.

\step We start with the first term on the RHS in~\eqref{eq:est1}. By Lemma~\ref{lem:exp_psi_tilde_remainder} with $u = w_n^{-1/2}$, the estimate
\begin{equation}\label{eq:est2}
\left|\tilde \psi (s;w_n^{-\frac 12}) - V_{r,n}(s)\right| \leq M_2 \eee^{-\frac 12 s^2} (|s|+1)^{3k+3} w_n^{-\frac {r+1}2}
\end{equation}
holds provided that $w_n^{-1/2} < U$ and $1+|s|<w_n^{1/8}$. Since $\lim_{n\to\infty} w_n =+\infty$, we can choose a random variable $K$ such that $w_n^{-1/2} < U$ for $n>K$.

\step We estimate the second term on the RHS in~\eqref{eq:est1}.  Let $z_n=\frac{is}{\sigma \sqrt{w_n}}$ so that for sufficiently large $n$, we have $|z_n|<\eps_0$. With this notation, we have
$$
\left|\psi_n(s) - \tilde \psi (s; w_n^{-\frac 12}) \right|
=
\left|\eee^{ w_n \left(\varphi(\beta + z_n) - \varphi(\beta) - \varphi'(\beta)z_n\right)}\right| \left| W_\infty(\beta + z_n) - W_n(\beta + z_n)\right|.
$$
By Assumption~A3, see~\eqref{eq:Psi_n_to_W_infty_speed}, we have, for some a.s.\ finite number $M'$ depending on $\beta_0$ and $\varepsilon_0$ but not on $\beta$,
$$
\left|W_\infty(\beta + z_n) - W_n(\beta + z_n)\right| \leq \sup_{z\in \bD_{2\eps_0}(\beta_0)} \left| W_\infty(z) - W_n(z)\right| < M' w_n^{-\frac{r+1}{2}}.
$$
By the Taylor expansion of $\varphi$ at $\beta$, we obtain the following estimate in which the $O$-term is uniform as long as $|z_n|<\eps_0$ and $\beta \in \fI_0$:
\begin{align}
w_n  \left(\varphi(\beta + z_n) - \varphi(\beta) - \varphi'(\beta) z_n\right)
&= \left(\frac{\sigma^2} 2 z_n^2 + O(z_n^3)\right) w_n\notag \\
&=-\frac{s^2}{2} + O\left(\frac{s^3}{\sqrt{w_n}}\right) \leq -\frac {s^2}{2} + O(w_n^{-1/8})\label{eq:est22},
\end{align}
where in the last step we used the restriction $1+|s|<w_n^{1/8}$.
Combining the above estimates we obtain
\begin{equation}\label{eq:est3}
\left|\psi_n(s) - \tilde \psi (s; w_n^{-\frac 12}) \right|
\leq
M'  w_n^{-\frac{r+1}{2}} \left( \eee^{-\frac 12 s^2 + O(w_n^{-1/8})}\right).
\end{equation}
Taking~\eqref{eq:est2} and~\eqref{eq:est3} together, we obtain the statement of the lemma.
\end{proof}

In order to obtain the Edgeworth expansion for $\bL_n(k)$ we shall apply Fourier inversion to the expansion for $\psi_n$ established above. Recall formula~\eqref{eq:psi_n_def} for the characteristic function $\psi_n$. It follows by Fourier inversion that
$$
\sigma \sqrt{w_n} \eee^{\beta k - \varphi(\beta) w_n} \bL_n(k) = \frac 1 {2\pi} \int_{-\pi\sigma \sqrt {w_n}}^{\pi\sigma \sqrt {w_n}} \psi_n(s) \eee^{- is x_n(k)} \dd s,
$$
where $x_n(k)$ was defined in~\eqref{eq:x_n_k_def}.
\begin{lemma}
Recall from~\eqref{eq:V_r_n_def} the definition of $V_{r,n}$. For every fixed $r\in\N_0$,
$$
w_n^{\frac r2} \sup_{k\in\Z} \sup_{\beta \in \fI_0} \left|\int_{-\pi\sigma \sqrt {w_n}}^{\pi\sigma \sqrt {w_n}} \psi_n(s) \eee^{-isx_n(k)}\dd s - \int_\R  V_{r,n}(s) \eee^{-is x_n(k)}\dd s \right| \toas 0.
$$
\end{lemma}
\begin{proof}
\step We show that
$$
w_n^{\frac r2} \sup_{\beta \in  \fI_0} \int_{-w_n^{1/9}}^{w_n^{1/9}} \left|\psi_n(s) - V_{r,n}(s)\right|\dd s \toas 0.
$$
Indeed, we know from Lemma~\ref{lem:exp_psi_n_remainder} that, for all $\beta \in  \fI_0$,
$$
\left|\psi_n(s) - V_{r,n}(s)\right| \leq M_3 \eee^{-\frac 12 s^2} (|s|+1)^{3r+3} w_n^{-\frac{r+1}{2}}
$$
for $n>K$, $1+|s|< w_n^{1/8}$.  Integrating this, we obtain the required estimate.

\step
We show that there is an $a>0$ such that
\begin{equation}\label{eq:cond_int_tn_19}
w_n^{\frac r2} \sup_{\beta \in  \fI_0} \int_{|w_n|^{1/9} < |s| < a\sqrt{w_n}} |\psi_n(s)| \dd s \toas 0.
\end{equation}
Let $z_n=\frac{is}{\sigma \sqrt{w_n}}$. We can choose $a>0$ so small that $|z_n|<\eps_0$ provided that $|s|<a \sqrt{w_n}$.
From the uniform convergence of $W_n$ to $W_\infty$ on $\bD_{2\eps_0}(\beta_0)$ and from the Taylor series for $\varphi$ we infer
$$
|\psi_n(s)|
=
\left|\eee^{ w_n \left(\varphi(\beta + z_n) - \varphi(\beta) - \varphi'(\beta)z_n\right)}\right| | W_n(\beta + z_n)|\leq M'\eee^{-\frac{1}{2}s^2},
$$
for some a.s.\ finite $M'>0$ depending on $\beta_0$ and $\varepsilon_0$ but not on $\beta$. It follows that
$$
\sup_{\beta \in  \fI_0} \int_{|w_n|^{1/9} < |s| < a\sqrt{w_n}} |\psi_n(s)| \dd s \leq M' \int_{|w_n|^{1/9} < |s| < a\sqrt{w_n}} \eee^{-\frac 12 s^2} \dd s
= o(w_n^{-\frac r2}) \quad \text{ a.s.}
$$
This completes the proof of~\eqref{eq:cond_int_tn_19}.

\step We prove that for every $a>0$,
\begin{equation}\label{eq:cond_int_asrt_sigmapisqrt}
w_n^{\frac r2}  \sup_{\beta \in   \fI_0} \int_{a\sqrt{w_n} < |s| < \sigma \pi \sqrt{w_n}} |\psi_n(s)| \dd s \toas 0.
\end{equation}
In this case, $z_n=\frac{is}{\sigma \sqrt{w_n}}$ need not satisfy $|z_n|\leq \eps_0$ so that $W_n$ need not converge (nor even be well-defined) and the estimate from Step~2 does not work. Instead, we shall use Assumption~A4. Using the definition of $\psi_n$, see~\eqref{eq:psi_n_def},
\begin{align*}
\int_{a\sqrt{w_n} < |s| < \sigma \pi \sqrt{w_n}} |\psi_n(s)| \dd s
&=
\eee^{-\varphi(\beta) w_n} \int_{a\sqrt{w_n} < |s| < \sigma \pi \sqrt{w_n}} \left|\sum_{k\in\Z} \bL_n(k) \eee^{k\left(\beta +  \frac{is}{\sigma \sqrt{w_n}}\right)}\right| \dd s\\
&=
\eee^{-\varphi(\beta) w_n}\sigma \sqrt{w_n}\int_{a/\sigma < |u| < \pi} \left|\sum_{k\in\Z} \bL_n(k) \eee^{k(\beta + iu)}\right| \dd u,
\end{align*}
so that~\eqref{eq:cond_int_asrt_sigmapisqrt} is implied by Assumption~A4 since $\sigma$ is bounded on $\fI_0$.

The same estimates as~\eqref{eq:cond_int_tn_19} and~\eqref{eq:cond_int_asrt_sigmapisqrt}, but with $V_{r,n}(s)$ instead of $\psi_n(s)$, hold since $V_{r,n}$ is a product of $\eee^{-s^2/2}$ and a polynomial in $s$. Combining pieces together we obtain the claim of the lemma.
\end{proof}

To complete the proof of Theorem~\ref{theo:edgeworth_general} it remains to show that
$$
\int_{\R}V_{r,n}(s)\eee^{-isz}\dd s = \sqrt{2\pi}W_{\infty}(\beta)\eee^{-\frac{1}{2}z^2}\sum_{k=0}^r\frac{G_k(z)}{w_n^{k/2}},\quad z\in\R,
$$
which, in turn, amounts to
\begin{equation}\label{eq:v_r_n_explicit}
\frac{1}{k!}\int_{\R}B_k(a_1(s),\ldots,a_k(s))\eee^{isx}\eee^{-\frac{1}{2}s^2}\dd s =\sqrt{2\pi}G_k(-x)\eee^{-\frac{1}{2}x^2},\quad x \in\R,
\end{equation}
for every $k\in\N_0$. To check~\eqref{eq:v_r_n_explicit} note that
$$
B_k(a_1(s),\ldots,a_k(s))\eee^{isx}=B_k(D_1,\ldots,D_k)(\eee^{isx}),\quad s\in\R,
$$
where the differential operators $D_1,\ldots,D_k$ are given by~\eqref{eq:D_def}. This yields
\begin{align*}
\int_{\R}B_k(a_1(s),\ldots,a_k(s))\eee^{isx}\eee^{-\frac{1}{2}s^2}\dd s
&=B_k(D_1,\ldots,D_k)\left(\int_{\R}\eee^{isx}\eee^{-\frac{1}{2}s^2}\dd s\right)\\
&=\sqrt{2\pi}B_k(D_1,\ldots,D_k)\eee^{-\frac{1}{2}x^2}\\
&=\sqrt{2\pi}(-1)^k k! \eee^{-\frac{1}{2}x^2}G_k(x).
\end{align*}
Formula~\eqref{eq:v_r_n_explicit} now follows from the observation $(-1)^kG_k(x)=G_k(-x)$, $k\in\N_0$,
see Remark~\ref{rem:G_odd_even}.
The proof of Theorem~\ref{theo:edgeworth_general} is complete.

\subsection{Alternative expression for \texorpdfstring{$G_j(x;\beta)$}{Gj(x;beta)}}
\label{subsec:alternative_expression}
In this section we show that upon multiplying by $W_\infty(\beta)$ the functions $x\mapsto G_j(x;\beta)$, $j\in\N_0$, become polynomials in $x$ whose coefficients are linear combinations of $1, W_\infty(\beta), \ldots, W_{\infty}^{(j)}(\beta)$.
Write $D=\frac 1{\sigma(\beta)}\frac{{\rm d}}{{\rm d} x}$ and recall from~\eqref{eq:D_def} that
$$
D_k = \frac{\varphi^{(k+2)}(\beta)}{(k+1)(k+2)} D^{k+2} + \chi_k(\beta) D^k.
$$
From now on we consider $D$ as a formal variable rather than a differential operator. Formula~\eqref{eq:G_def} shows that $W_\infty(\beta) G_j(x;\beta)$ can be obtained from the expression $(-1)^j/j!W_\infty(\beta) B_j(D_1,\ldots,D_j)$  by replacing each $D^k$-term by the Hermite polynomial $(-1)^k\Herm_k(x)$.
Therefore, it suffices to show that this expression is a polynomial in $D$ whose coefficients are linear combinations of   $1,W_\infty(\beta), \ldots, W_{\infty}^{(j)}(\beta)$.  By the definition of the Bell polynomials, see~\eqref{eq:bell_poly_def1}, we have
$$
\frac 1 {j!}B_j(D_1,\ldots,D_j) = [y^j] \exp\left\{\sum_{k=1}^{\infty} \frac{y^k}{k!} D_k\right\},
$$
where $[y^j]f(y)$ denotes the coefficient of $y^j$ in the formal power series $f(y)$.
It follows that, in the sense of formal power series,
$$
\frac {1} {j!}B_j(D_1,\ldots,D_j)= [y^j] \left(\exp\left\{\sum_{k=1}^{\infty} \frac{y^k \varphi^{(k+2)}(\beta)}{(k+2)!}  D^{k+2}\right\}
\exp\left\{ \sum_{k=1}^\infty \frac {(yD)^k} {k!} \chi_k(\beta)\right\}\right).
$$
Multiplying both sides by $W_{\infty}(\beta)=\eee^{\chi_0(\beta)}$ and observing that, by Taylor's expansion,
$$
W_{\infty}(\beta+yD)=\exp\left\{\sum_{k=0}^\infty \frac {(yD)^k} {k!} \chi_k(\beta)\right\},
$$
we obtain
\begin{align*}
&\frac {W_{\infty}(\beta)} {j!}B_j(D_1,\ldots,D_j) \\
&= [y^j]\left(\exp\left\{\sum_{k=1}^{\infty} \frac{y^k \varphi^{(k+2)}(\beta)}{(k+2)!} D^{k+2}\right\} W_\infty(\beta + yD)\right)\\
&= [y^j] \left(\exp\left\{\sum_{k=1}^{\infty} \frac{y^k\varphi^{(k+2)}(\beta)}{(k+2)!}  D^{k+2}\right\}\left(\sum_{k=0}^\infty \frac {(yD)^k} {k!} W_\infty^{(k)}(\beta) \right)\right).
\end{align*}
Clearly, the right-hand side is a polynomial in $D$ whose coefficients are linear combinations of  $1, W_\infty(\beta), \ldots, W_{\infty}^{(j)}(\beta)$.

\subsection{Proofs of Theorems~\ref{theo:mode}, \ref{theo:width}, \ref{theo:width_limit_law} and Proposition~\ref{cor:one_split_mode}}\label{subsec:proof_mode_width}
Our proof runs along the same lines as the proof of Theorem~2.17 in~\citet{gruebel_kabluchko_BRW}. In order  to keep this paper self-contained we present all details.

The aim is to search for $k\in\Z$ maximizing the profile $\bL_n(k)$. Write $k=k_n(a)=\varphi'(0)w_n + a$, where $a\in \Z - \varphi'(0) w_n$. We use Theorem~\ref{theo:edgeworth_general} with $\beta=0$ and $r=2$. Inserting this $k$ into~\eqref{eq:edgeworth_general_exp} we obtain
\begin{eqnarray}
&&\hspace{-1.4cm}\sigma(0)\sqrt{2\pi w_n} \eee^{-\varphi(0)w_n} \bL_n(k)=W_{\infty}(0)\eee^{-\frac{a^2}{2\sigma^2(0)w_n}}\times\notag\\
&&\left(1+\frac{1}{\sqrt{w_n}}G_1\left(\frac{a}{\sigma(0)\sqrt{w_n}}\right)+\frac{1}{w_n}G_2\left(\frac{a}{\sigma(0)\sqrt{w_n}}\right)+o\left(\frac{1}{w_n}\right)\right)\label{eq:tech_expansion_r2},
\end{eqnarray}
where the $o$-term is uniform in $a$. This and similar expansions below hold a.s.

\vspace*{2mm}
\noindent
\textsc{Step 1.} Let us assume that $|a| < w_n^{1/4-\eps}$ for some small $\eps>0$, say $\eps=1/100$.   Since $x_n(k)=a/(\sigma(0)\sqrt{w_n})$ and $|a| < w_n^{1/4-\eps}$ we have
$$
x_n(k)=o(1)\quad\text{and}\quad x^3_n(k)=o\left(\frac{1}{\sqrt{w_n}}\right).
$$
Hence, by~\eqref{eq:G1},
$$
G_1\left(\frac{a}{\sigma(0)\sqrt{w_n}}\right)=\frac{a}{\sqrt{w_n}}\left(\frac{\chi_1(0)}{\sigma^2(0)}-\frac{\kappa_3(0)}{2\sigma^4(0)}\right)+o\left(\frac{1}{\sqrt{w_n}}\right),
$$
and by~\eqref{eq:G2},
\begin{align*}
G_2\left(\frac{a}{\sigma(0)\sqrt{w_n}}\right)
&=
-\frac{\chi_1^2(0)+\chi_2(0)}{2\sigma^2(0)}
+3\left(\frac{\kappa_4(0)}{24\sigma^4(0)}
+\frac{\kappa_3(0)\chi_1(0)}{6\sigma^4(0)}\right)
-\frac{15\kappa_3^2(0)}{72\sigma^6(0)} +o(1)\\
&=:C+o(1).
\end{align*}
By a standard Taylor expansion, we get
\begin{equation*}
\eee^{-\frac{a^2}{2\sigma^2(0)w_n}}=1 - \frac {a^2}{2\sigma^2(0) w_n} + o\left(\frac{1}{w_n}\right).
\end{equation*}
Inserting these expansions into~\eqref{eq:tech_expansion_r2}, we arrive at
\begin{eqnarray}
&&\hspace{-1cm}\sigma(0)\sqrt{2\pi w_n}\eee^{-\varphi(0)w_n}\bL_n(k)\notag\\
&&=W_{\infty}(0)\left(1-\left(\frac{a^2}{2\sigma^2(0)}-\left(\frac{\chi_1(0)}{\sigma^2(0)}-\frac{\kappa_3(0)}{2\sigma^4(0)}\right)a-C\right)\frac{1}{w_n}\right)+o\left(\frac{1}{w_n}\right)\label{eq:tech_final_exp}.
\end{eqnarray}
Differentiation with respect to $a$ shows that the maximum is attained at
\begin{equation}\label{eq:a_*}
a_* = \chi_1(0)-\frac{\kappa_3(0)}{2\sigma^2(0)}.
\end{equation}
Using~\eqref{eq:tech_final_exp} and~\eqref{eq:a_*}, we obtain
\begin{equation}\label{eq:L_n_k_plus_1_L_n_k}
\sigma(0)\sqrt{2\pi w_n} \eee^{-\varphi(0)w_n}(\bL_n(k+1) - \bL_n(k)) = \frac {W_\infty(0)} {\sigma^2(0) w_n} \left(a_* - \frac {1}{2} -a \right) + o\left(\frac{1}{w_n}\right).
\end{equation}
Put $u_n^* = \varphi'(0)w_n + a_*$. From~\eqref{eq:L_n_k_plus_1_L_n_k} it is clear that,  for all sufficiently large $n$, the value $u_n$ maximizing $\bL_n(\cdot)$ in the region $|a|<w_n^{1/4-\eps}$ is either $\lfloor u_n^*\rfloor$ or $\lceil u_n^*\rceil$. 

\vspace*{2mm}
\noindent
\textsc{Step 2.} Note that $a_*=O(1)$ a.s. If $\lfloor u_n^*\rfloor$ or $\lceil u_n^*\rceil$ is indeed the mode $u_n$ of $\bL_n(\cdot)$ (over the whole range $k\in\Z$), then, from~\eqref{eq:tech_final_exp}, we deduce
$$
\sigma(0)\sqrt{2\pi w_n}\eee^{-\varphi(0)w_n}\bL_n(u_n)=W_{\infty}(0)+O\left(\frac{1}{w_n}\right),
$$
and Theorem~\ref{theo:width} follows. To complete the proof of Theorem~\ref{theo:mode} it remains to show that (for $n$ large enough) the mode cannot lie in the region
$|a|\geq w_n^{1/4-\eps}$. To this end we shall show that for every $B>0$ there exists an a.s.\ finite $K\in\N$ such that
\begin{equation}\label{eq:mode_not_large}
W_\infty(0) - \sigma(0)\sqrt{2\pi w_n}\eee^{-\varphi(0)w_n} \bL_n(k) > \frac{B}{w_n}
\end{equation}
for all $n>K$ and $|a|\geq w_n^{1/4-\eps}$.

\vspace*{2mm}
\noindent
\textsc{Step 3.} Let $|a|>\sigma(0) \sqrt{w_n}$. By Theorem~\ref{ref:LLT_general},
$$
W_\infty(0)-\sigma(0)\sqrt{2\pi w_n}\eee^{-\varphi(0)w_n} \bL_n(k) = W_\infty(0) \left(1-\eee^{-\frac{a^2}{\sigma^2(0)w_n}}\right) + o(1).
$$
Since $W_{\infty}(0)>0$ a.s.,  the expression on the right-hand side is larger than $W_{\infty}(0)/10$ for all $n$ sufficiently large.  Hence, there exists an a.s.\ finite $K_1 \in \N$ such that~\eqref{eq:mode_not_large} holds for $|a|>\sigma(0) \sqrt{w_n}$ and $n>K_1$.

\vspace*{2mm}
\noindent
\textsc{Step 4.} Let  $\sigma(0) w_n^{3/8 + \eps} < |a| \leq \sigma(0) \sqrt {w_n}$ for some small $\eps>0$. By Theorem~\ref{theo:edgeworth_general} with $r=1$ and $\beta=0$, we have
\begin{eqnarray*}
&&\hspace{-1cm}\sigma(0)\sqrt{2\pi w_n} \eee^{-\varphi(0)w_n} \bL_n(k)\\
&&=W_{\infty}(0)\eee^{-\frac{a^2}{2\sigma^2(0)w_n}}\left(1+\frac{1}{\sqrt{w_n}}G_1\left(\frac{a}{\sigma(0)\sqrt{w_n}}\right)\right)+o\left(\frac{1}{\sqrt{w_n}}\right) \quad \text{a.s.}.
\end{eqnarray*}
Since $|a|\leq \sigma(0) \sqrt {w_n}$ we have $G_1\left(\frac{a}{\sigma(0)\sqrt{w_n}}\right)=O(1)$ a.s. Therefore,
\begin{align*}
W_{\infty}(0)-\sigma(0)\sqrt{2\pi w_n}\eee^{-\varphi(0)w_n}\bL_n(k) & =W_{\infty}(0)\left(1-\eee^{-\frac{a^2}{2\sigma^2(0)w_n}}\right)+O\left(\frac{1}{\sqrt{w_n}}\right) \\
& \geq W_{\infty}(0)\left(1-\eee^{-\frac 1 2 w_n^{-1/4 + 2 \eps}}\right)+O\left(\frac{1}{\sqrt{w_n}}\right). \end{align*}
From the elementary inequality \begin{align} \label{inq:ee} 1- \eee^{-y/2} > y/3, \quad y \in [0,1], \end{align} it now easily follows that there exist
an a.s.\ finite $K_2\in\N$ such that  ~\eqref{eq:mode_not_large} is satisfied for $\sigma(0) w_n^{3/8 + \eps} < |a| \leq \sigma(0) \sqrt {w_n}$ and $n>K_2$.

\vspace*{2mm}
\noindent
\textsc{Step 5.} Finally, let  $\sigma(0) w_n^{1/4-2\eps} < |a| \leq \sigma(0) w_n^{3/8 + \eps}$. Using~\eqref{eq:tech_expansion_r2} and noting that
$$
\frac{1}{\sqrt{w_n}} G_1\left(\frac{a}{\sigma(0)\sqrt{w_n}}\right) =O(w_n^{-\frac{5}{8} + \eps}),
\quad
\frac{1}{w_n} G_2\left(\frac{a}{\sigma(0)\sqrt{w_n}}\right) = O\left(\frac{1}{w_n}\right)\quad \text{a.s.},
$$
we obtain
\begin{align*}
W_\infty(0)-\sigma(0)\sqrt{2\pi w_n}\eee^{-\varphi(0)w_n}\bL_n(k) & = W_\infty(0) \left(1-\eee^{-\frac{a^2}{2\sigma^2(0)w_n}}\right) +  O\left(w_n^{-\frac{5}{8} + \eps}\right) \\
& \geq \frac 13 W_\infty(0)   w_n^{-1/2-4\eps} +  O\left(w_n^{-\frac{5}{8} + \eps}\right),
\end{align*}
where the second inequality follows from~\eqref{inq:ee} and holds for all $n$ sufficiently large.
Hence, upon choosing $\eps$ sufficiently small, there exists an a.s.\ finite $K_3 \in \N$ such that~\eqref{eq:mode_not_large} holds for all $n>K_3$ and $\sigma(0) w_n^{1/4-2\eps} < |a| \leq \sigma(0) w_n^{3/8 + \eps}$. The proof of Theorem~\ref{theo:mode} is complete.

\vspace*{2mm}
\noindent
\textsc{Step 6.} It remains to prove Theorem~\ref{theo:width_limit_law}. From what we have already proved, it follows that, for all $n$ sufficiently large, the mode $u_n$ can be written as
$$
u_n=u_n^*+\gamma_n=\varphi'(0)w_n+a_*+\gamma_n,\quad |\gamma_n|<1,
$$
where $\gamma_n$ equals either $\lfloor u_n^*\rfloor-u_n^*$ or $\lceil u_n^*\rceil-u_n^*$. Inserting $u_n$ into~\eqref{eq:tech_final_exp}, we obtain
\begin{eqnarray*}
&&\hspace{-2cm}\sigma(0)\sqrt{2\pi w_n}\eee^{-\varphi(0)w_n}\bL_n(u_n)\\
&=&W_{\infty}(0)\left(1-\left(\frac{(a_*+\gamma_n)^2}{2\sigma^2(0)}-\frac{a_*(a_*+\gamma_n)}{\sigma^2(0)}-C\right)\frac{1}{w_n}\right)+o\left(\frac{1}{w_n}\right)\\
&=&W_{\infty}(0)\left(1+\frac{a_*^2-\gamma_n^2}{2\sigma^2(0)w_n}+\frac{C}{w_n}\right)+o\left(\frac{1}{w_n}\right)\quad\text{a.s.}
\end{eqnarray*}
Since $u_n$ is the mode and hence maximizes the left-hand side, $\gamma_n$ in the above relation can be replaced by the $\theta_n:=\min\{u_n^*-\lfloor u_n^*\rfloor,\lceil u_n^*\rceil-u_n^*\}$. Upon rearranging the terms and recalling the notation $\tilde{M}_n$ from Theorem~\ref{theo:width_limit_law}, this becomes
$$
\tilde{M}_n-\theta_n^2=-a_*^2-2\sigma^2(0)C+o(1)=\chi_2(0)+\frac{\kappa_3^2(0)}{6\sigma^4(0)}-\frac{\kappa_4(0)}{4\sigma^2(0)}+o(1)\quad\text{a.s.},
$$
where the last equality follows from~\eqref{eq:a_*} and the definition of $C$. This concludes the proof of Theorem~\ref{theo:width_limit_law}.

\begin{proof}[Proof of Proposition~\ref{cor:one_split_mode}]
Both assertions follow from properties of the logarithm. The first claim (i) follows immediately from the fact that, for every fixed $L>0$, we have  $\log (n+L)-\log n\to 0$ as $n\to\infty$. To show (ii), it is sufficient to verify that,  almost surely,
$$\limsup_{\eps \to 0} \limsup_{n \to \infty} \frac{\#\{1 \leq k  \leq n: \text{dist}(u_k^*, \Z + 1/2) < \eps \} }{n} = 0.$$
Since $\varphi'(0) \neq 0$, using the explicit expression~\eqref{eq:u_n_star_one_split}, the claim follows if, for any $\alpha>0$ and $\beta \in \R$,
$$\limsup_{\eps \to 0} \limsup_{n \to \infty} \frac{\#\{1 \leq k  \leq n: \text{dist}(\log k, \alpha \Z + \beta) < \eps \}  }{n} = 0.$$
Let us show how to estimate the numerator of the last fraction. We have, assuming that $\varepsilon<\alpha/2$,
\begin{align*}
\#\{1 \leq k  \leq n: \text{dist}(\log k, & \alpha \Z + \beta) < \eps \}
=\sum_{k=1}^{n} \# \{ j \in \Z : \text{dist}(\log k, \alpha j + \beta) < \eps\}\\
& = \sum_{j\in\Z}\#\{1 \leq k  \leq n: \eee^{\alpha j+\beta-\eps} < k < \eee^{\alpha j+\beta+\eps} \} \\
& \leq \sum_{j\in\Z} \# \{k \in \N : \eee^{\alpha j+\beta-\eps}\vee 1 \leq k \leq \eee^{\alpha j+\beta+\eps}\wedge n \}.
\end{align*}
The inner sum on the right-hand side is the number of integers in the interval $[\eee^{\alpha j+\beta-\eps}\vee 1,\,\eee^{\alpha j+\beta+\eps}\wedge n]$ (which is empty if either $\eee^{\alpha j+\beta-\eps}>n$ or $\eee^{\alpha j+\beta+\eps}<1$) and hence is bounded from above by
$\left(\eee^{\alpha j+\beta+\eps}\wedge n - \eee^{\alpha j+\beta-\eps}\vee 1+1\right)_{+}$. Therefore,
$$
\#\{1 \leq k  \leq n: \text{dist}(\log k, \alpha \Z + \beta) < \eps \} |\leq \sum_{j\in\Z}\left(\eee^{\alpha j+\beta+\eps}\wedge n - \eee^{\alpha j+\beta-\eps}\vee 1+1\right)_{+}.
$$
The relation
$$
\limsup_{\eps \to 0} \limsup_{n \to \infty}\frac{\sum_{j\in\Z}\left(\eee^{\alpha j+\beta+\eps}\wedge n - \eee^{\alpha j+\beta-\eps}\vee 1+1\right)_{+}}{n}=0
$$
can be checked by direct calculations. We omit further details; see~\cite[Proof of Theorem 1.4 (iii)] {kabluchko_marynych_sulzbach_on_stirling}.
\end{proof}

\section{Proofs for random trees}\label{sec:proof_random_trees}
\subsection{Embedding the one-split BRW into a continuous-time BRW}\label{subsec:embedding}
Con\-tin\-uous-time embeddings of discrete-time Markov chains in the study of random discrete structures go back at least to Athreya and Karlin~\cite{atka} in the context of P\'olya urn models. In the framework of random trees, Pittel~\cite{pittel2} was the first to use a continuous-time embedding in the analysis of the height of BSTs. In the study of the profile of BSTs, the idea was introduced in a series of works by Chauvin and collaborators; see~\cite{chauvin_etal_yule,chauvin_etal_tiltings,chauvin_etal}. More recent works crucially relying on this technique are, among others,  \cite{gruebel_kabluchko,schopp,sulzbach} and~\cite{sulzbach_mart_tail_sums}. Start with a one-split BRW as described in Section~\ref{subsec_one_split_BRW_def}. Consider a \emph{continuous-time} BRW which starts with a single particle at the origin at time $\tau_0 :=  0$ and in which any particle splits, with intensity $1$, into a cluster of particles described by the same point process $\zeta$ as in the one-split BRW. The particles do not move between the splits. Denote the split times by $\tau_1<\tau_2<\ldots$ and write $N_t$ for the number of particles in the process at time $t \geq 0$. Note that $(N_t)_{t\geq 0}$ is a Galton--Watson process in continuous time. Further, we let $z_{1, t}, \ldots,  z_{N_t, t}$ be the positions of the particles and
$$
\mathcal L_{t}(k) =  \#\{1\leq j\leq N_t \colon z_{j,t} = k\}, \quad k\in\Z,
$$
be the corresponding profile at time $t\geq 0$.  We have the following correspondence:
$$
S_n = N_{\tau_n}, \quad\;\; x_{i, n} = z_{i, \tau_n}, \;1\leq i\leq S_n, \quad\;\; \bL_{n}(k) = \mathcal L_{\tau_n}(k),\; k\in\Z.
$$
For $\beta \in \C, \Re\beta \in \fI$ (see Assumption B3) consider the Biggins martingale:
$$
\mathcal W_t(\beta) = \eee^{-m(\beta)t} \sum_{k = 1}^{N_t} \eee^{\beta z_{k,t}},
\quad t\geq 0.
$$
Set $H_n(\beta):=\eee^{m(\beta)\tau_n-\varphi(\beta)\log n}$, and note that $W_n(\beta)$ (see formula~\eqref{eq:jabbour_one_split}) and $\mathcal{W}_{\tau_n}(\beta)$ are connected via the relation
\begin{align} \label{connection}
W_{n}(\beta) = \mathcal W_{\tau_n} (\beta) H_n(\beta).
\end{align}

Our aim is to show that $W_n$ converges, with speed $(\log n)^{-r}$, to a random analytic function $W_\infty$, thus verifying Assumptions A2 and A3 of Theorem~\ref{theo:edgeworth_general}.
Let us analyze the factors on the right-hand side of~\eqref{connection}. Let $m^*(\beta) = \eee^{m(\beta)}$. For $\gamma \in (1,2]$, define the open sets
\begin{align*}
\Omega_\gamma^1 &= \text{int} \left \{\beta \in \C : \Re \beta \in \fI, \E \left[ \left(\sum_{k \in \Z} \eee^{(\Re \beta) k} \mathcal L_1(k) \right)^\gamma \right] < \infty \right \}, \\
\Omega_\gamma^2 &=  \left \{\beta \in \C: \gamma \Re \beta \in \fI, \frac{m^*(\gamma \Re \beta )}{|m^*(\beta)|^\gamma} < 1 \right \},
\end{align*}
and let
$$
\fD = \bigcup_{\gamma \in (1,2]} (\Omega_\gamma^1  \cap \Omega_\gamma^2) \subset \C.
$$
Note that the set $\fD$ is open. 
Biggins~\cite{biggins_uniform} proved that, with probability $1$, $\mathcal W_t$ converges locally uniformly on $\fD$, as $t\to\infty$.
The next proposition is a slight extension of this classical result adapted to our needs.

\begin{proposition} \label{prop:bigg} Under Assumptions B1--B3 and B5, there exists a random analytic function $\mathcal W_\infty$ on $\fD$ such that, for all compact sets
$K \subset \fD$, there exists $ 0 < r = r(K) < 1$
with
\begin{align} \label{conv_unif_cont}
r^{-t} \sup_{\beta \in K} \left| \mathcal W_t(\beta) - \mathcal W_\infty(\beta) \right| \toast 0.
\end{align}
It holds that $(\beta_-, \beta_+) \subset \fD$.
Finally, for $\gamma \in (1,2]$ and $\beta \in \Omega_\gamma^1  \cap \Omega_\gamma^2$, we have
\begin{align} \label{conv:lp}
\lim_{t\to\infty}\E{\left| \mathcal W_t(\beta) - \mathcal W_\infty(\beta) \right|^\gamma} = 0.
\end{align}
\end{proposition}
\begin{proof}
Fix $\varepsilon>0$ small enough. By compactness we can assume that $K=\overline{\mathbb{D}}_{\varepsilon}(z_0)\subset \Omega_\gamma^1  \cap \Omega_\gamma^2$ for some $\gamma \in (1,2]$ and some $z_0\in\fD$, and, moreover, $\overline{\mathbb{D}}_{2\varepsilon}(z_0)\subset \Omega_\gamma^1  \cap \Omega_\gamma^2$.

Note that, for $\Re \beta \in \fI$, the process  $(\mathcal W_n(\beta))_{n \in \N_0}$ is the discrete-time Biggins martingale corresponding to a standard (many-split) discrete-time branching random walk whose point process $\zeta^*=\sum_{k\in\Z} \mathcal L_1(k) \delta_k$ has moment generating function $m^*(\beta)=\eee^{m(\beta)}$. Hence, following the proof of Theorem 2 in~\cite{biggins_uniform}, there exists a random analytic function $\mathcal W_\infty$ on $\fD$ and $r_1=r_1(z_0,\varepsilon)\in(0,1)$ such that
$$r_1^{-n} \sup_{\beta \in \overline{\mathbb{D}}_{\varepsilon}(z_0)} \left| \mathcal W_n(\beta) - \mathcal W_\infty(\beta) \right| \toas 0.$$
To prove~\eqref{conv_unif_cont} it remains to show that there exist $r_2\in(0,1)$ such that
\begin{equation}\label{eq:continuous_time_biggins_convergence}
r_2^{-n}  \sup_{t\in [n,n+1]}\sup_{\beta \in \overline{\mathbb{D}}_{\varepsilon}(z_0)}\left| \mathcal W_t(\beta) - \mathcal W_n(\beta) \right| \toas 0.
\end{equation}
By Cauchy's integral formula, c.f.~\cite[Lemma 3]{biggins_uniform}
$$
\sup_{\beta \in \overline{\mathbb{D}}_{\varepsilon}(z_0)}\left| \mathcal W_t(\beta) - \mathcal W_n(\beta) \right|\leq \frac{1}{\pi}\int_0^{2\pi}\left|\mathcal W_t(z_0+2\eps \eee^{i\phi}) - \mathcal W_n(z_0+2\eps \eee^{i\phi})\right|{\rm d}\phi
$$
and, using H\"{o}lder's inequality for integrals,
\begin{multline*}
\sup_{\beta \in \overline{\mathbb{D}}_{\varepsilon}(z_0)}\left| \mathcal W_t(\beta) - \mathcal W_n(\beta) \right|^{\gamma} \leq \frac{1}{\pi^{\gamma}}\left(\int_0^{2\pi}\left|\mathcal W_t(z_0+2\eps \eee^{i\phi}) - \mathcal W_n(z_0+2\eps \eee^{i\phi})\right|{\rm d}\phi\right)^{\gamma}\\
\leq \frac{2^{\gamma-1}}{\pi}\int_0^{2\pi}\left|\mathcal W_t(z_0+2\eps \eee^{i\phi}) - \mathcal W_n(z_0+2\eps \eee^{i\phi})\right|^{\gamma}{\rm d}\phi.
\end{multline*}
Therefore,
\begin{multline*}
\E \left[\sup_{t\in [n,n+1]}\sup_{\beta \in \overline{\mathbb{D}}_{\varepsilon}(z_0)}\left| \mathcal W_t(\beta) - \mathcal W_n(\beta) \right|^{\gamma}\right]\\
\leq
\frac{2^{\gamma-1}}{\pi}\int_0^{2\pi}\E\left[\sup_{t\in [n,n+1]}\left|\mathcal W_t(z_0+2\eps \eee^{i\phi}) - \mathcal W_n(z_0+2\eps \eee^{i\phi})\right|^{\gamma}\right]{\rm d}\phi.
\end{multline*}
Since $(\mathcal W_{t}(\beta) - \mathcal W_n(\beta))_{t \geq n}$ is a martingale, by Doob's inequality, there exists a universal constant $C > 0$ such that, for all $n \in \N$ and $\beta\in\fD$,
$$\E \left[ \sup_{t \in [n, n+1]}  \left| \mathcal W_{t}(\beta) - \mathcal W_n(\beta) \right|^\gamma \right] \leq C \E{ \left| \mathcal W_{n+1}(\beta) - \mathcal W_n(\beta)  \right| ^\gamma}.$$
By~\cite[Lemma 2(i)]{biggins_uniform}, there is  $M>0$ such that for  all $\beta\in\fD$,
$$
\E{ \left| \mathcal W_{n+1}(\beta) - \mathcal W_n(\beta)  \right| ^\gamma}\leq M \left(\frac{m^{\ast}(\gamma \Re \beta)}{|m^{\ast}(\beta)|^{\gamma}}\right)^n.
$$
Combining pieces together, we obtain
$$
\E \left[\sup_{t\in [n,n+1]}\sup_{\beta \in \overline{\mathbb{D}}_{\varepsilon}(z_0)}\left| \mathcal W_t(\beta) - \mathcal W_n(\beta) \right|^{\gamma}\right] \leq 2^{\gamma}MC r_3^n,
$$
where
$$
r_3:=\sup_{\beta\in\overline{\mathbb{D}}_{2\varepsilon}(z_0)}\frac{m^{\ast}(\gamma \Re \beta)}{|m^{\ast}(\beta)|^{\gamma}}<1.
$$
By the Borel--Cantelli lemma and the Markov inequality, \eqref{eq:continuous_time_biggins_convergence} holds with arbitrary $r_2\in (r_3,1)$.

Equation~\eqref{conv:lp} follows analogously upon noting that, by~\cite[Theorem 1]{biggins_uniform}, for $\gamma \in (1,2]$ and $\beta \in \Omega_\gamma^1  \cap \Omega_\gamma^2$, it holds
\begin{align*}
\lim_{n\to\infty} \E{\left| \mathcal W_n(\beta) - \mathcal W_\infty(\beta) \right|^\gamma} = 0.
\end{align*}
\end{proof}

\begin{proposition} \label{prop:bigg_no_zeros}
Under Assumptions B1--B3 and B5, almost surely, the function $\mathcal W_\infty$ has no zeros on the interval $(\beta_{-}, \beta_{+})$.
\end{proposition}
\begin{proof}
For any fixed $\beta\in (\beta_{-},\beta_{+})$, it is known that
\begin{equation}\label{eq:W_infty_no_fixed_zeros}
\P[\mathcal W_\infty(\beta)=0]
=0
\end{equation}
since the extinction probability of our BRW equals zero by Assumption B2; see Theorem 1 in~\cite{BigGre79}.
We use a well-known fact that the limit process $(\mathcal W_\infty(\beta))_{\beta\in\fD}$, satisfies a stochastic fixed-point equation. Let $\tau_1^*$ be the time of the first \textit{non-trivial} split, that is the first time the number of particles in the BRW becomes at least $2$.  For $i\in\N$, denote by $(\mathcal W_t^{(i)})_{t\geq 0}$ the Biggins martingale corresponding to the  continuous-time BRW initiated by the $i$-th individual at time $\tau_1^*$. With probability $1$,
$$
\mathcal W_t(\beta) = \ind_{\{t \geq \tau_1^*\}} \sum_{i = 1}^{N_{\tau_1^*}} \eee^{-m(\beta) \tau_1^*} \eee^{\beta z_{i, \tau_1^*}} \mathcal W^{(i)}_{t - \tau_1^*}(\beta)
+ \ind_{\{t < \tau_1^*\}} \eee^{-m(\beta) t} \eee^{\beta z_{1,t}},
\quad \beta \in \fD.
$$
Thus, sending $t$ to $+\infty$, we obtain
\begin{equation}\label{eq:functional_eq_for_W_infty}
\mathcal W_\infty(\beta) = \sum_{i = 1}^{N_{\tau_1^*}} \eee^{-m(\beta) \tau_1^*} \eee^{\beta z_{i,\tau_1^*}} \mathcal W^{(i)}_\infty(\beta),
\quad \beta\in\fD.
\end{equation}
Here, the martingale limits $(\mathcal W^{(i)}_\infty(\beta))_{\beta \in \fD}$, $i \in\N$, are independent and have the same law as the process $(\mathcal W_\infty(\beta))_{\beta \in \fD}$. Also, these processes are independent of the number and the positions of the offspring of the initial particle, but not independent of $(\mathcal W_\infty(\beta))_{\beta \in \fD}$.

If the function $\mathcal W_\infty$ has a zero at some $\beta\in (\beta_-,\beta_+)$, then it follows from~\eqref{eq:functional_eq_for_W_infty} that the processes $\mathcal W_\infty^{(i)}$, $1\leq i \leq N_{\tau_1^*}$, have zeros at the same point $\beta$. Since $N_{\tau_1^*}\geq 2$, we infer
\begin{multline*}
\P[\mathcal W_\infty(\beta)=0\text{ for some }\beta\in (\beta_-,\beta_+)]\\
\leq \P[\mathcal W_\infty^{(1)}(\beta)=\mathcal W_\infty^{(2)}(\beta)=0\text{ for some } \beta\in (\beta_-,\beta_+)].
\end{multline*}
Since $(\mathcal W^{(1)}_\infty(\beta))_{\beta \in \fD}$ is a random analytic function on $\fD$ (which is not identically zero, with probability $1$; see~\eqref{eq:W_infty_no_fixed_zeros}), its zeros form a point process on $(\beta_-,\beta_+)$ and we can construct a sequence of random variables $X_1,X_2,\ldots$ (which depend only on $\mathcal W^{(1)}_\infty$) such that all zeros of $\mathcal W^{(1)}_\infty$ are contained in this list.
Since $(\mathcal W^{(1)}_\infty(\beta))_{\beta \in \fD}$ and $(\mathcal W^{(2)}_\infty(\beta))_{\beta \in \fD}$ are independent, we infer
\begin{align*}
\P[\mathcal W_\infty(\beta)=0\text{ for some }\beta\in (\beta_-,\beta_+)]
&\leq
\sum_{i=1}^\infty \P[\mathcal W_\infty^{(2)}(X_i)=0]\\
&=
\sum_{i=1}^\infty \int_{(\beta_{-},\beta_{+})} \P[\mathcal W^{(2)}_\infty(x)=0] \P[X_i\in {\rm d}x],
\end{align*}
which vanishes because $\P[\mathcal W^{(2)}_\infty(x)=0]=0$ for every $x\in (\beta_-,\beta_+)$ by~\eqref{eq:W_infty_no_fixed_zeros}.
\end{proof}


From Proposition~\ref{prop:bigg} we can easily obtain the following a.s.\ asymptotics for the $n$-th split time $\tau_n$.
\begin{lemma}\label{lem:split_time}
There exists a deterministic $\varepsilon > 0$ such that, as $n\to\infty$,
\begin{equation}\label{eq:tau_n_precise}
\tau_n = \frac{\log n}{m(0)} +\frac{\log m(0)-\log \mathcal{W}_{\infty}(0)}{m(0)}+o(n^{-\varepsilon}) \quad a.s.
\end{equation}
\end{lemma}
\begin{proof}
From Proposition~\ref{prop:bigg} with $\beta=0$ we obtain
$$
\mathcal W_{t}(0)\toast \mathcal{W}_{\infty}(0).
$$
The continuous-time Galton--Watson process $(N_t)_{t\geq 0}$ does not explode because the expected number of particles in the cluster $\zeta$ is finite by Assumption B3. This means that $\tau_n\to\infty$ a.s., as $n\to\infty$, and the last display implies
$$
\mathcal W_{\tau_n}(0)=\eee^{-m(0)\tau_n}S_n \toas \mathcal{W}_{\infty}(0).
$$
From Remark~\ref{rem:W_infty_0} we know that $S_n/n \to m(0)$ a.s.,\ which yields
\begin{equation}\label{eq:tau_n_first_order}
\tau_n=\frac{\log n}{m(0)}+O(1)\quad \text{a.s.}
\end{equation}
Using Proposition~\ref{prop:bigg} with $\beta=0$ and $t=\tau_n$ gives:
$$
\eee^{-m(0)\tau_n}S_n-\mathcal{W}_{\infty}(0)=o(r^{\tau_n})\quad\text{a.s.}
$$
From~\eqref{eq:tau_n_first_order} we deduce $r^{\tau_n}=o(n^{-\varepsilon_1})$ a.s., as $n\to\infty$, for every $\varepsilon_1 < |\log r|/m(0)$. The variance of $S_1$ is finite by Assumption B3. By the law of iterated logarithm, for every $\delta>0$, as $n\to\infty$,
$$
S_n = m(0)n + o(n^{1/2+\delta})\quad\text{a.s.}
$$
Combining the estimates, we see that~\eqref{eq:tau_n_precise} holds for $\varepsilon<(|\log r|/m(0)\wedge 1/2)$.
\end{proof}

Recall that $H_n(\beta)=\eee^{m(\beta)\tau_n-\varphi(\beta)\log n}$. Lemma~\ref{lem:split_time} immediately yields

\begin{lemma} \label{lem:h}
For $\beta \in \fD$, let $H_\infty(\beta) =  (\mathcal W_{\infty}(0))^{-\varphi(\beta)} m(0)^{\varphi(\beta)}$. For any compact set $K \subset \fD$, there exists
$\eps = \eps(K) > 0$ such that
$$
n^{\varepsilon} \sup_{\beta \in K} | H_n(\beta) - H_\infty(\beta) | \toas 0.
$$
\end{lemma}


\begin{proof}[Proof of Theorem~\ref{theo:W_n_converges_one_split_BRW}]
Recall from~\eqref{connection} that $W_{n}(\beta) = \mathcal W_{\tau_n} (\beta) H_n(\beta)$. Define
$$
W_\infty(\beta) = \mathcal W_\infty(\beta) H_\infty(\beta) = \mathcal W_\infty(\beta) (\mathcal W_{\infty}(0))^{-\varphi(\beta)} m(0)^{\varphi(\beta)}.
$$
By Proposition~\ref{prop:bigg}, Lemmas~\ref{lem:split_time} and~\ref{lem:h} and the triangle inequality, $W_n(\beta)$ converges to $W_\infty(\beta)$ locally uniformly on $\fD$, with probability $1$ and speed $(\log n)^{-r}$. Since $H_\infty(\beta)>0$ for real $\beta$ and,  by Proposition~\ref{prop:bigg_no_zeros}, the function $\mathcal W_\infty$ has no zeros on the interval $(\beta_-,\beta_+)$ (with probability $1$), the same is true for the function $W_\infty$.
\end{proof}

\subsection{Proof of Theorem~\ref{theo:edgeworth_one_split_BRW}}\label{subsec:proof_biggins}
Consider a one-split BRW satisfying Assumptions B1--B5. We are going to verify Assumptions A1--A4 of Theorem~\ref{theo:edgeworth_general} for the sequence of its profiles $\LLL_1,\LLL_2,\ldots$ and $w_n=\log n$.
Assumption A1 is fulfilled because the number of particles in the one-split BRW is finite at any time and hence, the function $\LLL_n$ has a.s.\ finite support. Assumptions A2 and A3 were verified in  Theorem~\ref{theo:W_n_converges_one_split_BRW}.

The next proposition verifies an analogue of Assumption A4 for the continuous-time BRW. The result is essentially shown in~\cite{biggins_uniform} without rate of convergence and only in the non-lattice case. For the sake of completeness, we include the proof here.

\begin{proposition} \label{prop_cond_4_cont}
For any compact set $K \subset (\beta_-, \beta_+)$ and $0 < a < \pi$, under Assumptions B1--B5, there exists $0 < r = r(K,a) < 1$ such that
\begin{align} \label{cond_4_cont}
r^{-t} \sup_{\beta \in K} \sup_{a \leq \eta \leq \pi} \eee^{-m(\beta)t} \left | \sum_{k = 1}^{N_t} \eee^{(\beta + i \eta) z_{k,t}} \right | \toast 0.
\end{align}
\end{proposition}
\begin{proof}

Set $\psi(\beta) = |m^*(\beta)| / m^*(\Re\beta)$ and note that, for $\beta\in\C, \Re\beta \in \fI$,
$$
(m^*(\Re\beta))^{-t} \left | \sum_{k = 1}^{N_t} \eee^{\beta z_{k,t}} \right | = (\psi(\beta))^t |\mathcal W_t(\beta)|.
$$
Therefore, \eqref{cond_4_cont} is equivalent to
\begin{equation} \label{cond_4_cont1}
r^{-t} \sup_{\beta \in G} 	\psi^t(\beta) |\mathcal W_t(\beta)|\toast 0,
\end{equation}
where $G := \{ \beta \in \C: \Re\beta \in K, \; \Im\beta \in [a,\pi]\}$.  By compactness, it is enough to check that, for any $\beta_0 \in G$, there exists $\eps > 0$ such that
\begin{equation} \label{cond_4_cont2}
r^{-t} \sup_{\beta \in \overline{\mathbb{D}}_{\varepsilon}(\beta_0)} \psi^{t}(\beta) |\mathcal W_t(\beta)|\toast 0.
\end{equation}
By the Borel--Cantelli lemma, \eqref{cond_4_cont2} follows from summability of the sequence
\begin{align}\label{summ}
r^{-n} \E  \left[ \sup_{t \in [n,n+1]} \sup_{\beta \in \overline{\mathbb{D}}_{\varepsilon}(\beta_0)} \psi^n(\beta) |\mathcal W_t(\beta)|\right],\quad n\in\N.
\end{align}
By Cauchy's integral formula, c.f. \cite[Lemma 3]{biggins_uniform},
$$
\sup_{\beta \in \overline{\mathbb{D}}_{\varepsilon}(\beta_0)}|\mathcal W_t(\beta)|\leq \frac{1}{\pi}\int_0^{2\pi}|\mathcal{W}_t(\beta_0+2\eps\eee^{i\phi})|{\rm d}\phi,
$$
whence, for $\gamma>1$,
\begin{eqnarray*}
\E  \left[ \sup_{t \in [n,n+1]} \sup_{\beta \in \overline{\mathbb{D}}_{\varepsilon}(\beta_0)} |\mathcal W_t(\beta)|\right]
&\leq&\frac{1}{\pi}\int_0^{2\pi}\E\left[\sup_{t\in[n,n+1]}|\mathcal{W}_t(\beta_0+2\eps\eee^{i\phi})|\right]{\rm d}\phi\\
&\leq&\frac{1}{\pi}\int_0^{2\pi}\left(\E\left[\sup_{t\in[n,n+1]}|\mathcal{W}_t(\beta_0+2\eps\eee^{i\phi})|^{\gamma}\right]\right)^{1/\gamma}{\rm d}\phi\\
&\leq&\frac{\gamma}{\pi(\gamma-1)}\int_0^{2\pi}\left(\E|\mathcal{W}_{n+1}(\beta_0+2\eps\eee^{i\phi})|^{\gamma}\right)^{1/\gamma}{\rm d}\phi,
\end{eqnarray*}
having utilized Doob's inequality in the last passage. Choose $\eps >0$ small enough such that there exists $\gamma > 1$ with $\overline{\mathbb{D}}_{2\varepsilon}(\Re \beta_0) \subseteq \Omega_\gamma^1 \cap \Omega_\gamma^2$. We fix
this $\gamma$ in the remainder of the proof. Further, let
$$
\kappa(\beta):=\frac{m^{\ast}(\gamma\Re \beta)}{|m^{\ast}(\beta)|^{\gamma}}.
$$
From our choice of $\gamma$, it follows that $\E|\mathcal{W}_1(\beta)|^{\gamma}<\infty$ for all $\beta\in\overline{\mathbb{D}}_{2\varepsilon}(\beta_0)$. Using \cite[Lemma~2(ii)]{biggins_uniform}, we obtain
\begin{eqnarray*}
\E  \left[ \sup_{t \in [n,n+1]} \sup_{\beta \in \overline{\mathbb{D}}_{\varepsilon}(\beta_0)} |\mathcal W_t(\beta)|\right]&\leq& C\int_0^{2\pi}\left(\sum_{j=0}^{n}\kappa^j(\beta_0+2\eps\eee^{i\phi})\right)^{1/\gamma}{\rm d}\phi,\\
&\leq& C(n+1)^{1/\gamma}\int_0^{2\pi}\left(\kappa(\beta_0+2\eps\eee^{i\phi})\vee 1\right)^{n/\gamma}{\rm d}\phi
\end{eqnarray*}
for some $C>0$. Therefore,
\begin{eqnarray*}
&&\hspace{-2cm}\E  \left[ \sup_{t \in [n,n+1]} \sup_{\beta \in \overline{\mathbb{D}}_{\varepsilon}(\beta_0)} \psi^n(\beta)|\mathcal W_t(\beta)|\right]\\
&\leq& C(n+1)^{1/\gamma}\left(\sup_{\beta \in \overline{\mathbb{D}}_{\varepsilon}(\beta_0)} \psi^n(\beta)\right)\int_0^{2\pi}\left(\kappa(\beta_0+2\eps\eee^{i\phi})\vee 1\right)^{n/\gamma}{\rm d}\phi\\
&\leq& 2C\pi (n+1)^{1/\gamma}\left(\left(\sup_{\beta \in \overline{\mathbb{D}}_{2\varepsilon}(\beta_0)} \psi^\gamma(\beta)\right)\left(\sup_{\beta \in \overline{\mathbb{D}}_{2\varepsilon}(\beta_0)}\left(\kappa(\beta)\vee 1\right)\right)\right)^{n/\gamma}.
\end{eqnarray*}
Note that both $\beta\mapsto \psi^\gamma(\beta)$ and $\beta\mapsto \kappa(\beta)\vee 1$ are continuous in some neighborhood of $\beta_0$. Hence, for $\delta>0$, upon possibly decreasing $\varepsilon>0$, we obtain
$$
\sup_{\beta \in \overline{\mathbb{D}}_{2\varepsilon}(\beta_0)} \psi^\gamma(\beta)\leq (1+\delta)\psi^\gamma(\beta_0)\quad\quad\text{and}\quad\quad\sup_{\beta \in \overline{\mathbb{D}}_{2\varepsilon}(\beta_0)}(\kappa(\beta)\vee 1) \leq (1+\delta)(\kappa(\beta_0)\vee 1).
$$
\eqref{summ} now follows from these bounds by a suitable choice of $\delta$ upon verifying that
$$
\psi^{\gamma}(\beta_0)(\kappa(\beta_0)\vee 1)=\frac{m^{\ast}(\gamma\Re \beta_0)\vee |m^{\ast}(\beta_0)|^{\gamma}}{(m^{\ast}(\Re \beta_0))^{\gamma}}<1.
$$
First, on the one hand, since $\Re \beta_0\in \Omega_\gamma^1 \cap \Omega_\gamma^2$ we have
$$
\frac{m^{\ast}(\gamma\Re \beta_0)}{(m^{\ast}(\Re \beta_0))^{\gamma}}<1.
$$
On the other hand, since $a\leq \Im\beta_0\leq \pi$, we have
$$
\frac{|m^{\ast}(\beta_0)|}{m^{\ast}(\Re \beta_0)}=\exp\left\{\sum_{k\in\Z}\nu_k\eee^{k\Re \beta_0}(\cos(k\Im\beta_0)-1)\right\}<1,
$$
having utilized Assumptions B1 and B4. The proof of Proposition~\ref{prop_cond_4_cont} is complete.
\end{proof}

Now we can pass back to the one-split BRW. By combining Lemma~\ref{lem:split_time} and Proposition~\ref{prop_cond_4_cont}, one deduces that, for any compact set $K \subset (\beta_-, \beta_+)$ and $0 < a < \pi$,  there exists $\varepsilon = \varepsilon(K,a) > 0$ such that
\begin{align} \label{cond_4_discrete}
n^\varepsilon \sup_{\beta \in K} \sup_{a < \eta \leq \pi} n^{-\varphi(\beta)} \left | \sum_{k = 1}^{S_n} \eee^{(\beta + i \eta) x_{k,n}} \right | \toas 0.
\end{align}
Assumption A4 follows readily. Theorem~\ref{theo:edgeworth_one_split_BRW} now follows from Theorem~\ref{theo:edgeworth_general}.

\subsection{Proof of Theorem~\ref{theo:edgeworth_one_split_BRW_mean}}\label{subsec:proof_edgeworth_mean}
We apply Theorem~\ref{theo:edgeworth_general} to the (deterministic) profile function $\tilde {\mathbb L}_n(k) := \E [\mathbb L_n(k)]$. Obviously, the corresponding moment generating function $\tilde W_n(\beta)$ is simply $\E{W_n(\beta)}$. Its limit $\tilde W_\infty(\beta)$ was calculated in~\eqref{mean_wlimit_exact}: for any $\beta \in \C$ with $\Re \beta \in (\beta_-, \beta_+)$,
\begin{equation}\label{eq:lim_E_W_n}
\tilde W_\infty(\beta) :=  \lim_{n\to\infty} \E W_n(\beta) =  \frac{\Gamma\left(\frac{1}{m(0)}\right)}{\Gamma\left(\frac{m(\beta) + 1}{m(0)}\right)}.
\end{equation}
Using the explicit formula~\eqref{mean_w_exact}, a direct application of Stirling's formula shows that Assumption A3 is satisfied.
Similarly, Assumption A4 is easily verified using~\eqref{mean_w_exact} and Assumption B4.
To conclude the proof, it remains to show that
$\tilde W_\infty(\beta) = \E{W_\infty(\beta)}$ for real $\beta \in (\beta_-, \beta_+)$ which is true if (and only if) the sequence $(W_n(\beta))_{n\in\N}$ is uniformly integrable.
\begin{proposition}
Consider a one-split BRW with deterministic number of descendants and satisfying Assumptions B1--B3 and B5. Then, for every $\beta\in (\beta_-,\beta_+)$, the sequence $(W_n(\beta))_{n\in\N_0}$ is bounded in $L_\gamma$, for some $\gamma=\gamma(\beta)>1$.
\end{proposition}
\begin{proof}
For $\beta = 0$ and $\gamma = 2$, the relevant argument is given in the proof of Proposition~6 in~\cite{sulzbach_mart_tail_sums}. Fix $\beta \in (\beta_-, \beta_+)$ and $\gamma \in (1,2]$ such that $\beta \in \Omega^1_\gamma \cap \Omega^2_\gamma$. Note that $W_n(\beta)$ and $H_n(\beta)$ are independent and $\mathcal W_{\tau_n}(\beta) = H_n(\beta) W_n(\beta)$. By the optional stopping theorem, $(\mathcal W_{\tau_n}(\beta))_{n\in\N}$ is a martingale with mean 1 and bounded in $L_\gamma$. By indepedence, $\E \mathcal W_{\tau_n}(\beta) = \E H_n(\beta) \E W_n(\beta)$. Since $\E{W_n(\beta)}$ converges to a non-zero limit, see~\eqref{eq:lim_E_W_n}, it follows that $\E{H_n(\beta)}$ is bounded away from zero. Thus, by independence and  Jensen's inequality,
$$
\E \mathcal W_{\tau_n}^\gamma (\beta) = \E H_n^\gamma(\beta) \cdot \E W_n^\gamma(\beta) \geq \E{W_n^\gamma(\beta)} (\E{H_n(\beta)})^\gamma.
$$
It follows that  $\sup_{n \geq 0} \E W_n^\gamma (\beta) < \infty$ which completes the proof.
\end{proof}

\subsection*{Acknowledgement} Zakhar Kabluchko is grateful to Rudolf Gr\"ubel for useful discussions. The work of Alexander Marynych was supported by a Humboldt Research Fellowship of the Alexander von Humboldt Foundation. The work of Henning Sulzbach was supported by a Feodor Lynen Research Fellowship of the Alexander von Humboldt Foundation.

\bibliographystyle{plainnat}
\bibliography{KabMarSulz2017_bib}

\begin{thebibliography}{41}
\providecommand{\natexlab}[1]{#1}
\providecommand{\url}[1]{\texttt{#1}}
\expandafter\ifx\csname urlstyle\endcsname\relax
  \providecommand{\doi}[1]{doi: #1}\else
  \providecommand{\doi}{doi: \begingroup \urlstyle{rm}\Url}\fi

\bibitem[Athreya and Karlin(1968)]{atka}
K.~B. Athreya and S.~Karlin.
\newblock Embedding of urn schemes into continuous time {M}arkov branching
  processes and related limit theorems.
\newblock \emph{Ann. Math. Statist.}, 39:\penalty0 1801--1817, 1968.

\bibitem[Bergeron et~al.(1992)Bergeron, Flajolet, and Salvy]{Bergeron1992}
F.~Bergeron, P.~Flajolet, and B.~Salvy.
\newblock Varieties of increasing trees.
\newblock In \emph{C{AAP} '92 ({R}ennes, 1992)}, volume 581 of \emph{Lecture
  Notes in Comput. Sci.}, pages 24--48. Springer, Berlin, 1992.

\bibitem[Biggins(1992)]{biggins_uniform}
J.~D. Biggins.
\newblock Uniform convergence of martingales in the branching random walk.
\newblock \emph{Ann. Probab.}, 20\penalty0 (1):\penalty0 137--151, 1992.

\bibitem[Biggins and Grey(1979)]{BigGre79}
J.~D. Biggins and D.~R. Grey.
\newblock Continuity of limit random variables in the branching random walk.
\newblock \emph{J. Appl. Probab.}, 16\penalty0 (4):\penalty0 740--749, 1979.

\bibitem[Biggins and Grey(1997)]{biggins_grey}
J.~D. Biggins and D.~R. Grey.
\newblock A note on the growth of random trees.
\newblock \emph{Statist. Probab. Lett.}, 32\penalty0 (4):\penalty0 339--342,
  1997.

\bibitem[Chauvin and Rouault(2004)]{chauvin_etal_yule}
B.~Chauvin and A.~Rouault.
\newblock {Connecting {Y}ule Process, {B}isection and {B}inary {S}earch {T}ree
  via {M}artingales}.
\newblock \emph{{J. Iran. Statist. Soc.}}, 3\penalty0 (2):\penalty0 89--116,
  2004.

\bibitem[Chauvin et~al.(2001)Chauvin, Drmota, and
  Jabbour-Hattab]{chauvin_drmota_jabbour}
B.~Chauvin, M.~Drmota, and J.~Jabbour-Hattab.
\newblock {The profile of binary search trees.}
\newblock \emph{Ann. Appl. Probab.}, 11\penalty0 (4):\penalty0 1042--1062,
  2001.

\bibitem[Chauvin et~al.(2003)Chauvin, Klein, Marckert, and
  Rouault]{chauvin_etal_tiltings}
B.~Chauvin, T.~Klein, J.-F. Marckert, and A.~Rouault.
\newblock {Martingales, embedding and tilting of binary trees}.
\newblock \emph{Preprint}, 2003.

\bibitem[Chauvin et~al.(2005)Chauvin, Klein, Marckert, and
  Rouault]{chauvin_etal}
B.~Chauvin, T.~Klein, J.-F. Marckert, and A.~Rouault.
\newblock {Martingales and profile of binary search trees.}
\newblock \emph{{Elect. J. Probab.}}, 10:\penalty0 420--435, 2005.

\bibitem[Chen(2001)]{chen}
X.~Chen.
\newblock {Exact convergence rates for the distribution of particles in
  branching random walks.}
\newblock \emph{{Ann. Appl. Probab.}}, 11\penalty0 (4):\penalty0 1242--1262,
  2001.

\bibitem[Delbaen et~al.(2015)Delbaen, Kowalski, and
  Nikeghbali]{delbaen_kowalski_nikeghbali}
F.~Delbaen, E.~Kowalski, and A.~Nikeghbali.
\newblock Mod-{$\varphi$} convergence.
\newblock \emph{Int. Math. Res. Not. IMRN}, \penalty0 (11):\penalty0
  3445--3485, 2015.

\bibitem[Devroye(1986)]{devroyeheight}
L.~Devroye.
\newblock A note on the height of binary search trees.
\newblock \emph{J. Assoc. Comput. Mach.}, 33\penalty0 (3):\penalty0 489--498,
  1986.

\bibitem[{Devroye}(1987)]{devroye_branching}
L.~{Devroye}.
\newblock {Branching processes in the analysis of the heights of trees.}
\newblock \emph{{Acta Inf.}}, 24:\penalty0 277--298, 1987.

\bibitem[Devroye and Hwang(2006)]{devroye_hwang}
L.~Devroye and H.-K. Hwang.
\newblock {Width and mode of the profile for some random trees of logarithmic
  height.}
\newblock \emph{{Ann. Appl. Probab.}}, 16\penalty0 (2):\penalty0 886--918,
  2006.

\bibitem[Drmota(2009)]{drmota_book}
M.~Drmota.
\newblock \emph{Random trees: An interplay between combinatorics and
  probability}.
\newblock Springer--Verlag Wien, 2009.

\bibitem[Drmota and Hwang(2005)]{drmota_hwang}
M.~Drmota and H.-K. Hwang.
\newblock Profiles of random trees: correlation and width of random recursive
  trees and binary search trees.
\newblock \emph{Adv. Appl. Probab.}, 37\penalty0 (2):\penalty0 321--341, 2005.

\bibitem[Drmota et~al.(2008)Drmota, Janson, and
  Neininger]{drmota_janson_neininger}
M.~Drmota, S.~Janson, and R.~Neininger.
\newblock A functional limit theorem for the profile of search trees.
\newblock \emph{Ann. Appl. Probab.}, 18\penalty0 (1):\penalty0 288--333, 2008.

\bibitem[Erd{\"o}s(1953)]{erdoes_on_hammersley}
P.~Erd{\"o}s.
\newblock On a conjecture of {H}ammersley.
\newblock \emph{J. London Math. Soc.}, 28:\penalty0 232--236, 1953.

\bibitem[F{\'e}ray et~al.(2016)F{\'e}ray, M{\'e}liot, and
  Nikeghbali]{feray_meliot_nikeghbali}
V.~F{\'e}ray, P.-L. M{\'e}liot, and A.~Nikeghbali.
\newblock \emph{Mod--{$\phi$} {C}onvergence: {N}ormality {Z}ones and {P}recise
  {D}eviations}.
\newblock SpringerBriefs in Probability and Mathematical Statistics. Springer
  International Publishing, 2016.

\bibitem[Fuchs et~al.(2006)Fuchs, Hwang, and Neininger]{fuchs_hwang_neininger}
M.~Fuchs, H.-K. Hwang, and R.~Neininger.
\newblock Profiles of random trees: limit theorems for random recursive trees
  and binary search trees.
\newblock \emph{Algorithmica}, 46\penalty0 (3-4):\penalty0 367--407, 2006.

\bibitem[Gr{\"u}bel and Kabluchko(2015)]{gruebel_kabluchko_BRW}
R.~Gr{\"u}bel and Z~Kabluchko.
\newblock Edgeworth expansions for profiles of lattice branching random walks.
\newblock 2015.
\newblock Ann. Inst. H. Poincar{\'e}, to appear. Preprint at
  \url{http://arxiv.org/abs/1503.04616}.

\bibitem[Gr\"ubel and Kabluchko(2016)]{gruebel_kabluchko}
R.~Gr\"ubel and Z.~Kabluchko.
\newblock A functional central limit theorem for branching random walks, almost
  sure weak convergence and applications to random trees.
\newblock \emph{Ann. Appl. Probab.}, 26\penalty0 (6):\penalty0 3659--3698,
  2016.

\bibitem[Hwang(2007)]{hwang_PORT}
H.-K. Hwang.
\newblock Profiles of random trees: plane-oriented recursive trees.
\newblock \emph{Random Structures Algorithms}, 30\penalty0 (3):\penalty0
  380--413, 2007.

\bibitem[Jabbour-Hattab(2001)]{jabbour}
J.~Jabbour-Hattab.
\newblock Martingales and large deviations for binary search trees.
\newblock \emph{Random Structures Algorithms}, 19\penalty0 (2):\penalty0
  112--127, 2001.

\bibitem[Jacod et~al.(2011)Jacod, Kowalski, and
  Nikeghbali]{jacod_kowalski_nikeghbali_mod_Gauss}
J.~Jacod, E.~Kowalski, and A.~Nikeghbali.
\newblock Mod-{G}aussian convergence: new limit theorems in probability and
  number theory.
\newblock \emph{Forum Math.}, 23\penalty0 (4):\penalty0 835--873, 2011.

\bibitem[Kabluchko(2012)]{kabluchko_distr_of_levels}
Z.~Kabluchko.
\newblock {Distribution of levels in high-dimensional random landscapes.}
\newblock \emph{Ann. Appl. Probab.}, 22\penalty0 (1):\penalty0 337--362, 2012.

\bibitem[Kabluchko et~al.(2016)Kabluchko, Marynych, and
  Sulzbach]{kabluchko_marynych_sulzbach_on_stirling}
Z.~Kabluchko, A.~Marynych, and H.~Sulzbach.
\newblock Mode and {E}dgeworth expansion for the {E}wens distribution and the
  {S}tirling numbers.
\newblock \emph{J. Integer Seq.}, 19, 2016.
\newblock Article 16.8.8.

\bibitem[Katona(2005)]{katona}
Z.~Katona.
\newblock Width of a scale-free tree.
\newblock \emph{J. Appl. Probab.}, 42\penalty0 (3):\penalty0 839--850, 2005.

\bibitem[Kowalski and Nikeghbali(2010)]{kowalski_nikeghbali_mod_Poi}
E.~Kowalski and A.~Nikeghbali.
\newblock Mod-{P}oisson convergence in probability and number theory.
\newblock \emph{Int. Math. Res. Not. IMRN}, \penalty0 (18):\penalty0
  3549--3587, 2010.

\bibitem[Kowalski and Nikeghbali(2012)]{kowalski_nikeghbali_zeta}
E.~Kowalski and A.~Nikeghbali.
\newblock Mod-{G}aussian convergence and the value distribution of
  {$\zeta(\frac12+it)$} and related quantities.
\newblock \emph{J. Lond. Math. Soc. (2)}, 86\penalty0 (1):\penalty0 291--319,
  2012.

\bibitem[Kuba and Panholzer(2007)]{kubapanhol}
M.~Kuba and A.~Panholzer.
\newblock The left-right-imbalance of binary search trees.
\newblock \emph{Theoret. Comput. Sci.}, 370\penalty0 (1-3):\penalty0 265--278,
  2007.

\bibitem[Mahmoud(1992)]{mahmoud_book}
H.~M. Mahmoud.
\newblock \emph{Evolution of random search trees}.
\newblock Wiley-Interscience Series in Discrete Mathematics and Optimization.
  John Wiley \& Sons, Inc., New York, 1992.

\bibitem[M\'eliot and Nikeghbali(2015)]{meliot_nikeghbali_statmech}
P.-L. M\'eliot and A.~Nikeghbali.
\newblock Mod-{G}aussian convergence and its applications for models of
  statistical mechanics.
\newblock In \emph{In memoriam {M}arc {Y}or---{S}\'eminaire de {P}robabilit\'es
  {XLVII}}, volume 2137 of \emph{Lecture Notes in Math.}, pages 369--425.
  Springer, Cham, 2015.

\bibitem[Petrov(1975)]{petrov_book}
V.~V. Petrov.
\newblock \emph{Sums of independent random variables}.
\newblock Springer--Verlag, New York-Heidelberg, 1975.
\newblock Ergebnisse der Mathematik und ihrer Grenzgebiete, Band 82.

\bibitem[Pittel(1984)]{pittel2}
B.~Pittel.
\newblock On growing random binary trees.
\newblock \emph{J. Math. Anal. Appl.}, 103\penalty0 (2):\penalty0 461--480,
  1984.

\bibitem[R{\'e}gnier(1989)]{regnier}
M.~R{\'e}gnier.
\newblock A limiting distribution for quicksort.
\newblock \emph{RAIRO Inform. Th\'eor. Appl.}, 23\penalty0 (3):\penalty0
  335--343, 1989.

\bibitem[R{\"o}sler(1991)]{roesler}
U.~R{\"o}sler.
\newblock A limit theorem for ``{Q}uicksort''.
\newblock \emph{RAIRO Inform. Th\'eor. Appl.}, 25\penalty0 (1):\penalty0
  85--100, 1991.

\bibitem[Schopp(2010)]{schopp}
E.-M. Schopp.
\newblock A functional limit theorem for the profile of {$b$}-ary trees.
\newblock \emph{Ann. Appl. Probab.}, 20\penalty0 (3):\penalty0 907--950, 2010.

\bibitem[Sulzbach(2008)]{sulzbach}
H.~Sulzbach.
\newblock A functional limit law for the profile of plane-oriented recursive
  trees.
\newblock In \emph{Fifth {C}olloquium on {M}athematics and {C}omputer
  {S}cience}, Discrete Math. Theor. Comput. Sci. Proc., AI, pages 339--350.
  Assoc. Discrete Math. Theor. Comput. Sci., Nancy, 2008.

\bibitem[Sulzbach(2017)]{sulzbach_mart_tail_sums}
H.~Sulzbach.
\newblock {On martingale tail sums for the path length in random trees}.
\newblock \emph{Random Structures Algorithms}, 50\penalty0 (3):\penalty0
  493--508, 2017.

\bibitem[Uchiyama(1982)]{uchiyama}
K.~Uchiyama.
\newblock Spatial growth of a branching process of particles living in {${\bf
  R}^{d}$}.
\newblock \emph{Ann. Probab.}, 10\penalty0 (4):\penalty0 896--918, 1982.

\end{thebibliography}

\end{document}